\documentclass{article}

\usepackage[utf8]{inputenc}
\usepackage[T1]{fontenc}
\usepackage{amsmath,amssymb,amsthm,bbm,float,graphicx,geometry,lmodern,mathtools,parskip,setspace,subcaption}
\usepackage[colorlinks, linkcolor = blue!80!black, citecolor = blue!80!black,breaklinks, pdfauthor={Benedikt Jahnel, Lukas Luechtrath, Marcel Ortgiese}]{hyperref}
\usepackage{mathrsfs}
\usepackage{enumerate}
\usepackage{nicefrac}
\usepackage{todonotes}
\usepackage{comment}
\usepackage{titling}
\usepackage{fullpage}
\usepackage{comment}
\usepackage{orcidlink}

\usepackage{tikz}
\usetikzlibrary{backgrounds}
\usetikzlibrary{patterns}
\usetikzlibrary{positioning, shapes.geometric}

\usepackage{caption}
\usepackage{graphicx}
\graphicspath{{Images/}}
\allowdisplaybreaks

\newcommand{\N}{\mathbb{N}}

\newtheorem{theorem}{Theorem}[section]
\newtheorem{lemma}[theorem]{Lemma}

\newtheorem{prop}[theorem]{Proposition}
\newtheorem{conjecture}[theorem]{Conjecture}

\theoremstyle{definition}

\newtheorem{remark}[theorem]{Remark}

\usepackage{nameref}

\makeatletter
\let\orgdescriptionlabel\descriptionlabel
\renewcommand*{\descriptionlabel}[1]{%
  \let\orglabel\label
  \let\label\@gobble
  \phantomsection
  \edef\@currentlabel{#1}%
  \let\label\orglabel
  \orgdescriptionlabel{#1}%
}

\renewcommand{\P}{\mathbb{P}}
\newcommand{\x}{{\boldsymbol{x}}}
\newcommand{\y}{{\boldsymbol{y}}}
\newcommand{\1}{\mathbbm{1}}
\newcommand{\G}{\mathscr{G}}
\newcommand{\C}{\mathscr{C}}
\newcommand{\X}{\boldsymbol{X}}
\newcommand{\E}{\mathbb{E}}
\newcommand{\0}{\mathbf{o}}
\renewcommand{\o}{\0}
\newcommand{\R}{\mathbb{R}}
\renewcommand{\d}{\mathrm{d}}
\newcommand{\z}{{\boldsymbol{z}}}

\newcommand{\U}{\mathscr{U}}
\newcommand{\cX}{\mathscr{X}}
\newcommand{\cG}{\mathscr{G}}
\newcommand{\M}{\mathscr{D}^d}
\newcommand{\scN}{\mathscr{N}}
\newcommand{\cF}{\mathcal{F}}
\newcommand{\cE}{\mathcal{E}}
\newcommand{\calG}{\mathcal{G}}
\newcommand{\cH}{\mathcal{H}}
\newcommand{\cI}{\mathcal{I}}

\pretitle{\centering\LARGE\scshape}
 \posttitle{\vskip 0.75cm}

 \predate{\vskip 0.75 cm \centering\large}
 \postdate{\par}

\usepackage[style = numeric, sorting=nyt, url = false, abbreviate=false, maxbibnames=9, sortcites=true, doi = true, backend = biber, giveninits = true, isbn=false]{biblatex}
\renewbibmacro{in:}{\ifentrytype{article}{}{\printtext{\bibstring{in}\intitlepunct}}}
\DeclareFieldFormat{journaltitle}{\mkbibemph{#1\isdot}}
\bibliography{bib}

\title{Cluster sizes in subcritical soft Boolean models}

\thanksmarkseries{arabic}

\author{
Benedikt Jahnel \orcidlink{0000-0002-4212-0065}\thanks{Technische Universit\"at Braunschweig, Universit\"atsplatz 2, 38106 Braunschweig, Germany}\thanksgap{0.4ex} \thanks{Weierstrass Institute for Applied Analysis and Stochastics, Anton-Wilhelm-Amo-Str.~39, 10117 Berlin, Germany} \\ benedikt.jahnel@tu-braunschweig.de
\and
Lukas L\"{u}chtrath \orcidlink{0000-0003-4969-806X}\thanksmark{2}\\lukas.luechtrath@wias-berlin.de \\
\and
Marcel Ortgiese \orcidlink{0000-0001-7803-8584}\thanks{University of Bath, Claverton Down, Bath BA2 7AY, United Kingdom} \\ m.ortgiese@bath.ac.uk 
}

\date{June 15, 2026}

\begin{document}

\maketitle
\begin{spacing}{0.9}
\begin{abstract} 
\noindent 
We consider the soft Boolean model, a model that interpolates between the Boolean model and long-range percolation, where vertices are given via a stationary Poisson point process. Each vertex carries an independent Pareto-distributed radius and each pair of vertices is assigned another independent Pareto weight with a potentially different tail exponent. Two vertices are now connected if they are within distance of the larger radius multiplied by the edge weight. We determine the tail behaviour of the Euclidean diameter and the number of points of a typical maximally connected component in a subcritical percolation phase. 
For this, we present a sharp criterion in terms of the tail exponents of the edge-weight and radius distributions that distinguish a regime where the tail behaviour is controlled only by the edge exponent from a regime in which both exponents are relevant. 
Our proofs rely on fine path-counting arguments identifying the precise order of decay of the probability that far-away vertices are connected.

\smallskip
\noindent\footnotesize{{\textbf{AMS-MSC 2020}: 60K35, 05C80}

}
\end{abstract}
\end{spacing}


\section{Introduction} \label{sec:Intro}
Sparse random-graph models that combine \emph{heavy-tailed degree distributions} and \emph{long-range effects} have attracted a lot of attention in recent years due to their ability to describe many real-world networks.
For these random graphs vertices are embedded into the $d$-dimensional  Euclidean space via a homogeneous {\em Poisson point process} with unit intensity and additionally each vertex is assigned an i.i.d.~weight or \emph{interaction radius}. Given the vertices and their radii, edges are then drawn independently such that edges to vertices with a large radius or to vertices nearby are more probable. The overall number of edges is controlled by an intensity parameter \(\beta>0\)  that influences the connectivity in such a way that increasing \(\beta\) leads to more edges on average. Although short edges are more likely, the connection mechanism still allows edges between arbitrarily far-apart vertices to occur occasionally. An important question then is whether there exists a non-trivial critical intensity \(\beta_c\) such that the graph contains an infinite connected component with positive probability if \(\beta>\beta_c\) but does not if \(\beta<\beta_c\). The ergodicity entailed in the model ensures that the existence of an infinite connected component is a \(0\)-\(1\)-event and any existing infinite connecting component is unique~\cite{BurtonKeane89,JacobMoerters2017,chebuninLast2025}. {If they exist,} we call \((0,\beta_c)\) the \emph{subcritical phase} and \((\beta_c,\infty)\) the \emph{supercritical phase}. {For many models the existence of these phases has been shown, and one aims for more detailed results about the behaviour of the graph within each of those phases as a next step to gain a deeper structural understanding of the model.}
 
This article is devoted to the analysis of the tail distribution of the Euclidean diameter as well as the number of points of the typical connected component in the \emph{soft Boolean model} in the subcritical regime. This random graph was introduced in~\cite{GGM22} as a model that interpolates between the \emph{(Poisson) Boolean model} and the \emph{long-range random connection model}, arguably the two most studied models in the continuum-percolation literature. 
{The regimes, where there exists a subcritical phase in this model, can be described in terms of the radius distribution and the long-range effects
and are precisely identified in~\cite{GLM2021}, and we study the model in such a regime at low intensity}. 
For the diameter of the typical cluster in the soft Boolean model, we identify the decay exponent for the whole parameter regime where a subcritical phase exists and find a sharp criterion when the long-range effects, or the radii, or the combination of both determine the decay. For the number of points in the typical cluster we derive complementary results, at least in parts of the parameter regime, that show a strong qualitative difference to the behaviour of the diameter. 

The idea of constructing graphs on random point clouds in the continuum and then connecting the vertices based on their distances goes back to~\cite{Gilbert61}, where a model is proposed in which any two Poisson points are connected by an edge whenever their distance is smaller than a threshold \(\beta>0\), which can be seen as a continuum equivalent to nearest-neighbour percolation on the lattice. This model is commonly referred to as the \emph{Gilbert graph} or \emph{random geometric graph} \cite{Penrose2003}. The connection rule can be also seen as assigning a ball of radius \(\beta/2\) to each vertex and connecting any pair of vertices whose balls intersect. Later on, this idea was generalised to random radii in the physics literature~\cite{PikeSeager1974,KerteszVicsek1982} and rigorously introduced in~\cite{Hall85}. Nowadays, this model is commonly known under the name \emph{(Poisson) Boolean model}. It is a, by now, classical result that there always is a non-trivial phase transition in dimensions \(d\geq 2\), and no supercritical phase in dimension \(d=1\), in both cases, whenever the radius distribution has finite \(d\)-th moment~\cite{Hall85,MeesterRoy1994}. This especially includes heavy-tailed radius distributions which are of particular interest to us as they lead to heavy-tailed (or scale-free) degree distributions. In 2008, {the existence of a subcritical phase was shown in~\cite{Gouere08} based on the crossing of large annuli. This alternative description led to an alternative} critical intensity \(\widehat{\beta}_c\leq\beta_c\). {The advantage of that second, potentially smaller, subcritical phase is that the description via annulus crossings allows for renormalisation arguments and therefore for quantitative results. In particular, it was shown that} the Euclidean diameter of the typical connected component has finite \(s\)-th moment if and only if the radius distribution has finite \((d+s)\)-th moment whenever \(\beta<\widehat{\beta}_c\) and the same holds for the number of points in the typical cluster~\cite{JahnelAndrasCali2022, Gouere08}. For the planar Boolean model in \(d=2\), it was shown that \(\beta_c=\widehat{\beta}_c\) for all radius distributions with finite \(d\)-th moment in~\cite{AhlbergTassionTeixeira2018}. Under additional moment conditions, this was generalised to all dimensions \(d\geq 2\) via randomised algorithms and the OSSS inequality in~\cite{DCopinRaoufiTassion2020}. This result was extended to the class of Pareto distributed radii with some technical exceptions in~\cite{DembinTassion2022}. {Combined, these articles show that the subcritical phase in the Boolean model is appropriately explained by considering crossing events for large annuli.} 

Another generalisation of Gilbert's model is to replace the hard-threshold condition and instead connect any pair of vertices with a probability determined by a non-increasing function of the distance of the vertices. This was introduced in continuum percolation under the name \emph{(Poisson) random connection model} in~\cite{Penrose1991, Penrose1996} and further studied in~\cite{Meester1995,MeesterPenroseSarkar1997}. The generalised connection mechanism particularly includes connection functions that decay polynomially. Since we are particularly interested in these long-range effects coming from a connection function that decays polynomially in the distance, we will stick with the name of long-range percolation throughout the paper. The idea of connecting any pair of vertices with a probability given by a negative power of their spatial distance was first introduced in~\cite{Schulman1983} as a lattice model that exhibits a non-trivial percolation phase transition even in dimension \(d=1\). More precisely, any pair of lattice points is connected independently with a probability decaying polynomially in the distance of the vertices with exponent \(d\delta\), for some \(\delta>1\), and there is a non-trivial phase transition \(\beta_c\in(0,\infty)\) in dimension \(d=1\) if and only if \(\delta\leq 2\), see~\cite{NewmanSchulman1986,DCGT2019}. Moreover, for \(\delta=2\), the percolation function is discontinuous~\cite{AizenmanNewman1986}, which is rather atypical. This generalises to the continuum version. It was shown that there always exists a non-trivial phase transition in \(d\geq 2\) as long as the connection function is integrable~\cite{MeesterRoy1996}. Further, in the whole subcritical regime the expected number of points in a typical connected component is finite~\cite{Meester1995}, which is sometimes also referred to as a \emph{sharp phase transition}. {However, unlike in the Boolean model, the subcritical phase of this model cannot be appropriately described via the crossings of large annuli~\cite{jacobJahLu2024}. Indeed, for instance, if \(\delta<2\) and \(\beta<\beta_c\), then the model is subcritical but still the annulus between balls of radius \(r\) and \(2r\) is crossed with high probability, as all edges are independent and there are roughly of order \(r^{2d}\) pairs of vertices at distance \(r\), each sharing an edge with probability \(r^{-d\delta}\).} For a detailed overview of the random connection model, we refer the reader to~\cite{MeesterRoy1996}.

Besides the soft Boolean model, other models combining heavy-tailed degree distributions and long-range effects and their percolation behaviour have been studied recently in the literature. Examples include~\emph{scale-free percolation}~\cite{DeijfenHofstadHooghiemstra2013} and its continuum version~\cite{DeprezWuthrich2019}, also studied under the name \emph{geometric inhomogeneous random graphs}~\cite{BringmannKeuschLengler2017,BringmannKeuschLengler2019}. Here, the radii play the role of weights and
two vertices are connected with a probability that decays polynomially in their distance divided by the product of both weights. This model also appears as a generalised weak local limit of \emph{hyperbolic random graph models}~\cite{KrioukovEtAl2010}, cf.~\cite{KomjathyLodewijks2020}. Another model is the \emph{age-dependent random connection model}~\cite{GGLM2019} which appears as the weak local limit of a preferential-attachment type model, where the influence of a vertex is determined by its age. These models, together with the soft Boolean model, and the above mentioned classical continuum percolation models as well as the \emph{ultra-small scale-free geometric network}~\cite{Yukich2006} are contained in the class of \emph{weight-dependent random connection models}, introduced in~\cite{GHMM2022}. The existence of subcritical and supercritical percolation phases for these models are shown in~\cite{GLM2021,GraLuMo2022,jacobJahLu2024}. In~\cite{GGM22,luechtrath24}, the graph distances in supercritical phases are identified. In~\cite{Moench2023} it is shown that, for strong enough long-range effects, the percolation function on weight-dependent random connection models is continuous, a property that is believed to be generally true.  
Closely related to the weight-dependent random connection model, only using a different parametrisation, are the recently introduced \emph{kernel-based spatial random graphs}~\cite{JorritsmaMitscheKomjathy2023}. In that article, the authors show, under some additional assumptions, that the tail of the cluster-size distribution of a finite component in a supercritical regime decays stretched exponentially.

\medskip
Our article is organised as follows: In the next section, we give a description of the soft Boolean model and our main results. Our primary result, Theorem~\ref{thm:MetaMain}, describes the asymptotic behaviour of the Euclidean diameter.
Additionally, we derive results on the cardinality of the component of the origin, see Theorem~\ref{thm:NumberPoints}, in comparison. In Section~\ref{sec:Construction}, we give a formal construction of the model. We explain the proof strategy in Section~\ref{sec:Strategy}. In Section~\ref{sec:openProblems}, we further discuss our results, compare it with known results and elaborate on some open problems. We present the proofs of our results in Section~\ref{sec:proofs}. The main contribution here is the proof of the upper bound of the Euclidean diameter in Section~\ref{sec:MainThmLGupper}. Due to the number of long edges in our model in addition to the inhomogeneity coming from the  radii of the vertices, classical renormalisation arguments for the Boolean model or long-range percolation cannot be applied. Instead we rely on fine moment bounds applied to carefully chosen paths to derive our results.   

\section{Setting and main results} \label{sec:Result}
\subsection{Description of the model}

We now introduce the soft Boolean model~\cite{GGM22} that combines both heavy-tailed degree distributions and long-range effects. 
The vertex set is given by a homogeneous Poisson point process \(X\) in \(\R^d\) of intensity one. Note that fixing the intensity is no restriction for our results by rescaling. Next, we assign to each vertex \(x\in X\) an independent radius \(R_x\) distributed according to a Pareto distribution with tail exponent~\(d/\gamma\) for some \(\gamma\in(0,1)\), that is \(\P(R_x>r)=1\wedge r^{-d/\gamma}\). Further, given \(X\), we assign to each pair of distinct vertices \(\{x,y\}\subset X\) an independent edge weight \(W(x,y)\), which is also Pareto distributed with tail exponent \(d\delta\) for some \(\delta>1\). Then, given \(X\) and the collection of radii and edge weights, each pair of vertices \(x,y\in X\) is connected by an edge if and only if
\begin{equation}\label{eq:defSoftBool} 
    |x-y|\leq \beta^{1/d} W(x,y)\big(R_x\vee R_y\big),
\end{equation}
where \(\beta>0\) is the edge intensity, and \(|\cdot|\) denotes the Euclidean norm. In words, two vertices are connected if the vertex with the smaller radius, say \(x\), is contained in the {enlarged} ball of radius \(\beta^{1/d} W(x,y)R_y\) around the vertex with larger radius \(y\), see Figure~\ref{fig:SoftBool}. We denote the resulting graph by \(\G^\beta\). {The reason for taking the \(1/d\)-th power of \(\beta\) lies in the construction given in Section~\ref{sec:Construction} that is based on the \emph{volume} of the associated balls instead of their radii.
One of the advantages of this construction is that the volume 
translates directly to the degree distribution of the graph, which plays 
an important role, as we will see below.}

\begin{figure}
\begin{center}
\resizebox{1\textwidth}{!}{
\begin{tikzpicture}[scale=0.65, every node/.style={scale=1}, show background rectangle]
						\node (X) at (0,1.5) [circle, fill=blue, label = {$x$}] {};
    					\node (Y) at (3,3) [circle, fill=blue, label = {below:$y$}] {};
   						\node (Z) at (6.3, 1) [circle, fill=blue, label = {right: $z$}] {};
   						\draw[color = white] (X) circle (4);
   						\draw[color = white] (Z) circle (4.2);
   						\draw[] (X) circle (3.7);
   						\draw[] (Y) circle (1.6);
   						\draw[] (Z) circle (2);
   						\draw[thick] (X) to (Y);
\end{tikzpicture}
\begin{tikzpicture}[scale=0.65, every node/.style={scale=1},show background rectangle]
						\node (X) at (0,1.5)[circle, fill=blue, label = {$x$}] {};
    					\node (Y) at (3,3) [circle, fill=blue, label = {below:$y$}] {};
   						\node (Z) at(6.3, 1) [circle, fill=blue, label = {right: $z$}] {};
   						\draw[color = white] (Z) circle (4.2);
   						\draw[] (X) circle (3.7);
   						\draw[] (Y) circle (1.6);
   						\draw[] (Z) circle (2);
   						\draw[thick] (X) to (Y);
   						\draw[dashed] (X) circle (4);
   						\draw[dashed] (X) to (3.9,0.577);
   						\node (E1) at (1.4,0.4)[]{$\beta W(x,z)R_x$};
\end{tikzpicture}
\begin{tikzpicture}[scale=0.65, every node/.style={scale=1},show background rectangle]
						\node (X) at (0,1.5)[circle, fill=blue, label = {$x$}] {};
    					\node (Y) at (3,3) [circle, fill=blue, label = {below:$y$}] {};
   						\node (Z) at(6.3, 1) [circle, fill=blue, label = {right: $z$}] {};
   						\draw[color = white] (X) circle (4);
   						\draw[] (X) circle (3.7);
   						\draw[] (Y) circle (1.6);
   						\draw[] (Z) circle (2);
   						\draw[thick] (X) to (Y);
   						\draw[dashed] (Z) circle (4.2);
   						\draw[dashed] (Z) to (2.2, 0);
   						\node (E2) at (4.4,-0.2)[]{$\beta W(y,z)R_z$};
   						\draw[thick] (Z) to (Y);
\end{tikzpicture}
}
\end{center}
\caption{Example for the connection mechanism of the soft Boolean model in two dimensions. The solid lines represent the graph's edges. Left: the black circles represent the individual interaction balls governed by $R_x,R_y,R_z$ which already gives rise to an edge $(x,y)$. Middle: the enlarged dashed circle centred at $x$ is due to the edge weight $W(x,y)$, which in this case does not lead to a new edge. Right: On the other hand, the enlarged dashed circle centred at $z$ now includes the point $y$, leading to an edge $(y,z)$.}
\label{fig:SoftBool}
\end{figure}
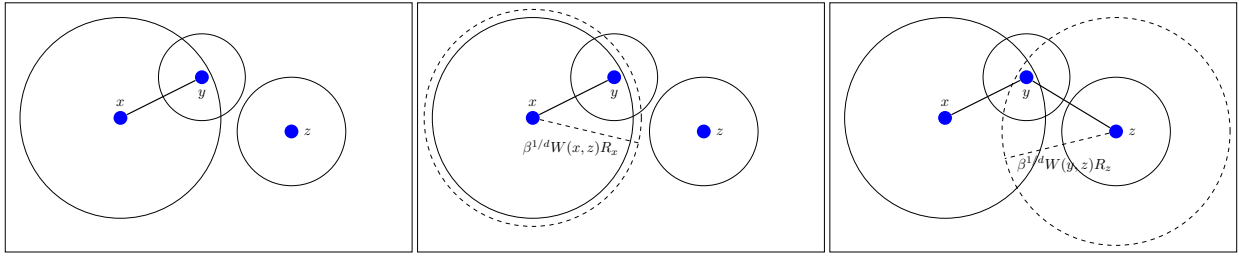

A variant of the model is given by replacing~\eqref{eq:defSoftBool} by the connection rule \(|x-y|\leq \beta^{1/d} W(x,y)\big(R_x + R_y\big)\),
where now any two vertices are connected when their enlarged balls intersect. However, since \(R_x\vee R_y\leq R_x+R_y\leq 2(R_x\vee R_y)\), there is no qualitative change of behaviour between the two model when it comes to the existence of global phase transitions. We therefore stick to the model~\eqref{eq:defSoftBool}.  
It is easy to see that the soft Boolean model interpolates between long-range percolation and the heavy-tailed Boolean model: The parameter \(\delta\) controls the rigidity of the geometric restrictions of the {Boolean model, where vertices must be located within their radius to form an edge. These restrictions are now softened by the \emph{independent} edge weights so} that they become stronger when \(\delta\) increases. Put differently, the smaller \(\delta\), the less a vertex feels the geometry, adding long-range effects to the model. The requirement \(\delta>1\) is needed to remain sparse in the sense that every vertex has finite expected degree. By sending \(\delta\to\infty\), one obtains a variant of the classical Boolean model. Similarly, {\(1/\gamma\), for \(\gamma\in(0,1)\),} measures the size of the vertex radii (or their volume, respectively). The larger \(\gamma\), the larger the radii and hence the more powerful vertices can be found in the graph. Since the radii are heavy-tailed, \(\G^\beta\) has a heavy-tailed degree distribution with tail exponent \(1/\gamma\), see~\cite{Lue2022, GHMM2022}. Again, the constraint \(\gamma<1\) is needed to maintain sparseness. By taking the limit \(\gamma\to 0\), one obtains a version of long-range percolation in the continuum.

\subsection{Results on the Euclidean diameter}
We are interested in the size of the component of a typical vertex in the subcritical regime. Let  \(x\leftrightarrow y\) denote the event that two vertices are connected in \(\G^\beta\) by a finite path of directly connected vertices. Let further \((o,R_o)\) be a vertex placed at the origin, assigned with an independent radius and added to the graph according to the rule~\eqref{eq:defSoftBool}. Denote the underlying law after the vertex \(o\) has been added by \(\P_o\). According to Palm theory, the vertex \(o\) plays the role of a typical vertex shifted to the origin~\cite{LastPenrose2017} and we study its component in the following. Denote by
\begin{equation*}
    \C_\beta:=\big\{x\in X\colon o\leftrightarrow x \text{ in }\G^\beta\big\}
\end{equation*}
the component of the origin {and by \(\scN_\beta=\sharp\C_\beta\) its cardinality.} When there is no infinite component in the graph, the component of the origin must be finite. Vice versa, if there is an infinite component, the origin is part of it with a positive probability. We can therefore write the critical percolation intensity as
\begin{equation*}
	\beta_c :=\beta_c(\gamma, 0, \delta):= \sup\big\{\beta>0\colon  \P_o(\scN_\beta=\infty)=0\big\},
\end{equation*}
where the notation of \(\beta_c(\gamma,0,\delta)\) is made clear in Section~\ref{sec:Construction} once we introduced the more general framework. It was shown in~\cite{GLM2021} that \(\beta_c>0\) if {\(1/\delta<1/\gamma-1\) but that \(\beta_c=0\) if \(1/\delta>1/\gamma-1\),} so we focus on the first parameter regime in the following. We are interested in the Euclidean diameter of the component of the origin in a subcritical regime as it measures how far a typical component still spreads out. We define
\begin{equation}\label{eq:defDiam}
		\mathscr{D}_\beta :=\sup\big\{|x|\colon x\in\C_\beta\big\}
\end{equation}
and call it the \emph{Euclidean diameter} of \(\C_\beta\). Note that \(\mathscr{D}_\beta\) is actually the radius of the smallest ball around \(o\) that contains the whole component. However it is of the same order as \(\sup\{|x-y|\colon x,y\in\C_\beta\}\) and we thus stick to the name Euclidean diameter in agreement
with~\cite{Gouere08}. Let us state our main result about the tail probability of the  Euclidean diameter of a typical cluster in a {low edge-intensity regime. Note that the event \(\{\mathscr{D}_\beta>m\}\) says that the origin is connected to distance at least \(m\), which is also referred to as the \emph{one-arm event}.}

\begin{theorem}[Results on the Euclidean diameter]\label{thm:MetaMain} 
	Let \(d\geq 1\), \(\delta>1\), and \(1/\delta<1/\gamma-1\). Then, there exists \(\widetilde{\beta}>0\) such that, for all \(\beta<\widetilde{\beta}\), we have,
	\begin{enumerate}[(i)]
		\item if \(\delta\leq 1/\gamma -1\), then			
			\[
				\lim_{m\to\infty} \frac{\log\P_o(\mathscr{D}_\beta>m)}{\log m} = {-d(\delta-1)},\quad\text{and}
			\]
		\item if \(\delta>1/\gamma-1\), then
			\[
				\lim_{m\to\infty}\frac{\log\P_o(\mathscr{D}_\beta>m)}{\log m} = {-d\frac{\delta-1}{\delta}\Big(\frac{1}{\gamma} -1\Big)}.
			\]
	\end{enumerate}
\end{theorem}
\begin{remark}
	We deduce the following bounds for \(\widetilde \beta\) from our proofs
	\begin{equation}\label{eq:boundsForTailBeta}
	   \widetilde{\beta}\geq  
        \begin{cases}
		  \frac{1}{2^{d\delta+3}+1}\cdot\tfrac{\delta-1}{\omega_d \delta}\big(1-\gamma\tfrac{\delta+1}{\delta}\big), & \text{if }\tfrac{1}{\gamma}\leq 2, \\
		  \tfrac{\delta-1}{\omega_d\delta}\big(\tfrac{1}{1-2\gamma}+\tfrac{1}{1-\gamma}\big)^{-1} \vee  \frac{1}{2^{d\delta+3}+1}\cdot\tfrac{\delta-1}{\omega_d \delta}\big(1-\gamma\tfrac{\delta+1}{\delta}\big), & \text{if } \tfrac{1}{\gamma}> 2,
		  \end{cases}
	\end{equation}
where \(\omega_d\) denotes the volume of the \(d\)-dimensional unit ball. In fact, we believe that the tail of the Euclidean diameter decays with the derived rates throughout the whole subcritical regime. This would particularly imply that \(\widetilde{\beta}\) coincides with the standard critical intensity of percolation \(\beta_c\) and we discuss this, even further in the context of various other critical values, in greater detail in Section~\ref{sec:openProblems}.
\end{remark}

{Theorem~\ref{thm:MetaMain} is a direct consequence of Theorem~\ref{thm:Main} below, which is proved in Section~\ref{sec:proofs}. The distinction between the two phases becomes apparent when considering suitable lower bounds for \(\P_o(\mathscr{D}_\beta>m)\). The corresponding exponents reflect the interpolation between two limiting models: the long-range percolation model, where edges appear independently and a single edge may span distance \(m\), and the Boolean model, where the presence of high-degree vertices is exploited by using two edges. 
To make this precise, observe first that the tail of the \emph{degree distribution} in the soft Boolean model coincides with that in the standard Boolean model and is thus determined by the \emph{volume} of large balls, i.e., \(\P_o(\text{degree of }o\geq k)\asymp \P_o(R_o^d \geq k) \asymp k^{-1/\gamma}\)~\cite{Lue2022,GHMM2022}. However, the largest degree in the neighbourhood of \(o\), used as an intermediate vertex in the two-edge strategy, follows the \emph{size-biased} degree distribution with tail index \(1-1/\gamma\). This effect is sometimes referred to as \emph{friendship paradox} in the complex network literature and can easily be seen by an exploration of the origin's neighbourhood described by a branching process. The distinction between Cases~(i) and (ii) therefore depends on whether the edge-weight distribution or the size-biased degree distribution has the heavier tail.

\begin{description}
	\item[Case~(i).]
		The tail of the edge weights is heavier and the diameter is dominated by the pure long-range effect and a lower bound for \(\P_o(\mathscr{D}_\beta>m)\) then arises from the probability that a single long edge connects \(o\) to distance \(m\).
	\item[Case~(ii).]
		The size-biased degree distribution has the heavier tail and it becomes advantageous to exploit high-degree vertices. However, the long-range effects still play a role. Consider first the standard Boolean model, where long-range effects are absent. Connecting \(o\) to distance \(m\) via two edges requires that a high-degree neighbour of \(o\) within distance \(m\) has radius of order \(m\). Since the degree particularly depends on the volume \(m^d\) of the ball, applying the size-biased degree distribution gives probability of order \(m^{d(1-1/\gamma)}\), agreeing with the results of~\cite{Gouere08}, which corresponds to the \(\delta\to\infty\) limit of our exponent. 
		In the soft Boolean model, however, even weak long-range effects allow high-degree vertices to connect beyond their own radius, so a smaller radius suffices. As shown in Lemma~\ref{lem:MainThmLGLower}, a radius of order \(m^{(\delta-1)/\delta}\) already ensures the existence of a neighbour at distance \(m\) with positive probability. Since this radius must be realised by the high-degree neighbour of \(o\), applying the size-biased degree distribution yields a probability of order \(m^{d[(\delta-1)/\delta](1-{1/\gamma})}\), implying the exponent in Case~(ii). 
\end{description}

Note that neither strategy depends qualitatively on the value of \(\beta\). The requirement that \(\beta\) is sufficiently small is only needed in the proof of the upper bound, where control of potentially long paths is required.

Similar considerations also appear in other contexts. More precisely, in~\cite{JorritsmaMitscheKomjathy2023}, the \emph{vertex boundary exponent}~\(\zeta\) is introduced to quantify the occurrence of long edges. Here, \(m^{d\zeta}\) gives the order of the expected number of vertices in a ball of radius \(m\) that have a neighbour of smaller radius at distance \(m\).
If \(\zeta>0\) such vertices occur with high probability and~\cite{JorritsmaMitscheKomjathy2023,jorritsmaMitscheKomjathy2023LRP,jorritsmaMitscheKomjathy2024LDP} show that, in many models, the tail distribution of finite clusters in \emph{supercritical} regimes decays stretched exponentially with exponent \(\zeta\), while~\cite{Moench2023} proves continuity of the percolation function and transience.  
If, on the other hand, \(\zeta<0\), vertices incident to long edges are rare and the probability of finding a long edge in the ball of radius \(m\) is precisely of order \(m^{d\zeta}\). In~\cite{jacobJahLu2024}, it is shown that this can be used to describe subcritical regimes via the crossings of large annuli, an idea that goes back to~\cite{Gouere08}, which characterises the subcritical phase in the standard Boolean model. Specifically,~\cite{jacobJahLu2024} shows that large annuli are either crossed by a single long edge or only by long paths whose existence can be suppressed by choosing \(\beta\) small enough. In the latter situation, the probability of traversing a large annulus between balls of radius \(m\) and \(2m\) is again of order \(m^{d\zeta}\). Furthermore,~\cite{luechtrath24} shows that when long edges are rare, graph distances are comparable to Euclidean distances. 
 To determine \(\zeta\), there are again two possible strategies: either the radii play no particular role, and one tries to independently connect pairs of vertices at distance \(m\); or the radii do play a role, and one requires a vertex of radius at least \(m^{(\delta-1)/\delta}\) within the \(m\)-ball as in Case~(ii) above. The important difference is that the vertices are no longer required to be connected to the vertex at the origin, so that there are of order \(m^d\) candidate vertices that may form the long edge. In the first case, this leads to \(\zeta=2-\delta\). In the second case, the vertex radii here follow the original distribution (rather then the size-biased one), leading to \(\zeta=1-(\delta-1)/(\gamma\delta)\). Note that the dominant effect is now determined by whether \(\delta<1/\gamma\) or not, meaning that the original Pareto tails of radii and edge-weights are compared. In summary, \(\zeta\) governs the strength of long-range effects and carries important structural information. However, it alone does not characterise subcritical behaviour. Specifically, there exists parameter regimes where \(\zeta>0\) yet the model has a subcritical phase. Moreover, \(|\zeta|\leq [(\delta-1)/\delta](1/\gamma-1)\wedge (\delta-1)\) so that \(\zeta\) does not accurately describe the decay of the Euclidean diameter. As a final remark, note that for the standard Boolean model, where no long-range effects are present, \(\zeta\) accurately describes the subcritical phase.

All the above remains true if we replace the Pareto distributed radii by more general heavy-tailed radii. As long as the new radius distribution has the same tail exponent, thus the same integrability properties, our results remain true in the sense that the decay exponent of the Euclidean diameter does not change {although it may, of course, change the bounds on \(\widehat{\beta}\) in~\eqref{eq:boundsForTailBeta}.} The same applies to our results on the number of points, presented in the following section.  

\subsection{Results on the number of points}
	Another important quantity to describe the subcritical typical component is the number of vertices it contains{, denoted by \(\scN_\beta\)}. Together with the Euclidean diameter, this quantity characterises both the spatial spread out and size of the component. While the Euclidean diameter is driven by the longest edges, the cardinality is driven by the highest degrees in the cluster. Indeed, due to subcriticality, there are no `long' paths (in the sense of the graph distance), hence, reaching many vertices with short paths only, requires at least a few high-degree vertices. There is however one direct connection between \(\mathscr{N}_\beta\) and \(\mathscr{D}_\beta\). {If \(\mathscr{N}_\beta>m^d\) occurs, then either \(\mathscr{D}_\beta>c m\) for some appropriate \(c>0\) or the underlying Poisson process must contain atypically many points in the ball of radius \(cm\). As the Poisson process concentrates exponentially fast around its mean, \(\P_o(\mathscr{N}_\beta>m^d)\) decays at least as fast as \(\P_o(\mathscr{D}_\beta>cm)\). For the standard Boolean model both probabilities are indeed of the same order~\cite{JahnelAndrasCali2022, Gouere08}. More precisely, in that case the neighbour with radius \(m\) that connects  \(o\) to distance \(m\) contributes its roughly \(m^d\) own neighbours to the cardinality. Since its degree follows the size-biased degree distribution, we infer the same tail behaviour. }
	
	The situation is different in long-range percolation, where the subcritical Euclidean diameter is heavy tailed with tail exponent \(d(\delta-1)\) but the tail of the number of vertices decays exponentially fast~\cite{Meester1995,MeesterRoy1996}. This is due to the fact that the long-range edges connect vertices at large distances, yielding a heavy-tailed Euclidean diameter, but the overall degree distribution is not strong enough to guarantee a heavy-tailed number of vertices in the component. 
    {Indeed, the degree distribution in long-range percolation has an exponentially decaying tail and consequentially so has the size-biased version.
	
	In the soft Boolean model both effects mix. As discussed above, large degrees are mainly driven by large balls, an effect that is independent of the edge weights. Indeed, edge weights allow a vertex to connect beyond its radius, but because of independence and integrability, the vast majority of neighbours of a large-ball vertex still lie within a constant multiple of its radius. Hence, to realise \(\mathscr{N}_\beta\geq m^d\) via a single large-degree neighbour, a vertex of radius of order \(m\) is still required, as the long-range effect cannot qualitatively contribute. This contrasts with the Euclidean diameter where a radius of order \(m^{(\delta-1)/\delta}\) suffices to connect the origin to distance \(m\), yet such a vertex only contributes \(m^{d(\delta-1)/\delta}<m^d\) neighbours. 
	
	Summarising, a large subcritical cluster must contain at least one vertex with large radius, determined solely by the size-biased degree distribution. Using a branching-process approach, we establish this in the regime where the size-biased degree distribution has finite expectation, and the edge intensity is sufficiently small.
	}
 
 	\begin{theorem}[Subcritical cardinality]\label{thm:NumberPoints}
		Let \(d\geq 1\), \(\delta>1\), and \(0<\gamma<1\). 
		\begin{enumerate}[(i)]
			\item If \(1/\gamma-1>1\), then there exists \(\beta'>0\) such that, for all \(0<\beta<\beta'\), there exist constants \(c,C\in(0,\infty)\) such that, for all \(m>1\),
			\[
				c m^{1-1/\gamma}\leq \P_o(\mathscr{N}_\beta>m)\leq C m^{1-1/\gamma}.
			\]
			\item If \(1/\gamma-1<1\), then \(\E \mathscr{N}_\beta = \infty\) for all \(\beta>0\).
		\end{enumerate}
	\end{theorem}

	{We believe the result to hold for the whole parameter regime where there is a subcritical phase, i.e., \(1/\delta<1/\gamma-1\) and within the whole subcritical phase, i.e., \(\beta'=\beta_c\). We comment on this in Section~\ref{sec:openProblems}. The proof is given in Section~\ref{sec:NumberPoints}.
	}
	
	Let us finally note that the lower bound of Case~(i) in Theorem~\ref{thm:NumberPoints} is valid for all \(\gamma\in(0,1)\) by our arguments above. In order to prove the upper bound of the theorem and apply the aforementioned branching argument, we couple the soft Boolean model with a version of \emph{continuous scale-free percolation}~\cite{DeprezWuthrich2019,DeijfenHofstadHooghiemstra2013} in which the minimum \(R_x\vee R_y\) in~\eqref{eq:defSoftBool} is replaced by the product \(R_x R_y\). As a result, we immediately observe the same bounds for this model as well. Interestingly, this model only has a subcritical phase for \(1/\gamma\geq 2\), see~\cite{DeprezWuthrich2019,GLM2021}. {Hence, except for the boundary case \(1/\gamma=2\), the result covers the whole parameter regime for which a subcritical phase exists. Again, we require a sufficiently small intensity but believe the result to be true within the whole subcritical regime.    }
		
	\begin{theorem}[Subcritical cardinality in scale-free percolation] \label{thm:SFP}
		Let \(d\geq 1\), \(\delta>1\), and  \(1/\gamma-1>1\). Consider continuous scale-free percolation, i.e.\ the graph constructed via connection rule~\eqref{eq:defSoftBool} with \(R_x\vee R_y\) replaced by \(R_x R_y\) and denote the cardinality of the component of the origin by \(\mathscr{N}_{\beta,\gamma,\gamma,\delta}\). Then, there exists {\(\beta'>0\) such that, for all \(0<\beta<\beta'\), there exist constants \(c,C\in(0,\infty)\) such that}, for all \(m>1\),
		\[
			c m^{1-1/\gamma}\leq \P_o(\mathscr{N}_{\beta,\gamma,\gamma,\delta}\geq m)\leq C m^{1-1/\gamma}. 
		\]
	\end{theorem}
    \newpage
	\begin{remark}  ~\
		\begin{enumerate}[(i)]
			\item 
				By the same coupling argument, Theorem~\ref{thm:NumberPoints} actually applies to all graphs constructed via connection rule~\eqref{eq:defSoftBool} but with \(R_x\vee R_y\) replaced by a symmetric function \(f(R_x,R_y)\) that satisfies
				\[
					R_x\vee R_y \leq f(R_x,R_y) \leq R_x R_y.
				\]
			\item 
			     Continuous scale-free percolation admits a natural hyperbolic interpretation: in one dimension with unit edge weights, it appears (after rescaling) as a weak local limit of the hyperbolic random graph~\cite{KomjathyLodewijks2020,KrioukovEtAl2010}. In higher dimensions, and with additional edge weights, it can be seen as a similar extension for the hyperbolic random graph as the soft Boolean model is for the standard Boolean model.  
		\end{enumerate}
	\end{remark}

\subsection{Formal construction} \label{sec:Construction}
	We formally introduce the soft Boolean model as a special instance of the \emph{weight-dependent random connection model}~\cite{GHMM2022}. Recall that \(X\) denotes a unit-intensity Poisson point process on~\(\R^d\). We can write 
	\(
		X = (X_i\colon i\in\N),
	\)
	c.f.~\cite{LastPenrose2017} and call the elements of \(X\) the \emph{vertex locations}. Let  \(\U=(U_i\colon i\in\N)\) be a family of independent random variables distributed uniformly on \((0,1)\) that we call the \emph{vertex marks}. Let us further write 
	\[
		\cX := \big(\X_i=(X_i,U_i)\in X\times\U\colon i\in\N\big)
	\]
	for the set of vertices and note that \(\cX\) is a unit-intensity Poisson point process on \(\R^d\times(0,1)\). Finally, let \(\mathscr{V}=(V_{i,j}\colon i<j\in\N)\) be another independent sequence of uniformly-distributed random variables on \((0,1)\) that we call the \(\emph{edge marks}\). We define
	\begin{equation}\label{eq:edgeMarking}
		\xi := \big(\big(\{\X_i,\X_j\},V_{i,j}\big)\in\cX^{[2]}\times(0,1),i<j\in\N\big),
	\end{equation}
	where \(\cX^{[2]}\) denotes the set of all subsets of \(\cX\) of size two. We call \(\xi\) an \emph{independent vertex-edge marking} in accordance with the construction in~\cite{HvdHLM20}. Note that \(\xi\) is an ergodic point process on \((\R^d\times(0,1))^{2}\times(0,1)\). Further note that the law of \(\xi\) does not depend on the ordering of the points and that \(\cX\) as well as \(X\) and \(\U\) can be recovered from it. 
	
	Now, fix \(\beta>0\), \(\gamma\in[0,1)\), \(\alpha\in[0,2-\gamma)\) and \(\delta>1\). We define the \emph{interpolation kernel} 
	\begin{equation}\label{eq:InterpolKern}
		g_{\gamma,\alpha}(s,t):=(s\wedge t)^\gamma (s\vee t)^\alpha, \quad s,t\in(0,1),
	\end{equation}
	as introduced in~\cite{GraLuMo2022} and the \emph{profile function}
	\[
		\rho(x):=\rho_\delta(x)=1\wedge x^{-\delta}, \quad x\in(0,\infty).
	\]
	The undirected graph \(\cG^{\beta,\gamma,\alpha,\delta}(\xi)\) is then defined through its vertex set \(\cX\) and edge set
	\begin{equation}\label{eq:edgeSet}
		E\big(\cG^{\beta,\gamma,\alpha,\delta}(\xi)\big) = \Big\{\{\X_i,\X_j\}\colon  V_{i,j}\leq \rho\big(\beta^{-1}g_{\gamma,\alpha}(U_i,U_j)|X_i-X_j|^d\big), i<j\in\N\Big\}.
	\end{equation}
	Note that for \(\alpha=0\), the soft Boolean model \(\G^{\beta}\) defined above and the graph \(\cG^{\beta,\gamma,0,\delta}(\xi)\) have the same law as
	\[
		V_{i,j}\leq \big(\beta^{-1}(U_i\wedge U_j)^\gamma|X_i-X_j|^d\big)^{-\delta}\qquad \Longleftrightarrow\qquad |X_i - X_j|^d\leq \beta V_{i,j}^{-1/\delta}\big(U_i^{-\gamma}\vee U_j^{-\gamma}\big).
	\]
	Therefore, we identify \(\cG^\beta=\cG^{\beta,\gamma,0,\delta}(\xi)\) in the following. Essentially, the parametrisation is now with respect to the \emph{volume} of the associated balls. In particular, \(U_j^{-\gamma}\) describes the volume of the ball associated to vertex \(X_j\), which is Pareto distributed with tail exponent \(1/\gamma\). Similarly, \(V_{i,j}^{-1/\delta}\) follows a Pareto distribution with tail exponent \(\delta\) thus following the same distribution as the \(d\)-th moment of the edge weights. The advantage of the volume parametrisation is that it makes parameters and results independent of the dimension and is, from our perspective, easier to work with. We therefore work from now on explicitly on the probability space on which the vertex-edge marking \(\xi\) lives and denote the underlying probability measure by \(\P\) and the corresponding expectation by \(\E\).
		
	The same parametrisation can also be used to consider models without long-range effects, particularly the random geometric graph or the standard Boolean model. To this end, only the function \(\rho\) in~\eqref{eq:edgeSet} has to be replaced by the indicator \(\1_{[0,1]}\). We identify this choice of profile function with \(\delta=\infty\) as it arrives as the \(\delta\to\infty\) limit of our previous choice. In this notion, the random geometric graph with edge length \(\beta\) is given by \(\cG^{\beta,0,0,\infty}(\xi)\), and the version of the Boolean model where connections are formed when the ball of the stronger vertex contains the weaker vertex is given by \(\cG^{\beta,\gamma,0,\infty}(\xi)\).
	
Next, in order to formulate our main result, we required a vertex located to the origin that had been added to the graph. To do so formally, let us denote by \(\x_0=(x_0,u_0)\) a vertex placed at \(x_0\in\R^d\) and with a given vertex mark \(u_0\in(0,1)\). Let  \(\cX_{\x_0}=\cX\cup\{(x_0,u_0)\}\) denote the vertex set with the additional vertex added. Note that almost surely no vertex in \(X\) has been placed at $x_0$ before. Let us further extend the edge marks by a sequence of independent uniform random variables \((V_{0,j}:j\in\N)\) and denote the resulting sequence by \(\mathscr{V}_{\x_0}\). Finally, we define the vertex-edge marking containing the extra vertex
	\[
		\xi_{\x_0}:=\xi \cup\big\{\big(\{\x_0,\X_i\},V_{0,i}\big)\colon i\in\N\big\}.	
	\]
	The graph containing the extra vertex is then \(\cG^{\beta,\gamma,\alpha,\delta}(\xi_{\x_0})\). In case of the soft Boolean model, we refer to it as \(\cG^\beta_{\x_0}\). We denote the probability measure and expectation governing \(\xi_{\x_0}\) by \(\P_{\x_0}\) or \(\P_{(x_0,u_0)}\), and \(\E_{\x_0}\), respectively. If the vertex mark of \(\x_0\) is not fixed but uniformly distributed independently from everything else, we also denote the vertex by \(\x_0=(x_0,U_0)\) and define \(\P_{x_0}:=\P_{(x_0,u)}\d u\). If $(x_0,U_0)=(o,U_o)$, i.e.\ the vertex is located at the origin, \(\P_o = \P_{(o,u)}\d u\) is the Palm version of the graph, which can be seen as the graph shifted such that a typical (i.e., uniformly chosen) vertex is located at the origin~\cite[Chapter~9]{LastPenrose2017}. Note that this is consistent with previous notation whenever the considered graph coincides with the soft Boolean model. If this is the case, we may also index our objects as before by~\(o\).
	
	In the same way, finitely-many given vertices \(\y_1=(y_1,t_1), \y_2=(y_2,t_2),\dots\) can be added to the graph using  negative indices and writing \(y_i=x_{-i}\) and \(t_i=u_{-i}\). We then write \(\xi_{\y_1,\y_2,\dots}\), and \(\P_{\y_1,\y_2,\dots}\), etc. 
	
	It is important to note that in the formal construction in the previous section the ordering of the Poisson points was important. However, the precise ordering does not change any distributional properties. 
 Therefore, 
 we drop this notation from here on onwards and denote given vertices by \(\x=(x,u_x)\) or \(\y=(y,u_y)\) but stick with the notation \(\o=(o,u_o)\) for the origin. We further write \(\x_1,\x_2,\dots,\x_n\) for any sequence of given vertices without referring to the ordering above. For two given vertices, we denote by \(\{\x\sim\y \text{ in } \cG^{\beta,\gamma,\alpha,\delta}(\xi_{\x,\y})\}\) the event that \(\x\) and \(\y\) are connected by an edge in the graph \(\cG^{\beta,\gamma,\alpha,\delta}(\xi_{\x,\y})\). If the graph is clear from the context, we simply write \(\{\x\sim\y\}\). Similarly, we denote by \(\{\x\leftrightarrow\y \text{ in }\cG^{\beta,\gamma,\alpha,\delta}(\xi_{\x,\y})\}\) (resp.\ \(\{\x\leftrightarrow\y\}\)) the event that \(\x\) and \(\y\) are connected by a finite path in the graph. We denote by \(\C_{\beta,\gamma,\alpha,\delta}\) the component of the origin in \(\cG^{\beta,\gamma,\alpha,\delta}_o\). {We further denote 
  	by \(\mathscr{D}_{\beta,\gamma,\alpha,\delta}\) and \(\scN_{\beta,\gamma,\alpha,\delta}\) the Euclidean diameter and the cardinality of the component. We abbreviate by \(\C_\beta\), \(\mathscr{D}_\beta\), and \(\scN_\beta\) the component quantities in the soft Boolean model \(\cG^\beta_o\).}
 	 
 	 \paragraph{Tail bounds for \(\M_\beta\).} {Rather than proving Theorem~\ref{thm:MetaMain} directly, one can also derive tail bounds for \(\M_\beta\), which aligns better with our volume-based viewpoint of the model (i.e., attaching balls of a certain volume rather than radii to the vertices) and keeps our results independent of the dimension.} 
 	 The tail bounds in Theorem~\ref{thm:MetaMain} and the lower bounds for \(\widetilde{\beta}\) in~\eqref{eq:boundsForTailBeta} are consequences of the following theorem.
 	 
 	 \begin{theorem}\label{thm:Main}
	Let \(d\geq 1\), \(\delta>1\), and {\(1/\delta<1/\gamma-1<\infty\)} and set \(\beta_0:=\frac{1}{2^{d\delta+3}+1}\cdot\tfrac{\delta-1}{\omega_d \delta}\big(1-\gamma\tfrac{\delta+1}{\delta}\big)>0\). Then, for all \(\beta<\beta_0\), there exist constants \(c_1,C_2,c_3,C_4,c_5,C_6 \in(0,\infty)\), depending on \(\beta,\gamma\), and \(\delta\), such that for all \(m>1\),
	\begin{enumerate}[(i)]
		\item if \(\delta < 1/\gamma-1\), we have
			\[
				c_1m^{1-\delta}\leq \P_o(\M_\beta>m)\leq C_2m^{1-\delta},
			\]
        \item if \(\delta = 1/\gamma-1\), we have
            \[
                c_3 m^{1-\delta} \leq \P_o(\M_\beta >m)\leq C_4 \log(m) m^{1-\delta}, \text{ and }
            \]
		\item if \(1/\delta<1/\gamma-1<\delta\), we have
			\[
				c_5 m^{-\tfrac{\delta-1}{\delta}(\tfrac{1}{\gamma}-1)}\leq \P_o(\M_\beta>m)\leq C_6 \log(m)^{1\vee d(\delta-1)} m^{-\tfrac{\delta-1}{\delta}(\tfrac{1}{\gamma}-1)}.
			\]
	\end{enumerate}
\end{theorem}

Let us remark here for completeness that the alternative lower bound for \(\widetilde \beta\) in~\eqref{eq:boundsForTailBeta} in the \(1/\gamma>2\) regime is a consequence of Proposition~\ref{thm:UpperBoundSmallGamma} below.

	\subsection{Strategy of proof}\label{sec:Strategy}
	We briefly explain the strategy of the proof in this section. The upper bounds in Theorems~\ref{thm:MetaMain} and~\ref{thm:Main} rely on controlling the number of admissible paths. As in the lower bounds, two mechanisms compete: connection through a single long edge when edge weights dominate, or through a high-degree neighbour with a large ball when the size-biased degree distribution dominates. For upper bounds, however, we must consider full paths, and our main tool is a decomposition into skeletons and connectors, which ensures exponential suppression of long paths in the low-intensity regime.

	\begin{description}
		\item[Step~(A)\label{stepA}] In order to get decent bounds on the probability that certain paths exist, we introduce the concept of the skeleton of a path. The skeleton vertices can be seen as the key vertices of the path and allow us to decompose each path in a set of skeleton vertices and a set of connectors  building the subpaths between two skeleton vertices. {This allows us to bound the probability that the full path exists by the expected number of skeletons that form their own path, multiplied by a term that decays exponentially fast in the remaining path length if \(\beta\) is sufficiently small. Since the expected number of such skeleton paths also decays exponentially in the number of skeleton vertices for small enough \(\beta\), we obtain exponential bounds on the probability of long paths, cf.\ Section~\ref{sec:MainThmLGSkeleton}. This explains why our method requires the low-intensity regime and cannot cover the full subcritical phase.} 
	\end{description} 
	
	{Once this is established, the upper bound for \(\P_o(\M_\beta>m)\) in the single-edge case follows immediately, as we can directly apply \ref{stepA} to all paths that contain an edge of length of order \(m^{1/d}\). The other regime is more involved, as vertices with large balls play a central role, and they also strongly affect the expected number of paths. We must therefore separate paths according to whether they include a powerful vertex and, if not, whether they rely on long edges or many short edges.	
	\begin{description}
		\item[Step~(B)\label{stepB}] 
			We use \ref{stepA} to deduce that in a low-intensity regime the probability that the origin's cluster contains a vertex with associated volume of order \(m^{(\delta-1)/\delta}\) is of the same order as the existence of a direct neighbour of that strength, cf.\ Lemma~\ref{lem:upperPowerfulVertex}. 
		\item[Step~(C)\label{stepC}] 
			Next, we consider paths that connect the origin to distance \(m^{1/d}\) without using strong vertices where the last vertex of the paths is located far away; more precisely, at distance \(2m^{1/d}\). We show that the probability of this happening is of the same order as the lower bound, cf.\ Lemma~\ref{lem:MainThmLGupperH}.
		\item[Step~(D)\label{stepD}] 
			Finally, we consider paths that connect to distance \(m^{1/d}\) and that do not use strong vertices as well as end within distance \(2m^{1/d}\) of the origin. We distinguish two kind of such paths, those using an edge longer than \(m/\log m\) and those that do not use such an edge and must therefore consist of at least \(\log m\) edges. We use a calculation comparable to \ref{stepC} to deal with the first case and \ref{stepA} once more to deal with the second case. 
	\end{description}
	
	Let us remark that vertices with balls of volume \(m^{(\delta-1)/\delta}\) yield the main contribution to \(\P_o(\M_\beta>m)\) still. This is most transparent in \ref{stepB} and in the calculations regarding \ref{stepC}, which is dominated by the case where the final long edge is connected to a powerful vertex that has almost volume \(m^{(\delta-1)/\delta}\). However, the logarithmic correction in Part~(iii) of Theorem~\ref{thm:Main} is a result of \ref{stepD} where we have to restrict edge lengths of slightly smaller order than \(m\) to force the remaining paths to contain a growing number of edges.  
	}


	\subsection{Further discussion} \label{sec:openProblems}
	In this section, we discuss some further details which particularly concern critical intensities such as \(\beta_c\) and the logarithmic correction in Part~(iii) of the Theorem~\ref{thm:Main}.	
	\subsubsection{Critical intensities}
	Let us consider the standard definition of the critical percolation intensity for \(\cG^{\beta,\gamma,\alpha,\delta}(\xi_o)\) i.e. 
	\begin{equation}\label{eq:betacr}
		\begin{aligned}
		\beta_c(\gamma,\alpha,\delta)& := \sup\big\{\beta>0\colon \P_o(\sharp\C_{\beta,\gamma,\alpha,\delta}=\infty)=0\big\} \\ 
			& =\sup\big\{\beta>0\colon \lim_{n\to\infty}\P_o\big(o\text{ starts a path of length }n \text{ in }\cG^{\beta,\gamma,\alpha,\delta}\big)=0\big\},
		\end{aligned}
	\end{equation}
	where the function \(\theta(\beta):=\P_o(\sharp\C_{\beta,\gamma,\alpha,\delta}=\infty)\) is also called the percolation function. To prove the existence of a subcritical percolation phase in the standard Boolean model \(\cG^{\beta,\gamma,0,\infty}\) and to obtain the tails of the Euclidean diameter in this phase, Gou\'{e}r\'{e} introduced in~\cite{Gouere08, } a new critical intensity (translated to our setting) as
	\[
		\widehat{\beta}_c(\gamma,\alpha,\delta):=\sup\big\{\beta>0\colon \liminf_{m\to\infty}\P_o\big(\exists \x, \y\in\cX\colon |x|^d<m, |y|^d>2^d m, \ \x\leftrightarrow\y \text{ in }\cG^{\beta,\gamma,\alpha,\delta}\big)=0\big\}.
	\] 
	Generally, \(\widehat{\beta}_c(\gamma,\alpha,\delta)\) can be seen as the critical \emph{annulus-crossing} intensity. {For the Boolean model, this annulus-crossing event can further be interpreted as the event that the Euclidean diameter of the component of the origin exceeds \(2^d m\) when setting the radius of the vertex at the origin to \(m^{1/d}\). Considering such an annulus-crossing event has the advantage that it allows for renormalisation arguments in many situations.} Since \(\widehat{\beta}_c(\gamma,\alpha,\delta)\leq\beta_c(\gamma,\alpha,\delta)\), the positivity of the first critical intensity implies the existence of a subcritical phase. 
	Results for the Boolean model in~\cite{AhlbergTassionTeixeira2018,DCopinRaoufiTassion2020,DembinTassion2022} indicate further that \(\widehat{\beta}_c(\gamma,0,\infty)= \beta_c(\gamma,0,\infty)\). In fact, it is shown in~\cite{AhlbergTassionTeixeira2018} that this equality holds for all radius distributions in \(d=2\) and in~\cite{DembinTassion2022} that \(\widehat{\beta}_c(\gamma,0,\infty)= \beta_c(\gamma,0,\infty)\) in all dimensions \(d\geq 2\) for all but at most countably many \(\gamma\). Hence, in these situations the above upper bound for the Euclidean diameter's tail distributions holds  for all \(\beta<\beta_c(\gamma,0,\infty)\).  

 The above shows that we can effectively think of attaching a growing radius of length \(m^{1/d}\) to the origin in the Boolean model at no extra cost and still observe subcritical behaviour. However, as we can see from our proofs below, this argument is no longer possible in the soft Boolean model. The reason is that there are simply too many long edges in the graph. Assigning a radius of length \(m^{1/d}\) to the origin combined with the long-range effects dramatically increases (by order \(m\)) the number of long edges contained in the component of the origin. We are therefore restricted to work with the original radius of the origin and its original component at all times when dealing with similar events as outlined above in~\ref{stepC} and~\ref{stepD}. {More precisely, as already pointed out in the discussion below Theorem~\ref{thm:MetaMain} and shown in~\cite{jacobJahLu2024}, the number of long edges that alone cross an annulus is determined by the vertex-boundary exponent \(\zeta\), introduced in~\cite{KomjathyLapinskasLengler2021} and the renormalisation can only be applied if \(\zeta<0\). However, \(\zeta<0\) is only satisfied in a smaller parameter regime and there are regimes with \(\widehat{\beta}_c(\gamma,0,\delta)=0\) but \(\beta_c(\gamma,0,\delta)>0\) and that \(\zeta\) does not give the right exponent even if it can be used.} Summarising, for the soft Boolean model the situation of critical intensities is less clear.  
	
	In order to go one step further, let us consider another critical intensity. To prove the existence of a subcritical percolation phase for the age-dependent random connection model \(\cG^{\beta,\gamma,1-\gamma,\delta}(\xi_o)\) of~\cite{GGLM2019} and models dominated by it, the authors in~\cite{GLM2021} consider
	\[
		\beta_1(\gamma,\alpha,\delta):=\sup\Big\{\beta>0\colon \exists c>0 \text{ such that } \P_o\Big(\substack{\o \text{ starts a shortcut-free path in } \cG^{\beta,\gamma,\alpha,\delta}\text{ of length }n \\ \text{ whose end vertex has the smallest mark in the path}}\Big)\leq {\rm e}^{-c n}\Big\}.
	\]
	Here, a path \(P=(\x_0,\x_1,\dots,\x_n)\) is called \emph{shortcut-free} if it contains no shorter subpath \(Q\subset P\) also connecting \(\x_0\) and \(\x_n\). The restriction to shortcut-free paths is possible as the existence of an infinite shortcut-free path is equivalent to the existence of an infinite path since all degrees are finite. However note that in general the existence of a path of length \(n\) does not necessarily imply the existence of a shortcut-free paths of the same length. The idea behind the definition is to make use of the fact that the vertices with highest degree are those with smallest marks which can be seen as the skeleton vertices from \ref{stepA}. To build an infinite path, one may want to use vertices with smaller and smaller marks to find sufficiently many vertices that have not been visited yet to continue the path. However, such a path contains infinitely many finite subpath ending in the vertex with smallest vertex mark. For \(\beta<\beta_1(\gamma,\alpha,\delta)\) only finitely many such subpaths exist by the Borel--Cantelli Lemma. Hence, each potentially infinite path must have marks that are bounded from zero. The latter ultimately implies 
	\[
		\inf\{u_x\colon \x\in\mathscr{C}_{\beta,\gamma,\alpha,\delta}\}>0,
	\]
		which is equivalent to \(\sharp\mathscr{C}_{\beta,\gamma,\alpha,\delta}<\infty\) {since, by stationarity, an infinite subset of the graph cannot avoid a whole interval in the mark space of positive mass.} For the age-dependent random connection model it is then shown in~\cite{GLM2021} that \(\beta_1(\gamma,1-\gamma,\delta)>0\) if \(1/\delta<1/\gamma-1)\) and \(\beta_1(\gamma,1-\gamma,\delta)=0\) if \(1/\delta>1/\gamma-1)\) and the same applies to the soft Boolean model. From our proofs in Lemma~\ref{lem:MainThmLGupperH} and~\ref{lem:PathBoundUpperInterGamma} below it is easy to see that
	\[
		\beta_1(\gamma,0,\delta) \geq 
		\begin{cases}
		  \frac{1}{2^{d\delta+3}+1}\cdot\tfrac{\delta-1}{\omega_d \delta}\big(1-\gamma\tfrac{\delta+1}{\delta}\big), & \text{if }\tfrac{1}{\gamma}\leq 2, \\
		  \tfrac{\delta-1}{\omega_d\delta}\big(\tfrac{1}{1-2\gamma}+\tfrac{1}{1-\gamma}\big)^{-1} \vee  \frac{1}{2^{d\delta+3}+1}\cdot\tfrac{\delta-1}{\omega_d \delta}\big(1-\gamma\tfrac{\delta+1}{\delta}\big), & \text{if } \tfrac{1}{\gamma}> 2,
		  \end{cases}
	\]
	matching the results of~\cite{GLM2021} when their proof is specified to the soft Boolean model. {Note that these are the same bounds observed in~\eqref{eq:boundsForTailBeta} for \(\widetilde{\beta}\). In~\cite{GouereTheret2019}, the authors study the standard Boolean model with ball volumes that have finite second moment (i.e., \(1/\gamma>2\)) and show that subcriticality, i.e., \(\beta<\beta_c(\gamma,0,\infty)\), is equivalent to \(\P_o(\0 \text{ starts paths of length } n)<{\rm e}^{-cn}\). Obviously, the defining event of \(\beta_1\) looks similar but the considered paths are subject to stronger restrictions in order to deal with the stronger influence of vertices in regimes where there is no second moment for the ball volume and the long-range effects.}
 
 	On a technical level, the necessity of \(\beta\) being small enters our proof in \ref{stepB}, \ref{stepC}, and \ref{stepD} when potentially arbitrarily long paths have to be considered. Based on the nature of our proof, one could suspect that the statement of our main theorem holds for all \(\beta<\beta_1(\gamma,0,\delta)\). Particularly in \ref{stepB} only paths ending in its most powerful vertex are considered. Unfortunately, bounding~\eqref{eq:MainThmLGupperHToBound2} requires slightly more than bounding the number of paths occurring in the definition of \(\beta_1(\gamma,0,\delta)\). However,~\eqref{eq:MainThmLGupperHToBound2} is the result of a moment bound. Additionally, while generally the decay of the probability of existence of shortcut-free paths of length \(n\) gives no information about the decay of the probability of a general path, in order to deduce the probability of \(\{\M_\beta>m\}\) one can always reduce to shortcut-free paths. We therefore believe the following for \(\widetilde{\beta}_c\) being the largest intensity, for which the tail bounds of Theorem~\ref{thm:Main} apply. 
		
	\begin{conjecture}\label{conj:beta1}
		 \(\widetilde{\beta}_c \geq  \beta_1(\gamma,0,\delta)\).
	\end{conjecture}	
	
  Further, the restriction to sufficiently small \(\beta\) only enters the upper bound. The lower bounds however do not depend on \(\beta\) and it seems to us that there is no particular reason indicating that increasing \(\beta\) within the subcritical regime would change the optimal strategy leading to the lower bounds. 
  We therefore conjecture:
	
	\begin{conjecture}\label{conj:beta2}
		 \(\widetilde{\beta}_c=\beta_c(\gamma,0,\delta)\).
	\end{conjecture}

	\subsubsection{The logarithmic term in the upper bound.}\label{subsec:discussion_log} 
	As already discussed above, the logarithmic correction in Part~(iii) of Theorem~\ref{thm:Main} is a result of \ref{stepD} where edges of length of order \(m/\log m\) must be controlled. However, we believe that this error term is a result of our method and should not appear in the case of Pareto-distributed radii. In fact it seems reasonable to conjecture the following. 
	\begin{conjecture}\label{conj:errorTerm}
		For all \(\beta<\widetilde{\beta}\), \(\delta>1\), and \(1/\delta<1/\gamma-1<\delta\), there exists \(C>0\) such that, for all \(m>1\), we have that
		\[
			\P_o(\M_{\beta}>m\big)\leq C m^{-\tfrac{\delta-1}{\delta}(\tfrac{1}{\gamma}-1)}.
		\]
	\end{conjecture}  
 	The other reason why we believe in Conjecture~\ref{conj:errorTerm} is that we can show that the error term indeed does not occur if \(1/\gamma>2\) and the degree distribution has finite variance, at least in another low-intensity regime. This is due to that fact that, in this case, the expected number of paths decays exponentially, cf.~Lemma~\ref{lem:WholePathBound}, which allows to replace the three steps, \ref{stepB}, \ref{stepC}, and \ref{stepD} with a more direct bound similarly to the proof of Part~(i). 
    
	\begin{prop}\label{thm:UpperBoundSmallGamma}
		Let \(\delta>1\), \(1/\gamma>2\) and \(0<\beta<\tfrac{\delta-1}{\omega_d\delta}\big(\tfrac{1}{1-2\gamma}+\tfrac{1}{1-\gamma}\big)^{-1}\). Then, there exists \(C>0\) such that, for all $m>1$,
		\[
			\P_o(\M_{\beta}>m) \leq C m^{-\tfrac{\delta-1}{\delta}(\tfrac{1}{\gamma}-1)}.
		\]
	\end{prop} 
	
	\subsubsection{Cardinality of the  cluster of the origin}
	In Section~\ref{sec:NumberPoints} we derive results for the cardinality of the component for the \(\gamma<1/2\) regime and yet another threshold for \(\beta\). The new threshold is on the one hand a result of the coupling with scale-free percolation coinciding with \(\alpha=\gamma\) in~\eqref{eq:InterpolKern}. On the other hand, we use a coupling of the  component of the origin with a multi-type branching process that requires a finite second moment of the degree distribution and a small enough \(\beta\) to control the exponential growth of integration constants associated with the second moment. While these branching-process arguments work very well for non-spatial random graphs, they have the disadvantage of not seeing the spatial clustering. Hence, we cannot expect to get precise results in a spatial setting where the effect of clustering is highly relevant. The relevance of clustering in the soft Boolean model can be seen for example in the fact that a subcritical phase still exists for parameter regimes with degree distributions with infinite variance (\(1/\gamma>2\)) while non-spatial models are always robust in the infinite-variance regime~\cite{vdH2024}. This is due to the effect that in the associated branching process, the offspring distribution follows the size-biased degree distribution and only has finite first moment, if the original distribution has finite second moment. For scale-free percolation the effect of clustering seems to be less relevant as many changes of behaviour happen again at \(1/\gamma=2\) when the second moment of the degree distribution becomes infinite. Still we do not see any reason why the established tail in Theorem~\ref{thm:SFP} should not be valid in the whole subcritical regime.
	\begin{conjecture}Consider $\beta_c(\gamma,\alpha,\delta)$ as defined in~\eqref{eq:betacr}.
		\begin{enumerate}[(i)]
			\item Consider the soft Boolean model \(\cG^{\beta,\gamma,0,\delta}\) for \(\delta>1\), \(1/\delta<1/\gamma-1\), and \(\beta<\beta_c(\gamma,0,\delta)\). Then, there exists a constant \(C\in(0,\infty)\) such that, for all \(m>1\),
			\[
				\P_o(\scN_{\beta,\gamma,0,\delta}\geq m) \leq C m^{1-1/\gamma}.
			\]
			\item Consider scale-free percolation \(\cG^{\beta,\gamma,\gamma,\delta}\) for \(\delta>1\), \(1/\gamma\geq 2\), and \(\beta<\beta_c(\gamma,\gamma,\delta)\). Then, there exists a constant \(C\in(0,\infty)\) such that, for all \(m>1\),
			\[
				\P_o(\scN_{\beta,\gamma,\gamma,\delta}\geq m) \leq C m^{1-1/\gamma}.
			\]
		\end{enumerate}
	\end{conjecture}   
	
	Let us further comment on Part~(i) of the conjecture. As already mentioned above,
	to achieve a cardinality of order \(m\) with a single vertex, a neighbour with radius of order \(m^{1/d}\) is required just like in the standard Boolean model, which is why the tails (of the lower bound) coincide in both models. To prove a matching upper bound for the soft model, one can easily adapt \ref{stepB} to vertices of that strength in the component of the origin. Further, our result for the Euclidean diameter can be used to derive that the whole component is contained in a ball of volume \(m^{\delta/(\delta-1)}\), with an error probability of the right order. Now, one can think of the  component of the origin in the soft model as a collection of classical Boolean clusters, connected with each other via long-range edges. On the event that no powerful vertex is present in the component of the origin, one can use Gou\'{e}r\'{e}'s method to deduce that no Boolean cluster in the considered ball has cardinality larger than \(m/\log(m)\) implying that the the  component of the origin decomposes in at least \(\log m\) many clusters when the long-range edges are removed. Additionally, we can adapt the bound on long paths of \ref{stepD} to observe with the right error probability that the component's depth is at most \(\log m\). Hence, it remains to prove that it is unlikely enough to connect at least \(\log m\) small Boolean clusters with long-range edges without creating shortest paths longer than \(\log m\). Since the tail of the cardinality in long-range percolation decays exponentially fast one would get the desired result if one thought of the Boolean clusters as nodes of an independent long-range percolation model. Unfortunately, this ignores correlations between the edges and we could not find a convincing way yet to control these. We do however believe that it should be possible.

\section{Proofs}\label{sec:proofs}
In this section we present the proofs of our results. We will use the common notation \(f\sim g\) for two positive functions satisfying \(f(x)/g(x)\to 1\) as \(x\to 0\), and \(f=o(g)\), if \(f(x)/g(x)\to 0\) as well as \(g\asymp f\), if \(f/g\) is bounded from zero and infinity. We further denote by \(B(x,r)\) the ball of radius \(r\) centred in \(x\). If \(x=o\), we simply write \(B(r):=B(o,r)\) for the ball around the origin and by \(\omega_d\) the volume of \(B(1)\). For a Borel set \(B\subset\R^d\), we also write \(B^c=\R^d\setminus B\), and we use the same notation for the complement of an event as usual. Finally, we denote by \(\sharp A\) the number of elements in an at most countable set $A$.

\subsection{Proof of the lower bounds in Theorem~\ref{thm:Main}} \label{sec:MainThmLower}
We formally prove the two strategies to connect the origin to a vertex at a large distance in this section outlined above. The first strategy is to use a single long edge incident to the origin. This strategy dominates on the long-range percolation model. The second strategy dominates on the classical Boolean model case where a powerful intermediate vertex connected to both the origin and the distant vertex is used. Both strategies combined give the lower bounds of our theorem. 

We start with the proof of the `long-range percolation' case when only one long edge is used giving the lower bound in Part~(i). 
\begin{prop}\label{lem:MainThmSG}
		Let \(\beta>0\), \(\delta>1\), and \(\gamma\in(0,1)\) and consider \(\cG^\beta_\0=\cG^{\beta,\gamma,0,\delta}(\xi_\0)\). Then, for all \(m>\beta\), 
				\[
					\P_o\big(\exists \x=(x,u_x)\colon |x|^d>m\text{ and }\x\sim\0 \text{ in }\cG^\beta_\o\big)\geq 1-\exp\big(-\tfrac{\beta^\delta \omega_d}{\delta-1}m^{1-\delta}\big)\sim \tfrac{\beta^\delta \omega_d}{\delta-1}m^{1-\delta}.
				\]
	\end{prop}
\begin{proof}		
	Since a version of the random connection model is given by choosing \(\gamma=\alpha=0\) in the interpolation kernel~\eqref{eq:InterpolKern} and \(g_{0,0}(s,t)= 1\geq g_{\gamma,0}(s,t)\), each edge that is present in a realisation of \(\cG^{\beta,0,0,\delta}(\xi_\0)\) is also present in \(\cG^\beta_\0\) by the construction rule~\eqref{eq:edgeSet}. (Note that the profile function \(\rho\) is decreasing and therefore a smaller value of \(g_{\gamma,\alpha}(U_i,U_j)\) increases the probability of having an edge.) Hence,
		\begin{equation*}
			\begin{aligned}
				\P_o\big(\exists \x=(x,u_x)\colon |x|^d>m\text{ and }\x\sim\0 \text{ in }\cG^\beta_\o\big) & \geq \P_o\big(\exists \x=(x,u_x)\in\cX\colon  |x|^d>m\text{ and }\x\sim\0 \text{ in }\cG_\o^{\beta,0,0,\delta}\big).
			\end{aligned}
		\end{equation*}
		Since the number of the  neighbours of the origin in \(\cG^{\beta,0,0,\delta}(\xi_\0)\) at distance at least \(m^{1/d}\) is Poisson distributed with parameter
		\begin{equation*}
			\begin{aligned}
				\int\limits_{|x|^d>m}\d x \ \rho(\beta^{-1}|x|^d)=\beta^\delta \int\limits_{|x|^d>m} \d x \ |x|^{-d\delta} = \tfrac{\beta^\delta \omega_d}{\delta-1} m^{1-\delta},
			\end{aligned}
		\end{equation*}
		assuming \(m>\beta\) in the first equality, we deduce
		\begin{equation*}
			\begin{aligned}
				\P_o(\M_\beta>m) & \geq 1-\exp\big(-\tfrac{\beta^\delta \omega_d}{\delta-1} m^{1-\delta}\big),
			\end{aligned}
		\end{equation*} 
		yielding the desired lower bound. 
	\end{proof}
	
	While in long-range percolation the order of the longest edge determines the order of the Euclidean diameter, this order is determined by the largest radius of a vertex connected to the origin in the classical Boolean model~\cite{Gouere08} {as explained below Theorem~\ref{thm:Main}.
	The following proposition makes the given heuristic precise and shows that a vertex requires a ball of volume no less than \(m^{(\delta-1)/\delta}\) in order to connect to distance \(m\), and that the ball volume in the neighbourhood of the root follows the size-biased volume distribution. For computational convenience, we restrict to vertices with large volume but no larger than \(m\). Recall that the associated volume of vertex \(\x\) is given as \(u_x^{-\gamma}\), and write 
	\begin{equation*}
		\zeta := (\delta-1)/\delta \quad \text{ and } \quad s_m:= m^{\zeta}
	\end{equation*}
	in the following. Note that the exponent in Part~(iii) of Theorem~\ref{thm:Main} can be written as \(-(1-\gamma)\zeta/\gamma\).}
	
	\begin{prop}\label{lem:MainThmLGLower}
	Let \(\beta>0\), \(\delta>1\), and \(\gamma>0\) and consider \(\cG_o^\beta=\cG^{\beta,\gamma,0,\delta}(\xi_\0)\). Then, there exists \(M>0\) such that, for all \(m\geq M\), we have
		\begin{equation*} 
			\begin{aligned}
				\P_o\big(\exists \x,\y\colon  |x|^d<m, s_m\leq u_x^{-\gamma}\leq m, m<|y|^d<2 m, u_y^{-\gamma}<2^{\gamma} \text{ and }\0\sim\x\sim\y\big) & \geq 1-{\rm e}^{- c_5 m^{-(1-\gamma)\zeta/\gamma}} \\ &\sim c_5 m^{-(1-\gamma)\zeta/\gamma},
			\end{aligned}
		\end{equation*}
	where \(c_5\) is given below in~\eqref{eq:MainThmLGLowerConst}. 
	\end{prop}
	\begin{proof}
		Since the considered event is monotone under the law of \(u\mapsto \P_{(o,u)}\) in the sense that the smaller the mark, thus the larger the volume, of the origin the more connections it forms and hence the likelier the occurrence of the considered event, and the fact that we aim for a lower bound, we work under \(\P_{(o,1)}\) in this proof. That is, we set the  mark of the origin to be \(1\). Let now $\mathscr{Y}\subset \cX$ be the set of all vertices $(y,u_y)$ with \(m < |y|^d< 2 m\) and \(u_y^{-\gamma}<2^\gamma\Leftrightarrow u<1/2\). Then, $\E[ \sharp \mathscr Y] = m \omega_d/2$ and a standard Poisson tail bound gives the existence of a constant \(\tilde c>0\) such that
  		\[ 
  			\P(\sharp\mathscr{Y} \geq m \omega_d/4) \geq 1- {\rm e}^{-\tilde c m \omega_d }.
  		\]
 		Now, consider a vertex \(\x = (x,u_x)  \in \cX \) with \(|x|^d < m\) and {associated volume \(s_m\leq u_x^{-\gamma}\leq m\).}  
 		Then, for $m$ sufficiently large, the probability that this vertex \(\x\) is connected to a particular vertex \(\y \in \mathscr{Y}\) is, by the monotonicity of $\rho$, lower bounded by
  		\[ 
  			\rho (\beta^{-1} u_x^\gamma |x - y|^d )  \geq \rho (\beta^{-1} u_x^\gamma 3m) \geq (\tfrac{\beta}{3})^\delta u_x^{-\gamma\delta} m^{-\delta}\geq (\tfrac{\beta}{3})^\delta m^{-1},  
  		\]
  		using the definition of \(s_m\). Thus, on the event 
  		\(\{\sharp \mathscr{Y} \geq m\omega_d/4\}\), the number of vertices \(\y\in\mathscr{Y}\) that are connected to \(\x\) is bounded from below by a Binomial random variable with \(m\omega_d/4\) trials and success probability \((\beta/3)^\delta m^{-1}\). Therefore, by Poisson approximation, we infer for sufficiently large \(m\) that the probability of existence of at least one such \(\y\) is no smaller than 
  		\[
  			1-\exp\big(-\beta^\delta \omega_d/4^{\delta+1} \big)=:c.
  		\] 
  		{This lower bound does not depend on the properties of the vertices in \(\mathscr{Y}\) nor on \(\x\) itself. Further, the neighbours of \((o,1)\) located in \(B(m^{1/d})\) with associated volume in \([s_m,m]\) form a Poisson process of intensity \(\rho(\beta^{-1}u_x^\gamma |x|^d)\d u_x \, \d x\), which is independent of \(\mathscr{Y}\). Therefore, the set of these neighbours that also have a neighbour in \(\mathscr{Y}\) stochastically dominate the neighbourhood Poisson process, where each vertex is independently retained with probability \(c\) and removed otherwise. Recalling that the ball volume associated to \(\x\) is given as \(u_x^{-\gamma}\) for \(u_x\) uniformly distributed, the expected number of vertices contained in this Poisson point process is given by
		\begin{equation*}
			\begin{aligned}
				c\int\limits_{m^{-1/\gamma}}^{s_m^{-1/\gamma}} \d u_x \int\limits_{|x|^d<m} \d x \, \rho(\beta u_x^{\gamma}|x|^d) 
				&
					\geq c \beta \int\limits_{m^{-1/\gamma+\varepsilon}}^{s_m^{-1/\gamma}} \d u \, u_x^{-\gamma}\Big( \int\limits_{|x|^d<u_x^\gamma m}\d x \, \rho(|x|^d)\Big)
				\\ &
					\geq \tfrac{c \beta \omega_d \delta}{2(1-\gamma)(\delta-1)}s_m^{(1-\gamma)/\gamma},
			\end{aligned}
		\end{equation*}	
		for any \(0<\varepsilon<1/(\gamma\delta)\) and large enough \(m\), using \(\int\limits_{|x^d|<u^\gamma m} \rho(|x|^d)\to \omega_d\delta/(\delta-1)\), as \(u^\gamma m\to\infty\), by choice of \(s_m\) and \(\varepsilon\). As the dominated Poisson process is non-empty with probability \(1-\exp(- c_5 m^{(1-\gamma)\zeta/\gamma})\)}, where
		\begin{equation}\label{eq:MainThmLGLowerConst}
			c_5 = \tfrac{c \beta \omega_d \delta}{2(1-\gamma)(\delta-1)}\big(1-\exp(-\tfrac{\beta^\delta \omega_d}{4^{\delta+1}})\big),
		\end{equation}	
        the proof is finished.
	\end{proof}
		
	\begin{proof}[Proof of the lower bounds in Theorem~\ref{thm:Main}.]
			Since {\(1-\delta\leq -(1-\gamma)\zeta/\gamma\) if and only if \(\delta\leq 1/\gamma-1\)}, the lower bound in Part~(i) is a consequence of Proposition~\ref{lem:MainThmSG}, the lower bound in Part~(iii) is a consequence of Proposition~\ref{lem:MainThmLGLower}, and the lower bound in Part~(ii) is a consequence of the combination of both, where the appearing constant is given by 
			\begin{equation} \label{eq:MainThmBClowerConst}
				c_3:= c_1 \vee c_5.
			\end{equation}
			This concludes the proof of the lower bounds.
	\end{proof}
	
\subsection{Proof of the upper bounds in Theorem~\ref{thm:Main}}\label{sec:MainThmLGupper}
	In this section, we prove the upper bounds of our main theorem. To this end, we show that no strategy of connecting the origin to some vertex at distance \(m^{1/d}\) works better in probability than the respective strategy in the lower bound. In order to do so, we rely on fine bounds for the probability that carefully chosen paths exist which we compute by using a first-moment method. Note that if \(1/\gamma<2\), the degree distribution has infinite second moment and already the expected number of paths of length two starting at the origin is infinite. Hence, we cannot apply direct moment bounds in the whole parameter regime. Instead, we decompose each considered path in its `powerful' vertices, the \emph{skeleton} of the path, and the remaining `weak' vertices. This concept was first introduced in~\cite{GLM2021} and gives a powerful tool to control the combinatorics of our path counts. Then, we combine this with a BK inequality for independent vertex-edge markings of~\cite{HvdHLM20} to reduce the probability of a path existing to a moment bound on the skeleton vertices that form a path themselves. This forms \ref{stepA} of our proof.

	\subsubsection{The skeleton of a path} \label{sec:MainThmLGSkeleton}
	The concept of a skeleton of a path for weight-dependent random connection models was introduced in~\cite{GLM2021} and is based on a decomposition of a \emph{shortcut-free} path into its powerful vertices, the \emph{skeleton}, and connectors which are weaker vertices that connect the skeleton vertices with each other. Recall that a path \(P=(\x_0,\x_1,\dots,\x_n)\) is called \emph{shortcut free} if it contains no shorter subpath \(Q\subset P\) also connecting \(\x_0\) and \(\x_n\). From now on, each path is considered to be shortcut free. 
	
	The idea is now the following: To have a long path it is important to have significantly powerful vertices which are those with \emph{smallest marks} since a small mark corresponds to a large inverse weight and therefore large degree. Let \(P=(\x_0,\x_1,\dots,\x_n)\) be a path of length \(n\). We call \(\x_i=(x_i,u_i)\) for \(i\not\in\{0,n\}\) a \emph{local maximum} if \(u_i>u_{i-1}\) and \(u_i>u_{i+1}\). We now construct a new path without any local maxima, see Figure~\ref{fig:Skeleton}. 
		\begin{figure}
		\begin{center}
			\begin{tikzpicture}[scale=0.35, every node/.style={scale=0.3}]
				\node (Z) at (-2.5,8.5)[circle, draw,scale=1.5]{1};
				\draw[->] (-1,-0.5) -- (-1, 8)
					node[left,scale=2] {$u$};
				\node (A) at (0,5)[circle, fill=black, label ={}]{};	
    			\node (B) at (2.5,3)[circle, fill = black, label={}] {};
    			\node (D) at (7.5,7)[circle, draw, label = {}] {};
    			\node (E) at (10, 6)[circle, draw, label={}] {};
    			\node (F) at (12.5,1) [circle, fill= black, label={} ] {};
    			\node (G) at (15,4.5) [circle, draw, label={}]{};
    			\node (H) at (17.5,2.5)[circle, fill=black, label={}]{};
    			\draw (A) to (B);
				\draw (B) to (D);
				\draw (D) to (E);
				\draw (E) to (F);
				\draw (F) to (G);
				\draw (G) to (H);
			\end{tikzpicture}
			\hspace{1 cm}
			\begin{tikzpicture}[scale=0.35, every node/.style={scale=0.3}]
				\node (Z) at (-2.5,8.5)[circle, draw,scale=1.5]{2};
				\draw[->] (-1,-0.5) -- (-1, 8)
					node[left,scale=2] {$u$};
				\node (A) at (0,5)[circle, fill=black, label ={}]{};
    			\node (B) at (2.5,3)[circle, fill = black, label={}] {};
    			\node (D) at (7.5,7)[circle, draw, label = {}, dotted] {};
    			\node (E) at (10, 6)[circle, draw, label={}] {};
    			\node (F) at (12.5,1) [circle, fill= black, label={} ] {};
    			\node (G) at (15,4.5) [circle, draw, label={}]{};
    			\node (H) at (17.5,2.5)[circle, fill=black, label={}]{};
    			\draw (A) to (B);
				\draw[dotted] (B) to (D);
				\draw[dotted] (D) to (E);
				\draw (B) to (E);
				\draw (E) to (F);
				\draw[] (F) to (G);
				\draw[] (G) to (H);
			\end{tikzpicture}
			\begin{tikzpicture}[scale=0.35, every node/.style={scale=0.3}]
				\node (Z) at (-2.5,8.5)[circle, draw,scale=1.5]{3};
				\draw[->] (-1,-0.5) -- (-1, 8)
					node[left,scale=2] {$u$};
				\node (A) at (0,5)[circle, fill=black, label ={}]{};
    			\node (B) at (2.5,3)[circle, fill = black, label={}] {};
    			\node (D) at (7.5,7)[label = {}] {};
    			\node (E) at (10, 6)[circle, draw, label={}, dotted] {};
    			\node (F) at (12.5,1) [circle, fill= black, label={} ] {};
    			\node (G) at (15,4.5) [circle, draw, label ={}]{};
    			\node (H) at (17.5,2.5)[circle, fill=black, label={}]{};
    			\draw (A) to (B);
				\draw (B) to (F);
				\draw[dotted] (B) to (E);
				\draw[dotted] (E) to (F);
				\draw (B) to (F);
				\draw[] (F) to (G);
				\draw[] (G) to (H);
			\end{tikzpicture}
			\hspace{1 cm}
			\begin{tikzpicture}[scale=0.35, every node/.style={scale=0.3}]
				\node (Z) at (-2.5,8.5)[circle, draw,scale=1.5]{4};
				\draw[->] (-1,-0.5) -- (-1, 8)
					node[left,scale=2] {$u$};
				\node (A) at (0,5)[circle, fill=black, label ={}]{};
    			\node (B) at (2.5,3)[circle, fill = black, label={}] {};
    			\node (D) at (7.5,7)[label = {}] {};
    			\node (E) at (10, 6)[label={}] {};
    			\node (F) at (12.5,1) [circle, fill= black, label={} ] {};
    			\node (G) at (15,4.5) [circle, draw, dotted, label={}]{};
    			\node (H) at (17.5,2.5)[circle, fill=black, label={}]{};
    			\draw (A) to (B);
				\draw[dotted] (F) to (G);
				\draw[dotted] (G) to (H);
				\draw (B) to (F);
				\draw (F) to (H);
			\end{tikzpicture}
			\caption{A path where the mark of a vertex is denoted on the $u$-axis and the spatial location of the vertices is not shown. The vertices of the skeleton are in black. We successively remove all local maxima, starting with the largest mark vertex, and replace them by direct edges until the path, only containing the skeleton vertices, is left.}
		\label{fig:Skeleton}
		\end{center}
	\end{figure}
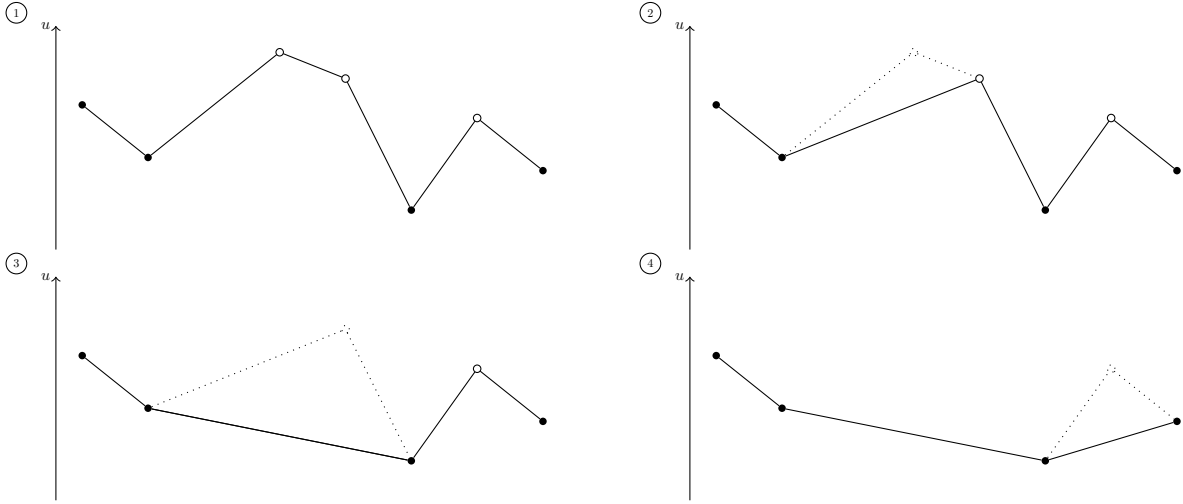	
	For this, we first take the local maximum in \(P\) with greatest vertex mark, remove it from \(P\) and connect its former neighbours by an edge. In the resulting path, we take the new local maximum of greatest vertex mark, remove it, and connect its former neighbours, repeating until there is no local maximum left in \(P\). We call the remaining vertices the \emph{skeleton} of the path. Note that start and end vertex of a path can never be a local maximum and are therefore always part of the skeleton. Further, the constructed paths of skeleton vertices is not necessarily an actual path of the graph. In particular, the skeleton vertices of a shortcut-free path cannot themselves form a path unless the vertices of the path already form a skeleton. {However, if it was a path, then, by construction, the vertices on that path have strictly decreasing vertex marks until the strongest vertex is reached and strictly increasing vertex marks afterwards.}
	
	We have decomposed the path in two sets of vertices: the skeleton vertices and the {weak vertices in between that become local maxima on some point in the procedure and} that connect two consecutive skeleton vertices. From now on, we refer to the latter as \emph{connectors}. In the following, we derive bounds for the probability that two given vertices are connected by a shortcut-free path consisting of connectors only. Afterwards, we state the BK inequality of~\cite{HvdHLM20} that can be used to decompose the whole path into its subpaths connecting two consecutive skeleton vertices to obtain bounds for the existence of paths.
	
	\paragraph{Connecting two powerful vertices.}
	In this paragraph, we consider two given vertices \(\x\) and \(\y\) and the probability that they are connected by a shortcut-free path in a weight-dependent random connection model \(\cG^{\beta,\gamma,\alpha,\delta}(\xi_{\x,\y})\) of length \(n\) consisting of connectors only. That is, the path's skeleton is given by \(\x\) and \(\y\) only. Let us denote this event by \(\{\x\xleftrightarrow[\x,\y]{n}\y \text{ in } \cG^{\beta,\gamma,\alpha,\delta}(\xi_{\x,\y})\}\). Since by construction \(\x\) and \(\y\) are connected by an edge with probability 
	\[
		\rho\big(\beta^{-1}g_{\gamma,\alpha}(u_x,u_y)|x-y|^d\big) = 
		\begin{cases}
			1, &\text{if } |x-y|^d \leq \beta g_{\gamma,\alpha}(u_x,u_y)^{-1}, \\
			\beta^\delta g_{\gamma,\alpha}(u_x,u_y)^{-\delta}|x-y|^{-d\delta}, & \text{if } |x-y|^d > \beta g_{\gamma,\alpha}(u_x,u_y)^{-1},	
		\end{cases}
	\] 
	we focus on given vertices \(\x\) and \(\y\) at distance \(|x-y|^d > \beta g_{\gamma,\alpha}(u_x,u_y)^{-1}\) to fulfil the shortcut-free property. The next lemma is key in this section and it is a combination of~\cite[Lemma~2.2 and 2.3]{GLM2021}. 
	
	\begin{lemma}[\(n\)-connection Lemma, \cite{GLM2021}] \label{lem:nConnection}
		Let \(\beta>0\), \(\gamma\in(0,1)\), \(\alpha\in[0,2-{\gamma})\), and \(\delta>1\). Let further \(\x=(x,u_x)\) and \(\y=(y,u_y)\) be two vertices satisfying the distance condition \(|x-y|^d > \beta g_{\gamma,\alpha}(u_x,u_y)^{-1}\). Assume that there is a constant \(C>0\), {depending only on \(\gamma,\alpha\) and \(d\),} such that
		\[
			\E_{\x,\y}\big[\sharp\{\z=(z,u_z)\colon u_z>u_x\vee u_y, \text{ and }\x\sim\z\sim\y \text{ in }\cG^{\beta,\gamma,\alpha,\delta}(\xi_{\x,\y})\}\big]\leq (\beta C) \rho\big(\beta^{-1}g_{\gamma,\alpha}(u_x,u_y)|x-y|^d\big).	
		\]
		Then, for all \(n\in\N\), we have
		\[
			\P_{\x,\y}\big(\x\xleftrightarrow[\x,\y]{n}\y\text{ in }\cG^{\beta,\gamma,\alpha,\delta}(\xi_{\x,\y})\big)\leq (4\beta C)^{n-1}\rho\big(\beta^{-1}g_{\gamma,\alpha}(u_x,u_y)|x-y|^d\big).
		\]
	\end{lemma}

	The idea behind the lemma is the following. One can represent the inner path consisting of all vertices excluding \(\x\) and \(\y\) as a binary tree which is labelled by the Poisson points of the graph such that each child has a greater vertex mark than its parent. Here, a binary tree is a tree where each vertex has either no child, a left child, a right child, or a left and a right child. The underlying  binary tree encodes important structural information of the path. Namely, it encodes the order of the local maxima where the precise vertex marks order is then given by the labelling. The proof works by induction. If \(n=1\) there is nothing to show. For \(n=2\) the binary tree consists only of a root vertex which is the intermediate connector and the claim follows from the assumptions. Now, for any \(n\geq 3\), fix a binary tree of \(n-1\) vertices (recall that start and end vertex are not represented in the tree) and label it with Poisson points such that each child has mark larger than its parent and the distance condition for shortcut-free paths is fulfilled. This then represents a paths of length \(n\). Pick the leaf with largest vertex mark in the tree. The vertex represented by this leaf is by necessity a connector of two vertices with smaller marks in the path. Therefore, we can apply the assumption and bound the expected number of possible labels for this leaf by \(\beta C\) times the probability that there is an edge between the vertices the leaf connects in the path. Since this yields a new path on \(n-1\) vertices represented by the tree of \(n-2\) vertices where the leaf was removed, the induction hypothesis applies and we infer that the expected number of possible labelings is bounded by \((\beta C)^{n-1}\rho\big(\beta^{-1}g_{\gamma,\alpha}(u_x,u_y)|x-y|^d\big)\). Now, the claim follows by applying a union bound over all unlabelled trees on \(n-1\) vertices and the fact that there are at most \(4^{n-1}\) such trees. For the details of the proof and particularly the tree representation, we refer the reader to~\cite{GLM2021}.
	
	{Key to an application of Lemma~\ref{lem:nConnection} is the assumption that the expected number of weak vertices connecting two given stronger vertices at a large distance scales no better than the probability of a direct connection via a single edge. If this is the case, the lemma tells us that the probability to connect the two vertices using many connectors scales even worse provided \(\beta\) is chosen small enough. Unfortunately, the soft Boolean model does not fulfil the assumption. On an intuitive level, this is easy to see. Indeed, for the connection of two given vertices \(\x\) and \(\y\) only the mark (i.e., ball volume) of the stronger vertex, say \(\y\), matters. Whether the vertex \(\x\) is rather weak, e.g., \(u_x^{-\gamma}=2\), or almost as strong as \(\y\), e.g., \(u_x^{-\gamma}=u_y^{-\gamma}/2\), has no effect on the connection probability, cf.\ connection rule~\eqref{eq:defSoftBool}. However, in the latter case, the expected neighbourhood of \(\x\) is much larger than in the first case, so that the expected intersection of both neighbourhoods becomes much larger as well so that two strong vertices may actually benefit from a connector strategy.}  
To tackle this issue and still make use of Lemma~\ref{lem:nConnection} and the skeleton strategy, we stochastically dominate the soft Boolean model by a graph that matches the assumption. To this end, consider the weight-dependent random connection model \(\cG^{\beta,\gamma,\gamma/\delta,\delta}(\xi)\) defined through its kernel
	\[
		g_{\gamma,\gamma/\delta}(s,t)=(s\wedge t)^\gamma (s\vee t)^{\gamma/\delta},
	\]
	which we also refer to as \emph{two-connection kernel}. 
	To lighten notation, we write \(\widehat{\cG}^{\beta}=\cG^{\beta,\gamma,\gamma/\delta,\delta}(\xi)\) in the following. The following lemma shows that \(\widehat{\cG}^\beta\) dominates \(\cG^\beta\) and, indeed, has the required \(n\)-connection property of Lemma~\ref{lem:nConnection} whenever the soft Boolean model has a subcritical phase.
	
	\begin{lemma}\label{lem:MainThmLGUpperDomModel}
		Let \(\beta>0\), \(1/\delta<1/\gamma-1<\infty\), and \(\delta>1\) and consider the soft Boolean model \(\cG^\beta=\cG^{\beta,\gamma,0,\delta}(\xi)\) and the graph \(\widehat{\cG}^\beta=\cG^{\beta,\gamma,\gamma/\delta,\delta}(\xi)\). Then
		\begin{enumerate}[(i)]
			\item the two edge sets satisfy \(E(\cG^{\beta})\subset E(\widehat{\cG}^\beta)\) almost surely and
			\item for two given vertices \(\x=(x,u_x)\) and \(\y=(y,u_y)\) at distance \(|x-y|^d > \beta g_{\gamma,\gamma/\delta}(u_x,u_y)^{-1}\) and all \(n\in\N\), we further have
			\[
				\P_{\x,\y}\big(\x\xleftrightarrow[\x,\y]{n}\y\text{ in }\widehat{\cG}^{\beta}_{\x,\y}\big)\leq (\beta \widehat{C})^{n-1}\rho\big(\beta^{-1}g_{\gamma,\gamma/\delta}(u_x,u_y)|x-y|^d\big),
			\]
		where \(\widehat{C}=\tfrac{2^{d\delta+3}\omega_d \delta^2}{(\delta-1)(\delta-\gamma(\delta+1))}\).
		\end{enumerate}
	\end{lemma}
	\begin{proof}
		Since \(\gamma/\delta>0\) and therefore \(g_{\gamma,0}(s,t)>g_{\gamma,\gamma/\delta}(s,t)\) and \(\rho\) is non increasing, each edge that is present in \(\cG^\beta\) is also present in \(\widehat{\cG}^\beta\) by construction~\eqref{eq:edgeSet}, proving (i). 
		
		To prove (ii), we assume without loss of generality \(u_x>u_y\), and calculate using Mecke's equation~\cite{LastPenrose2017}
		\begin{equation*}
			\begin{aligned}
			\E_{\x,\y} & \big[\sharp\{\z=(z,u_z)\in\cX\colon u_z>u_x\vee u_y, \text{ and }\x\sim\z\sim\y \text{ in }\widehat{\cG}^{\beta}_{\x,\y}\}\big]	\\ 
			&=\int\limits_{\R^d} \d z \int\limits_{u_x}^1 \d u_z \ \rho\big(\beta^{-1}u_y^\gamma u_z^{\nicefrac{\gamma}{\delta}}|y-z|^d\big)\rho\big(\beta^{-1} u_x^\gamma u_z^{\nicefrac{\gamma}{\delta}} |x-z|^d\big).
			\end{aligned}
		\end{equation*}
		Now, either \(|y-z|\geq \tfrac{1}{2}|x-y|\) or \(|x-z|\geq\tfrac{1}{2}|x-y|\). Splitting \(\R^d\) according to those cases yields
		\begin{equation*}
			\begin{aligned}
				 \int\limits_{\R^d} \d z \int\limits_{u_x}^1\d u_z & \ \rho(\beta^{-1}u_y^{\gamma}u_z^{\nicefrac{\gamma}{\delta}}|y-z|^d)\rho(\beta^{-1}u_x^{\gamma}u_z^{\nicefrac{\gamma}{\delta}}|x-z|^d) \\ 
				 & \leq  \int\limits_{u_x}^1 \d u_z \ \rho\big((2^d\beta)^{-1}u_y^{\gamma}u_z^{\nicefrac{\gamma}{\delta}}|y-x|^d\big)\int\limits_{\R^d} \ d z \ \rho(\beta^{-1}u_x^{\gamma}{u_z^{\nicefrac{\gamma}{\delta}}}|z|^d) 
                  \\
                  & \qquad + \int\limits_{u_x}^1\d u_z \ \rho((2^d\beta)^{-1}u_x^{\gamma}u_z^{\nicefrac{\gamma}{\delta}}|y-x|^d)\int\limits_{\R^d} \d z \ \rho(\beta^{-1}u_y^{\gamma}u_z^{\nicefrac{\gamma}{\delta}}|z|^d).
			\end{aligned}
		\end{equation*}
		For the first integral, we use the change of variables \(w=(\beta^{-1}u_x^\gamma u_z^{\gamma/\delta})^{1/d}z\) and the distance condition to deduce
		\begin{equation*}
			\begin{aligned}
				\int\limits_{u_x}^1  \d u_z \ \rho\big(&(2^d\beta)^{-1}u_y^{\gamma}u_z^{\nicefrac{\gamma}{\delta}}|y-x|^d\big)\int\limits_{\R^d} \ d z \ \rho(\beta^{-1}u_x^{\gamma}{u_z^{\nicefrac{\gamma}{\delta}}}|z|^d) \\
				& = 2^{d\delta} \beta^{1+\delta}u_x^{-\gamma}|x-y|^{-d\delta}u_y^{-\gamma\delta}\Big(\int\limits_{u_x}^1 \d u_z \  u_z^{-\gamma-\nicefrac{\gamma}{\delta}}\Big)\Big(\int\limits_{\R^d} \d w \ \rho(|w|^d)\Big) \\ 
				& \leq \beta \tfrac{2^{d\delta}\omega_d \delta^2}{(\delta-1)(\delta-\gamma(\delta+1))}\rho\big(\beta^{-1}g_{\gamma,\gamma/\delta}(u_x,u_y)|x-y|^d\big),
			\end{aligned}
		\end{equation*}
		where we have used \(1-\gamma-\gamma/\delta>0\), and \(\int \rho(|w|^d)\d w=\omega_d \delta / (\delta-1)\), together with the distance condition in the last step. A similar calculation for the second integral yields the same bound and summing both terms yields
		\[
			\E_{\x,\y}\big[\sharp\{\z=(z,u_z)\in\cX\colon u_z>u_x\vee u_y, \text{ and }\x\sim\z\sim\y \text{ in }\widehat{\cG}^{\beta}_{\x,\y}\}\big]\leq \tfrac{\beta\widehat{C}}{4} \rho\big(\beta^{-1}g_{\gamma,\gamma/\delta}(u_x,u_y)|x-y|^d\big).
		\]
		Hence, we can apply Lemma~\ref{lem:nConnection}, finishing the proof.
	\end{proof}
	
	\paragraph{BK inequality.}
	The last section was devoted to showing how a path whose skeleton consists of start and end vertex only can be reduced to a single edge in probability. Since we consider more general paths in the following, we need to decompose the whole path in its subpaths between two consecutive skeleton vertices to make use of the above results. To this end, we use a version of the BK inequality~\cite{BK85} for an independent vertex-edge marking of a Poisson point processes as in~\cite[Theorem 2.1]{HvdHLM20} that generalises the BK inequality for the classical Boolean model in~\cite{GuptaRao1999}. The application to our setting is described in detail in~\cite[p.~14]{HvdHLM20}. {In a nutshell, the BK inequality states that two \emph{increasing} events, i.e.\ events that become more probable if additional vertices are added to the graph, are less likely to occur on two disjoint subsets of the original graph than separately on two independent copies of the graph.}
	
	Let \(\x_0,\x_1,\dots,\x_k\) be given vertices with a vertex mark structure that  form a skeleton. That is, their marks are decreasing until they reach the vertex with minimum mark and only increasing afterwards. Let us denote by \(\{\x_0\xleftrightarrow[\x_0,\x_1,\dots,\x_k]{n}\x_k \text{ in } \widehat{\cG}^\beta_{\x_0,\x_1,\dots,\x_k}\}\) the event that \(\x_0\) and \(\x_k\) are connected by a (shortcut-free) path of length \(n\) and skeleton \(\x_0,\x_1,\dots,\x_k\) in \(\widehat{\cG}^\beta_{\x_o,\dots,\x_n}\), which is consistent with the previously used notation. 
 
    We start with the case \(k=2\) where the skeleton only consists of the three vertices \(\x_0,\x_1,\x_2\). We consider for \(n_1+n_2=n\) the two events \(\{\x_{i-1}\xleftrightarrow[\x_{i-1},\x_i]{n_i}\x_i \text{ in } \widehat{\cG}^\beta_{\x_{i-1},\x_i}\}\), for \(i=1,2\). We are interested in the event that there exists a path from \(\x_0\) to \(\x_2\) with intermediate skeleton vertex \(\x_1\) such that there are \(n_1\) connectors used between \(\x_0\) and \(\x_1\) and \(n_2\) connectors used between \(\x_1\) and \(\x_2\). Since all our paths are self avoiding and shortcut free, this event is the same as the \emph{disjoint occurrence} of the two events \(\x_{0}\xleftrightarrow[\x_{0},\x_1]{n_1}\x_1\), and \(\x_{1}\xleftrightarrow[\x_{1},\x_2]{n_2}\x_2\). Here, disjoint occurrence means that the two paths share no element of \(\xi\). That is, no Poisson vertex is used twice. Note here that \(\x_1\) is a given vertex and thus not a random element of \(\cX\). We denote the disjoint occurrence by \(\circ\). Hence, the event of interest is given by 
	\[
		\big\{\x_{0}\xleftrightarrow[\x_{0},\x_1]{n_1}\x_1 \text{ in }\widehat{\cG}^\beta_{\x_0,\x_1}\big\} \circ \big\{\x_{1}\xleftrightarrow[\x_{1},\x_2]{n_2}\x_2\text{ in }\widehat{\cG}^\beta_{\x_1,\x_2}\big\}.
	\]
	Further, the two events are \emph{increasing} in the following sense: Given any realisation \(\omega\) of, say, \(\xi_{\x_0,\x_1}\) such that the event \(\{\x_{0}\xleftrightarrow[\x_{1},\x_1]{n_1}\x_1 \text{ in }\widehat{\cG}^\beta_{\x_1,\x_2}\}\) occurs on \(\omega\), it also occurs on all realisations \(\omega'\) with \(\omega\subset\omega'\). That is, if there is such a path for some realisation, it will also be present if additional vertices are added to the graph. It is important to note that this notion of increasing refers to set inclusion only, i.e., we only may add additional vertices with their incident edges but we never add additional edges between already existing vertices as this may shorten the path of length \(n_1\) due to the shortcut-free property. The application of the BK inequality~\cite[Theorem 2.1]{HvdHLM20} as performed on~\cite[p.~14]{HvdHLM20} then yields
	\begin{equation*}
		\begin{aligned}
			\P_{\x_0,\x_1,\x_2} & \Big(\big\{\x_{0}\xleftrightarrow[\x_{0},\x_1]{n_1}\x_1 \text{ in }\widehat{\cG}^\beta_{\x_0,\x_1}\big\} \circ \big\{\x_{1}\xleftrightarrow[\x_{1},\x_2]{n_2}\x_2\text{ in }\widehat{\cG}^\beta_{\x_1,\x_2}\big\}\Big) \\
			&\leq \P_{\x_0,\x_1}\big(\x_{0}\xleftrightarrow[\x_{0},\x_1]{n_1}\x_1\text{ in }\widehat{\cG}^\beta_{\x_0,\x_1}\big)\P_{\x_1,\x_2}\big(\x_{1}\xleftrightarrow[\x_{1},\x_2]{n_2}\x_2\text{ in }\widehat{\cG}^\beta_{\x_1,\x_2}\big).
		\end{aligned}
	\end{equation*}
	Inductively, this applies to all \(k\geq 2\) and \(n_1+\dots+n_k=n\) and therefore
	\begin{equation}\label{eq:BK}
		\P_{\x_0,\dots,\x_k}\big(\x_0\xleftrightarrow[\x_0,\dots,\x_k]{n}\x_k \text{ in } \widehat{\cG}^\beta_{\x_0,\dots,\x_k}\big) \leq \sum_{\substack{n_1,\dots,n_k\in\N \\ n_1+\dots+n_k=n}} \prod_{i=1}^k \P_{\x_{i-1},\x_i}\big(\x_{i-1}\xleftrightarrow[\x_{i-1},\x_i]{n_i}\x_i\text{ in }\widehat{\cG}^\beta_{\x_{i-1},\x_i}\big).
	\end{equation}
	Observe that the above holds  for all versions of the weight-dependent random connection model. However, we only apply it to \(\widehat{\cG}^\beta\) in the following and we have hence restricted ourselves to this version to lighten notation. 
	
	\paragraph{Bounds on the probability of paths in the soft Boolean model.}
	In this paragraph, we combine the results of the two previous paragraphs to derive bounds on the probability that certain paths exist. Recall the notation of \(\C_\beta=\C_{\beta,\gamma,0,\alpha}(\0)\) for the component of the origin in the soft Boolean model and the notation \(\M_\beta=\M_{\beta,\gamma,0,\delta}(\0)\) for its Euclidean diameter (to the power \(d\)). Now, the event \(\M_\beta>m\) is equivalent to the existence of a path starting in \(\o\) where all vertices are located in \(B(m^{1/d})\) except for the last vertex that is located outside the ball. We prepare bounds for that type of events in the following.
	
	Let us denote by \(\{\x\xleftrightarrow[D\times J]{}\y \text{ in } \cG^{\beta}_{\x,\y}\}\) for two given vertices \(\x\) and \(\y\), a domain \(D\subset{\R}^d\), and a measurable set \(J\subset(0,1)\) the event that \(\x\) and \(\y\) are connected by a shortcut-free path in \(\cG^\beta\) where all skeleton vertices but \(\x\) and \(\y\) are elements of \(D\times J\). If \(J=(0,1)\), we simply write \(\{\x\xleftrightarrow[D]{}\y \text{ in } \cG^{\beta}_{\x,\y}\}\). We further denote by \(\{\x\xleftrightarrow[D\times J]{n}\y \text{ in } \widehat{\cG}^{\beta}_{\x,\y}\}\) the event that \(\x\) and \(\y\) are connected by a shortcut-free path of \emph{length \(n\)} in \(\widehat{\cG}^\beta\) where the skeleton has the same restriction as above. Note that now only the skeleton vertices are restricted to certain locations and marks. However, the results of the previous paragraphs tell us that this is what matters when it comes to bound the probability of a path. Note further that this notation is consistent with the one used above to describe paths with a \emph{given} skeleton. Indeed, if the whole skeleton is given, the location domain \(D\) and the vertex marks \(J\) reduce to the given points. We remark that it will always be clear from the context whether we refer to a domain for the vertex locations or to a concrete skeleton. Making use of the domination derived in Lemma~\ref{lem:MainThmLGUpperDomModel} Part~(i), we directly obtain
	\begin{equation}\label{eq:pathsBoundI}
		\P_{\x,\y} \big(\x\xleftrightarrow[D\times J]{}\y \text{ in } \cG^{\beta}_{\x,\y}\big) 
			 \leq \sum_{n\in\N}\P_{\x,\y}\big(\x\xleftrightarrow[D\times J]{n}\y \text{ in } \widehat{\cG}^{\beta}_{\x,\y}\big). 
	\end{equation} 
	Using Mecke's equation~\cite{LastPenrose2017}, the BK-inequality~\eqref{eq:BK}, the \(n\)-connection property of Lemma~\ref{lem:MainThmLGUpperDomModel} Part~(ii), and writing \(\x=\x_0, \y=\x_k\) the probability on the right-hand side can further be bounded by 
	\begin{equation}\label{eq:pathsBoundII}
		\begin{aligned}
			& \P_{\x,\y}  \big(\x\xleftrightarrow[D\times J]{n}\y \text{ in } \widehat{\cG}^{\beta}_{\x,\y}\big) \\
			& \leq \sum_{k=1}^n\E_{\x_0,\x_k}\Bigg[\sum_{\substack{\x_1,\dots,\x_{k-1}\in\cX \\ x_j\in D, u_j\in J, \, \forall j}}^{\neq}\1_{\{\x_0,\x_1,\dots,\x_k \text{ form skeleton}\}} \P_{\x_0,\x_1,\dots,\x_k}\big(\x\xleftrightarrow[\x_0,\x_1,\dots,\x_k]{n}\x_k \text{ in } \widehat{\cG}^{\beta}_{\x_0,\dots,\x_k}\big)\Bigg] \\
			& \leq \sum_{k=1}^n \sum_{\substack{n_1,\dots,n_k\in\N \\ n_1+\dots+n_k=n}}\int\limits_{\substack{(D\times J)^{k-1} \\ \text{skeleton structure}}} \hspace{-0.5 cm} \d \x_1\cdots\d \x_{k-1}\prod_{i=1}^k \P_{\x_{i-1},\x_i}\big(\x_{i-1}\xleftrightarrow[\x_{i-1},\x_i]{n_i}\x_i\text{ in }\widehat{\cG}^\beta_{\x_{i-1},\x_i}\big) \\
			& \leq \sum_{k=1}^n \binom{n}{k} (\beta \widehat{C})^{n-k}\int\limits_{\substack{(D\times J)^{k-1} \\ \text{skeleton structure}}} \hspace{-0.5 cm} \d \x_1\cdots\d \x_{k-1}\prod_{i=1}^k \rho\big(\beta^{-1}g_{\gamma,\gamma/\delta}(u_{i-1},u_i)|x_{i-1}-x_i|^d\big),		
		\end{aligned}
	\end{equation}
	where we have additionally used \(\sharp\{n_1,\dots,n_k\in\N\colon n_1+\dots+n_k= n\}=\binom{n-1}{k-1}\leq \binom{n}{k}\) in the last step. Let us further recall that \(\widehat{C}=\tfrac{2^{d\delta+3}\omega_d \delta^2}{(\delta-1)(\delta-\gamma(\delta+1))}\) is the constant given in Lemma~\ref{lem:MainThmLGUpperDomModel}. In conclusion, an upper bound for the left-hand side in~\eqref{eq:pathsBoundI} is given by summing the right-hand side of \eqref{eq:pathsBoundII} over all \(n\in\N\). 
	
	{Let us elaborate on the typical choice of \(D\times J\) in~\eqref{eq:pathsBoundI}. By nature of our question, we are interested in paths that end outside of \(B(m^{1/d})\) but all vertices except the last one are located inside the ball, so that typically \(D=B(m^{1/d})\). Furthermore, we will often require that vertices on the path are not too powerful to have good control on the expectation bound. More precisely, for the proof of Theorem~\ref{thm:Main} Part~(iii), we typically consider paths where no vertex but the last vertex has associated ball volume larger than \(s_m:=m^{(\delta-1)/\delta}\). Since the volume is parametrised as \(u^{-\gamma}\), we choose \(J=(s_m^{-1/\gamma},1)\) and consider paths of the form \(\x\xleftrightarrow[B(m^{1/d})\times (s_m^{-1/\gamma},1)]{n}\y\). That is, \(\x\) and \(\y\) are connected by a paths of length \(n\), and all intermediate skeleton vertices are located within \(B(m^{1/d})\) and have associated ball volume no larger than \(s_m\).}

	\subsubsection{Proof of the upper bound in Part~(i)}\label{sec:upperBound(i)}
	We prove the upper bound of Theorem~\ref{thm:Main} Part~(i) in this section. That is, we assume {\(\delta<1/\gamma-1\)}. Since the influence of the radii is rather weak compared to the long-range effects in this setting, we can apply the previously derived bounds directly to obtain the desired result which is summarised in the following proposition. To quantify the bound on \(\beta\) for which our results hold, we recall from Theorem~\ref{thm:Main}
	\[
		\beta_0= \tfrac{1}{2^{d\delta+3}+1}\cdot \tfrac{\delta-1}{\omega_d \delta}\big(1-\gamma\tfrac{\delta+1}{\delta}\big)=\tfrac{(\delta-1)(\delta-\gamma(\delta+1))}{(2^{d\delta+3}+1)\omega_d \delta^2}.
	\]
As we consider various paths starting in the origin \((o,u_o)\) in the following, we will also often refer to the origin as \(\x_0=(0,u_0)\) in order to have a clean numeration of the vertices on the path.
	
	\begin{prop}\label{prop:mainThm(i)upperBound}
		Let \(\beta<\beta_0\), \(\delta>1\) and {\(\delta<1/\gamma-1\)}. Consider the soft Boolean model \(\cG^\beta_\o=\cG^{\beta,\gamma,0,\delta}_\o\) and its Euclidean diameter \(\M_\beta=\M_{\beta,\gamma,0,\delta}\). Then, for all \(m>1\), we have
		\[
			\P_o\big(\M_\beta>m\big)\leq C_2 m^{1-\delta},
		\]
		where \(C_2\) is given below in~\eqref{eq:MainThmSGConstant}.
	\end{prop} 
	\begin{proof}
		Using the notation from the previous section, writing \(\o=\x_0\) we have
		\begin{equation*}
			\begin{aligned}
				\{\M_\beta>m\} & = \bigcup_{n\in\N}\big\{\exists \x\colon |x|^d>m \text{ and } \o\xleftrightarrow[B(m^{1/d})]{n}\x \text{ in }\widehat{\cG}_\o^\beta\big\} \\
				& =\bigcup_{n\in\N}\bigcup_{k=1}^n\big\{\exists \x_1,\dots, \x_k\colon |x_1|^d,\dots,|x_{k-1}|^d<m, |x_k|^d>m \text{ and } \x_0\xleftrightarrow[\x_0,\x_1,\dots,\x_k]{n}\x_k \text{ in }\widehat{\cG}_\o^\beta\big\}.
			\end{aligned}
		\end{equation*}
		From now on we assume that each path considered occurs in \(\widehat{\cG}_\o^\beta\). We deduce from Mecke's equation~\cite{LastPenrose2017}
		\begin{equation*}
			\begin{aligned}
				\P_o(
				&
					\M_\beta>m) 
					\leq \sum_{\substack{n\in\N \\ k\leq n}} \, \int\limits_0^1 \d u_0 \int\limits_{\substack{|x_k|^d>m \\ u_k\in(0,1)}}\hspace{-0.33cm}\d \x_k \, \P_{\x_0,\x_k}\big(\exists \x_1,\dots, \x_{k-1}\colon |x_1|^d,\dots,|x_{k-1}|^d<m \text{ and } \x_0\xleftrightarrow[\x_0,\x_1,\dots,\x_k]{n}\x_k\big).
			\end{aligned}
		\end{equation*}
        Let us focus on a fixed path of length \(n\) and skeleton length \(k\). Then, we must have \(|x_{\ell}-x_{\ell-1}|>m^{1/d}/k\) for some \(\ell\in\{1,\dots,k-1\}\) or \(|\x_{k}-\x_{k-1}|>|\x_k|/k\). {Let us start with the first case, fix some \(\ell\in\{1,\dots,k-1\}\), and work on the event \(|x_{\ell}-x_{\ell-1}|>m^{1/d}/k\). Then, applying the probability bounds on paths derived in~\eqref{eq:pathsBoundII} yields} 
		\begin{equation*}
			\begin{aligned}
				\int\limits_0^1 \d u_0 & \int\limits_{\substack{|x_k|^d>m \\ u_k\in(0,1)}}\d \x_k \, \P_{\x_0,\x_k}\big(\exists \x_1,\dots, \x_{k-1}\colon |x_1|^d,\dots,|x_{k-1}|^d<m, |x_{\ell}-x_{\ell-1}|^d>\tfrac{m}{k^d}, \text{ and } \x_0\xleftrightarrow[\x_0,\x_1,\dots,\x_k]{n}\x_k\big) \\
				& \leq \binom{n}{k}(\beta \widehat{C})^{n-k} \beta^\delta  k^{d\delta} m^{-\delta}\int\limits_{\substack{(0,1)^{k+1} \\ \text{skeleton structure}}}{\d u_0\cdots\d u_k} \, g_{\gamma,\nicefrac{\gamma}{\delta}}(u_\ell,u_{\ell-1})^{-\delta} \\
				& \phantom{\binom{n}{k}(\beta \widehat{C})^{n-k}\int\limits_{\text{skeleton structure}}\bigotimes_{j=0}^k \d u_j}
				\times \int\limits_{\substack{B(m^{1/d})^{k-1} \\ |x_k|^d>m}}{\d x_1\cdots \d x_k} \prod_{\substack{i=1 \\ i\neq \ell}}^k \rho\big(\beta^{-1}g_{\gamma,\nicefrac{\gamma}{\delta}}(u_{i-1},u_i)|x_{i-1}-x_i|^d\big),
			\end{aligned}
		\end{equation*}
		where we used the independence of vertex locations and vertex marks, i.e., \(\d\x_j = \d x_j \d u_j\), and the bound {\(\rho(\beta^{-1}g_{\gamma,\gamma/\delta}(u_{\ell-1},u_\ell)|x_\ell-x_{\ell-1}|^d)\leq \beta^{\delta}g_{\gamma,\gamma/\delta}(u_{\ell-1},u_\ell)^{-\delta} m^{-\delta}k^{d\delta}\)}. Focusing on the integral {and leaving aside the factor that precedes it for now}, we perform the change of variables \(z_i=\beta^{-1/d}g_{\gamma,\nicefrac{\gamma}{\delta}}(u_i,u_{i-1})^{1/d}(x_i-x_{i-1})\) {and drop the restriction on the domain of integration of \(x_i\)}, starting from \(i=n\) and successively going to \(i=\ell+1\), and continuing from \(i=\ell-1\) and successively going to \(i=1\), which yields the upper bound
		\begin{equation*}
			\begin{aligned}
				\beta^{k-1} 
				& 
					\int\limits_{\substack{(0,1)^{k+1} \\ \text{skeleton structure}}}{\d u_0\cdots\d u_k}\, g_{\gamma,\nicefrac{\gamma}{\delta}}(u_\ell,u_{\ell-1})^{-\delta} \prod_{\substack{i=1 \\ i\neq \ell}}^k g_{\gamma,\nicefrac{\gamma}{\delta}}(u_i,u_{i-1})^{-1} \Big(\int\limits_{B(m^{1/d})}\d x_\ell\Big)\Big(\int\limits_{\R^d}\d z \, \rho(|z|^d)\Big)^{k-1} 
				 \\& 
				 	\leq \omega_d m \beta^{k-1}(\tfrac{\omega_d \delta}{\delta-1})^{k-1}\int\limits_{\substack{(0,1)^{k+1} \\ \text{skeleton structure}}}{\d u_0\cdots\d u_k} \, g_{\gamma,\nicefrac{\gamma}{\delta}}(u_\ell,u_{\ell-1})^{-\delta} \prod_{\substack{i=1 \\ i\neq \ell}}^k g_{\gamma,\nicefrac{\gamma}{\delta}}(u_i,u_{i-1})^{-1}, 
			\end{aligned}
		\end{equation*}
		{using \(\int\limits_{\R^d} \rho(|z|^d)=\omega_d\delta/(\delta-1)\). Note here that \(\rho (\beta^{-1}g_{\gamma,\gamma/\delta}(u_{\ell+1},u_{\ell}) |x_{\ell+1}-x_\ell|^d)\) is the only term containing \(x_\ell\) in the product, namely in \(\rho\big(\beta^{-1}g_{\gamma,\gamma/\delta}(u_{\ell+1},u_{\ell}) |x_{\ell+1}-x_\ell|^d\big)\). Therefore, after applying the change of variables to \(x_{\ell+1}\) and then integrating over all of \(\R^d\), \(x_\ell\) no longer appears, resulting in the empty integral.} Hence,
		\begin{equation*}
			\begin{aligned}
				& \int\limits_0^1 \d u_0 \int\limits_{\substack{|x_k|^d>m \\ u_k\in(0,1)}}\d \x_k \, \P_{\x_0,\x_k}\big(\exists \x_1,\dots, \x_{k-1}\colon |x_1|^d,\dots,|x_{k-1}|^d<m, |x_{\ell}-x_{\ell-1}|^d>\tfrac{m}{k^d}, \text{ and } \x_0\xleftrightarrow[\x_0,\x_1,\dots,\x_k]{n}\x_k\big) \\
				& \ \leq m^{1-\delta}\binom{n}{k} \beta^{n+(\delta-1)}k^{d\delta}\widehat{C}^{n-k}(\tfrac{\omega_d \delta}{\delta-1})^k\int\limits_{\substack{(0,1)^{k+1} \\ \text{skeleton structure}}}{\d u_0\cdots\d u_k} \, g_{\gamma,\nicefrac{\gamma}{\delta}}(u_\ell,u_{\ell-1})^{-\delta} \prod_{\substack{i=1 \\ i\neq \ell}}^k g_{\gamma,\nicefrac{\gamma}{\delta}}(u_i,u_{i-1})^{-1},
			\end{aligned}
		\end{equation*}
		{adding an additional factor \(\omega_d\delta/(\delta-1)>1\) for convenience.} In the case that \(|x_{k}-x_{k-1}|>|x_k|/k\), we obtain similarly
		\begin{equation*}
			\begin{aligned}
				\int\limits_0^1 \d u_0 & \hspace{-0.2cm}\int\limits_{\substack{|x_k|^d>m \\ u_k\in(0,1)}} \hspace{-0.2cm} \d \x_k \, \P_{\x_0,\x_k}\big(\exists \x_1,\dots, \x_{k-1}\colon |x_1|^d,\dots,|x_{k-1}|^d<m, |x_{k}-x_{k-1}|^d>\tfrac{|x_k|^d}{k^d}, \text{ and } \x_0\xleftrightarrow[\x_0,\x_1,\dots,\x_k]{n}\x_k\big) \\
				& \leq \binom{n}{k}(\beta \widehat{C})^{n-k} \beta^\delta  k^{d\delta} \int\limits_{\substack{(0,1)^{k+1} \\ \text{skeleton structure}}}{\hspace{-0.7cm} \d u_0\cdots\d u_k} \, g_{\gamma,\nicefrac{\gamma}{\delta}}(u_k,u_{k-1})^{-\delta}\Big(\int\limits_{|x_k|^d>m} \d x_k \ |x_k|^{-d\delta}\Big) \\
				& \phantom{\binom{n}{k}(\beta \widehat{C})^{n-k}\int\limits_{\text{skeleton structure}}\bigotimes_{j=0}^k \d u_j}
				\times \int\limits_{\substack{B(m^{1/d})^{k-1} \\ |x_k|^d>m}}{\hspace{-0.5cm}  \d x_1\cdots\d x_{k-1}} \prod_{i=1}^{k-1} \rho\big(\beta^{-1}g_{\gamma,\nicefrac{\gamma}{\delta}}(u_{i-1},u_i)|x_{i-1}-x_i|^d\big) \\
				& \leq m^{1-\delta}\binom{n}{k} \beta^{n+(\delta-1)}k^{d\delta}\widehat{C}^{n-k}(\tfrac{\omega_d \delta}{\delta-1})^k\int\limits_{\substack{(0,1)^{k+1} \\ \text{skeleton structure}}}{\hspace{-0.7cm} \d u_0\cdots\d u_k} \, g_{\gamma,\nicefrac{\gamma}{\delta}}(u_k,u_{k-1})^{-\delta} \prod_{\substack{i=1}}^{k-1} g_{\gamma,\nicefrac{\gamma}{\delta}}(u_i,u_{i-1})^{-1}.
			\end{aligned}
		\end{equation*}
		Since both cases yield the same upper bound, we infer, by summing over \(\ell\),
		\begin{equation}\label{eq:part(i)upperBound}
			\begin{aligned}
				 \P_o & (\M_\beta>m) 
				 \\ &
				 	\leq m^{1-\delta}\sum_{\substack{n\in\N \\ k\leq n}} \binom{n}{k}\beta^{n+(\delta-1)}
					k^{d\delta}\widehat{C}^{n-k}(\tfrac{\omega_d \delta}{\delta-1})^k 
					\sum_{\ell=1}^k\int\limits_{\substack{(0,1)^{k+1} \\ \text{skeleton} \\ \text{structure}}}\frac{{\d u_0\cdots\d u_k}}{g_{\gamma,\nicefrac{\gamma}{\delta}}(u_\ell,u_{\ell-1})^{\delta}} \prod_{\substack{i=1 \\ i\neq \ell}}^k g_{\gamma,\nicefrac{\gamma}{\delta}}(u_i,u_{i-1})^{-1}.
			\end{aligned}
		\end{equation}
		It remains to calculate the integral with respect to the vertex marks. Denoting by \(h\) the index of the smallest vertex marks among the skeleton, the integral under consideration reads
		\begin{equation}\label{eq:intMinimalMarkLocation}
			\begin{aligned}
				\sum_{h=0}^{k} \, \int\limits_{\substack{1>u_0>u_1>\dots>u_h \\ u_h<u_{h+1}<\dots u_k}} {\hspace{-0.7cm} \d u_0\cdots\d u_k} \, g_{\gamma,\nicefrac{\gamma}{\delta}}(u_\ell,u_{\ell-1})^{-\delta} \prod_{\substack{i=1 \\ i\neq \ell}}^k g_{\gamma,\nicefrac{\gamma}{\delta}}(u_i,u_{i-1})^{-1}.
			\end{aligned}
		\end{equation}
		We have to distinguish various cases for \(\ell\) and \(h\). Let us start with the easy case \(h=k\). That is, we assume the skeleton's vertex marks are strictly decreasing. The above integral is then bounded from above for any \(\ell\), using the definition \(g_{\gamma,\nicefrac{\gamma}{\delta}}(u,v)=(u\vee v)^{\gamma/\delta}(u\wedge v)^\gamma\), {the fact \(u_0^{-\gamma/\delta}\leq u_0^{-\gamma-\gamma/\delta}\), a similar bound for \(u_k^{-\gamma/\delta}\), and that all \(u_i\leq 1\),}
		\begin{equation*}
			\begin{aligned}
				& \int\limits_{u_0>\dots>u_{\ell-2}} \hspace{-0.5cm} \d u_0\cdots\d u_{\ell-2} \prod_{i=0}^{\ell-2} u_i^{-\gamma-\gamma/\delta} \int\limits_{0}^{u_{\ell-2}}\d u_{\ell-1} u_{\ell-1}^{-2\gamma}\int\limits_{0}^{u_{\ell-1}}\d u_{\ell} u_\ell^{-\gamma/\delta-\gamma\delta} 	\int\limits_{u_{\ell+1}>\dots>u_k}\hspace{-0.5cm}\d u_{\ell+1}\cdots \d u_k \prod_{i=\ell+1}^{k} u_i^{-\gamma-\gamma/\delta}	
				\\ &
					\quad {\leq \Big(\int\limits_0^1 \d u \, u^{-\gamma-\gamma/\delta}\Big)^{k-1} \int\limits_{0}^{1}\d u_{\ell-1} \, u_{\ell-1}^{-2\gamma}\int\limits_{0}^{1}\d u_{\ell} \, u_\ell^{-\gamma/\delta-\gamma\delta}}
					\leq \frac{1}{(1-2\gamma)(1-\gamma(\delta+1/\delta))} \Big(\frac{\delta}{\delta-\gamma(\delta+1)}\Big)^{k},
			\end{aligned}
		\end{equation*}
		since \(1-\gamma(\delta+1/\delta)>0\) and \(1-2\gamma>0\), {as \(1/\delta+1<2<\delta+1<1/\gamma\), where we again added an additional factor \(\delta/(\delta-\gamma(\delta+1))\) for convenience.}
		It is easy to see that the case \(h=0\) and any \(\ell\) yields the same bound. We consider next, the case \(h\not\in\{0,k\}\). Here, our calculations depend on the relation between \(h\) and \(\ell\). We start with the case \(\ell\in\{h+2,\dots,k\}\) and obtain the upper bound, {by performing similar bounds and calculations as before},
		\begin{equation*}
			\begin{aligned}
				\int\limits_{u_0>\dots>u_{h-1}} 
				& 
					\hspace{-0.5cm}\d u_0\cdots\d u_{h-1}\prod_{i=0}^{h-1} u_i^{-\gamma-\gamma/\delta}\int\limits_{0}^{u_{h-1}}\d u_{h} \, u_{h}^{-2\gamma} \int\limits_{u_{h+1}<\dots<u_{\ell-2}} \hspace{-0.5cm}\d u_{h+1}\cdots\d u_{\ell-2} \prod_{i=h+1}^{\ell-2} u_i^{-\gamma-\gamma/\delta}
				\\ & 
					\qquad \times \int\limits_{u_{\ell-2}}^1\d u_{\ell-1} \, u_{\ell-1}^{-\gamma/\delta-\gamma\delta}\int\limits_{u_{\ell-1}}^1 \d u_\ell \, u_\ell^{-2\gamma}\int\limits_{u_{\ell+1}<\dots<u_k}\hspace{-0.5cm}\d u_{\ell+1}\cdots\d u_{k} \prod_{i=\ell+1}^{k}\d u_i \ u_i^{-\gamma-\gamma/\delta} 
				\\ & 
					\leq \frac{1}{(1-2\gamma)^2(1-\gamma(\delta+1/\delta)} \Big(\frac{\delta}{\delta-\gamma(\delta+1)}\Big)^k,
			\end{aligned}
		\end{equation*}
		{adding once more the factors required to have exponent \(k\).} Again, the case \(\ell\in\{1,\dots,h-1\}\) yields the same bound and only the cases \(\ell\in\{h,h+1\}\) remain. We start with the case \(\ell=h\). Then, the integral under consideration is bounded by
		\begin{equation}\label{eq:Bound(i)GammaCondition}
			\begin{aligned}
				\int\limits_{u_0>\dots>u_{h-2}} 
				& 
					\hspace{-0.5cm}\d u_0\cdots\d u_{h-2} \prod_{i=0}^{h-2} u_i^{-\gamma-\gamma/\delta}\int\limits_{0}^{u_{h-2}}\d u_{h-1} \, u_{h-1}^{-2\gamma} \int\limits_{0}^{u_{h-1}}\d u_{h} \, u_{h}^{-\gamma-\gamma\delta}\hspace{-0.5cm}\int\limits_{u_{h+1}<\dots<u_{k}} \hspace{-0.5cm}\d u_{h+1}\cdots\d u_{k} \prod_{i=h+1}^{k}u_i^{-\gamma-\gamma/\delta} 
				\\ & 
					\leq \frac{1}{(1-2\gamma)(1-\gamma(\delta+1))}\Big(\frac{\delta}{\delta-\gamma(\delta+1)}\Big)^k,
			\end{aligned}
		\end{equation}
		{as \(1/\delta+1<2<\delta+1<1/\gamma\).} Finally, for \(\ell=h+1\), we obtain the same bound with the same calculations. Hence, we obtain for~\eqref{eq:intMinimalMarkLocation},
		\begin{equation*}
			\begin{aligned}
				\sum_{h=0}^{k} \, \int\limits_{\substack{1>u_0>u_1>\dots>u_h \\ u_h<u_{h+1}<\dots u_k}} 
				& 
					\hspace{-0.7cm}\d u_0\cdots\d u_{k} \, g_{\gamma,\nicefrac{\gamma}{\delta}}(u_\ell,u_{\ell-1})^{-\delta} \prod_{\substack{i=1 \\ i\neq \ell}}^k g_{\gamma,\nicefrac{\gamma}{\delta}}(u_i,u_{i-1})^{-1}
				\\ &
					\leq \frac{k+1}{(1-2\gamma)^2(1-\gamma(\delta+1))}\Big(\frac{\delta}{\delta-\gamma(\delta+1)}\Big)^k.
			\end{aligned}
		\end{equation*}
		{Since the observed bound does not depend on \(\ell\) and as \(\ell\leq k\leq n\), we obtain by plugging this into~\eqref{eq:part(i)upperBound} and recalling the definition of \(\widehat{C}\) from Lemma~\ref{lem:MainThmLGUpperDomModel}},
		\begin{equation*}
			\begin{aligned}
				\P_o(\M_\beta>m) & \leq m^{1-\delta}\sum_{n\in\N}\tfrac{\beta^{\delta-1}(n+1)n^{d\delta+1}}{(1-2\gamma)^2(1-\gamma(\delta+1))} \beta^n\sum_{k=1}^n \binom{n}{k}\widehat{C}^{n-k}\Big(\tfrac{\omega_d \delta^2}{(\delta-1)(\delta-\gamma(\delta+1))}\Big)^k\leq C_2 m^{1-\delta},
			\end{aligned}
		\end{equation*}
		where 
	\begin{equation}\label{eq:MainThmSGConstant}
			C_2 = \tfrac{\beta^{\delta-1}}{(1-2\gamma)^2(1-\gamma(\delta+1))}\sum_{n\in\N}(n+1)n^{d\delta+1} (\tfrac{\beta}{\beta_0})^n
		\end{equation}
		is a finite constant as \(\beta<\beta_0\).
	\end{proof}

	\subsubsection{Proof of the upper bound in Part~(iii)}\label{sec:upperBound(iii)}
	In this section, we prove the upper bound of Part~(iii) of Theorem~\ref{thm:Main} following the strategy outlined in Section~\ref{sec:Strategy}. We fix the following notation throughout this section. As above, we have
	\[
		\beta_0= \tfrac{1}{2^{d\delta+3}+1}\cdot \tfrac{\delta-1}{\omega_d \delta}\big(1-\gamma\tfrac{\delta+1}{\delta}\big)=\tfrac{(\delta-1)(\delta-\gamma(\delta+1))}{(2^{d\delta+3}+1)\omega_d \delta^2}.
	\]
	Recall the notation of \(\C_\beta=\C_{\beta,\gamma,0,\alpha}(\0)\) for the component of the origin in the soft Boolean model and the notation \(\M_\beta=\M_{\beta,\gamma,0,\delta}\) for its Euclidean diameter (to the power \(d\)). Recall further
	\[ 
		\zeta = (\delta-1)/\delta \quad \text{ and } \quad s_m=m^{\zeta}.
	\]

	We start by bounding the probability of the presence of powerful vertices inside the component of the origin. Here, powerful refers to a vertex with {associated ball volume larger than \(s_m\)} as demanded in the lower bound. We show that for \(\beta<\beta_0\), the presence of such a vertex in the whole component is as unlikely as its presence in the direct neighbourhood of the origin and hence the probability of this event matches the probability of the lower bound. Afterwards, we deduce bounds for all other paths connecting \(\o\) to distance \(m^{1/d}\) when no powerful vertex is present depending on the occurrence of long edges.
	
	\paragraph{Powerful vertices contained in the component of the origin.}	
	{Recall that a vertex with mark \(u\) has associated volume \(u^{-\gamma}\).} Let us now define the event that there exists a powerful vertex in the component of the origin as 	
	\begin{equation}\label{eq:UpperPowerfulEndvertex} 
	\begin{aligned} 
		\cF(m): 
		& 
			= \big\{\exists \x\in\C_\beta\colon u_x^{-\gamma}>s_m\big\}\cap\{U_o^{-\gamma}<s_m\} 
			= \big\{\exists\x\colon  u_x^{-\gamma}>s_m \text{ and } \o\xleftrightarrow[\R^d\times(s_m^{-1/\gamma},1)]{}\x\big\}\cap\{U_o^{-\gamma}<s_m\}.
	\end{aligned}
	\end{equation}
	Here, we work on the event that, the origin is not itself a powerful vertex since this only happens with probability \(s_m^{-1/\gamma}\) which is of lower order than the desired one. {The second equality states that \(\0\) and the powerful vertex \(\x\) are connected by a path consisting of vertices with associated volume smaller than \(s_m\) only. Observe that this is no restriction since any previous vertex on the path with volume larger than \(s_m\) would already fulfil the event.}
	The following lemma coincides with \ref{stepB}.
	
	\begin{lemma}\label{lem:upperPowerfulVertex}
		Consider the soft Boolean model \(\cG^\beta_\0=\cG^{\beta,\gamma,0,\delta}(\xi_\0)\) with \(\delta>1\), {\(1/\gamma-1>1/\delta\)}, and \(\beta<\beta_0\). Then, there exists \(M>1\) such that for all \(m>M\), we have
		\[
			\P_o\big(\cF(m)\big)\leq C_6^{(1)} m^{-(1-\gamma)\zeta/\gamma},
		\]
		where \(C_6^{(1)}\) is given below in~\eqref{eq:constUpperPowerfulVertex}.
	\end{lemma}  
	\begin{proof}
	Observe that the skeleton of each path considered in \(\cF(m)\) has necessarily a skeleton with decreasing vertex marks. This is due to the fact that the target vertex \(\x\) is the most powerful vertex on the path. {Further, each volume except for the last vertex of the path is smaller than \(s_m\), which corresponds to a mark larger than \(s_m^{-1/\gamma}\).} Using Mecke's equation~\cite{LastPenrose2017}, \eqref{eq:pathsBoundI}, and~\eqref{eq:pathsBoundII}, we hence immediately infer
	\begin{equation*}
		\begin{aligned}
			\P_o(\cF(m))
			&
				\leq \int\limits_{s_m^{-1/\gamma}}^1 \d u_o \int\limits_{\R^d\times (0,s_m^{-1/\gamma})} \hspace{-0.5cm}\d\x \ \P_{\o,\x}\big(\0\xleftrightarrow[ \R^d\times (s_m^{-1/\gamma},1)]{}\x \text{ in } \cG_{\o,\x}^\beta\big) 
			\\& 
				\leq \sum_{\substack{n\in\N \\ k\in\{1,\dots,n\}}} \binom{n}{k} (\beta\widehat{C})^{n-k} \int\limits_{s_m^{-1/\gamma}}^1 \d u_0 \int\limits_{\substack{(\R^d)^k \\ u_0>u_1>\dots>u_{k-1}>s_m^{-1/\gamma}>u_k}}{\hspace{-1.3cm} \d\x_1\cdots \d\x_k} \ \prod_{i=1}^k \rho\big(\tfrac{1}{\beta}u_j^{\gamma}u_{j-1}^{\gamma/\delta}|x_j-x_{j-1}|^d\big),
		\end{aligned}
	\end{equation*} 
    where we again identified \(\o=\x_0\). We again make use of the independence of vertex locations and marks and use that \(\d \x_k=\d x_k \d u_k\) to perform the change of variables \(z_i=\beta^{-1/d}u_i^{\gamma/d}u_{j-1}^{\gamma/d\delta}(x_i-x_{i-1})\) together with \(\int \rho(|x|^d)\d x = (\omega_d \delta)/(\delta-1)\), the fact that \(1-\gamma-\nicefrac{\gamma}{\delta}>0\) and that all marks are bounded by one, to obtain
	\begin{equation*}
		\begin{aligned}
			\P_o (\cF(m)) & \leq \sum_{\substack{n\in\N \\ k\in\{1,\dots,n\}}} \beta^n \binom{n}{k} \widehat{C}^{n-k} \big(\tfrac{\omega_d \delta}{\delta-1}\big)^k \int\limits_{s_m^{-1/\gamma}}^1 \d u_0 \, u_0^{-\nicefrac{\gamma}{\delta}}\int\limits_{s_m^{-1/\gamma}}^{u_0}\d u_1 \cdots \int\limits_{s_m^{-1/\gamma}}^{u_{k-2}}\d u_{k-1} \, \prod_{j=1}^{k-1} u_j^{-\gamma-\nicefrac{\gamma}{\delta}}\int\limits_{0}^{s_m}\d u_k \, u_k^{-\gamma} \\
			& \leq \tfrac{\delta}{(\delta-\gamma)(1-\gamma)}s_m^{1-\gamma} \sum_{n\in\N}\beta^n 	\sum_{k=0}^n \binom{n}{k} \widehat{C}^{n-k} \big(\tfrac{\omega_d \delta^2}{(\delta-1)(\delta-\gamma(\delta+1))}\big)^k \\
			& \leq \tfrac{\delta}{(\delta-\gamma)(1-\gamma)}s_m^{1-\gamma} \sum_{n=0}^\infty \big(\tfrac{\beta}{\beta_0}\big)^n = \big(\tfrac{\delta}{(\delta-\gamma)(1-\gamma)}\cdot \tfrac{\beta_0	}{\beta_0-\beta}\big)s_m^{1-\gamma},
		\end{aligned}
	\end{equation*}
	where we used the definition of \(\widehat{C}\) from Lemma~\ref{lem:MainThmLGUpperDomModel} and \(\beta<\beta_0\). Hence, choosing
	\begin{equation}\label{eq:constUpperPowerfulVertex}
		C_6^{(1)} = \tfrac{(\delta-1)(\delta-\gamma(\delta+1))}{(\delta-\gamma)(1-\gamma)(2^{d\delta+3}+1)\omega_d\delta}\cdot\tfrac{1}{\beta_0-\beta}
	\end{equation}
	concludes the proof.
	\end{proof}

	Let us shortly remark that further restricting the target vertex \(\x\) in the event \(\cF(m)\) to be located in \(B(m^{1/d})\) corresponds to a version of the lower-bound strategy. More precisely, we define
 	\begin{equation*}
		\cE(m):= \big\{\exists \x\in\X\colon |x|^d<m, u_x<s_m, \text{ and } \0\xleftrightarrow[B(m^{1/d})\times (s_m^{-1/\gamma},1)]{}\x \text{ in } \cG_{\o}^\beta\big\}\cap\{U_o>s_m\}. 
	\end{equation*}
	Clearly, \(\cE(m)\subset \cF(m)\) and each vertex \(\x\) located in \(B(m^{1/d})\) fulfilling the event \(\cE(m)\) is connected to \(\0\) but also to some vertex at distance \(m^{1/d}\) of \(\o\) with a constant probability by our calculations in the proof of Proposition~\ref{lem:MainThmLGLower}.

	\paragraph{The occurrence of long edges in the component of the origin.} 
	In this section, we deal with the case that the origin is connected to a vertex at distance \(m^{1/d}\), where we can assume that no powerful vertex, as defined in the previous section, is used. {More precisely, we now only consider paths that do not contain any vertex with associated volume larger than \(s_m\), as the existence of such paths are contained in the event \(\cF(m^{1/d})\)}. To do so, we have to bound the occurrence of long edges in the component of the origin. Those edges play a crucial role as they contribute significantly to the spatial spread out and thus the Euclidean diameter of the component. Since we have no access to very powerful vertices for the strategy considered in this section, it is not a priori clear, where these long edges occur in a path. As outlined in Section~\ref{sec:Strategy}, we distinguish the cases whether the path connecting \(\0\) to distance \(m^{1/d}\) ends within distance \(2m^{1/d}\) or farther away. Additionally, we consider the event that there are no edges longer than \(d(m)^{1/d}\) present in \(\C^\beta\cap (B(2 m^{1/d})\times(s_m^{-1/\gamma},1))\), where \(d(m)\) is some real number in \((\beta m^{\gamma\zeta},m]\) to be specified later, {cf.\ the discussion in Section~\ref{subsec:discussion_log}.} With this at hand, we define the three events
	\begin{equation}\label{eq:MainThmLGupperEvents} 
		\begin{aligned}
			 \calG(m)
			 &
			 	:= \big\{\exists \x\colon m<|x|^d\leq 2^d m, \, u_x^{-\gamma}<s_m \text{ and }\0\xleftrightarrow[B(m^{1/d})\times(s_m^{-1/\gamma},1)]{}\x \big\}\cap\{U_o>s_m\}, 
			 \\
			 	\cH(m) 
			 &
			 	:= \big\{\exists \x,\y\colon |x|^d>2^d m, \, |y|^d<m, \, u_y^{-\gamma}, u_x^{-\gamma}<s_m \text{ with }\o\xleftrightarrow[B(m^{1/d})\times(s_m^{-1/\gamma},1)]{}\y\sim\x\big\} 
			 \\ & 
			 	\phantom{:=}\quad \cap \{U_o>s_m\}, 
			 \\
			 	\cI(d(m)) 
			 & 
			 	:= \big\{\exists \x,\y\colon |x|^d,|y|^d\leq 2^dm, \, u_x^{-\gamma},u_y^{-\gamma}<s_m, \, |y-x|^d>d(m),\o\xleftrightarrow[B(2m^{{1}/{d}})\times(s_m^{-1/\gamma},1)]{}\y\sim\x\big\} 
			 \\ & 
			 	\phantom{:=} \quad \cap\{U_o>s_m\}.
		\end{aligned}
	\end{equation}	
	Let us recall one last time that \(\beta_0=\tfrac{(\delta-1)(\delta-\gamma(\delta+1))}{(2^{d\delta+3}+1)\omega_d \delta^2}\), \(\zeta = (\delta-1)/\delta\), \( s_m=m^{\zeta}\) and that a vertex with associated ball volume smaller than \(s_m\) has mark larger than \(s_m^{-1/\gamma}\). We start by bounding the probability of the event \(\cH(m)\), which is \ref{stepC} in the outlined strategy.  
	\begin{lemma}\label{lem:MainThmLGupperH}
		Consider the soft Boolean model \(\cG^\beta_o=\cG^{\beta,\gamma,0,\delta}(\xi_o)\) with \(\delta>1\), {\(1/\delta<1/\gamma-1<\delta\)}, and \(\beta<\beta_0\). Then, there exists \(M>1\) such that, for all \(m>M\), we have
		\[
			\P_o\big(\cH(m)\big)\leq C_6^{(2)} m^{- (1-\gamma)\zeta/\gamma},
		\]
		where \(C_6^{(2)}\) is given below in~\eqref{eq:MainThmLGupperLemHConst}.
	\end{lemma}
	\begin{proof}
		Since the path considered in the event \(\cH(m)\) lies in the ball \(B(m^{1/d})\) until the final step that then connects to some vertex at distance \(2m^{1/d}\), we can write using the conditional independence of edges and Mecke's equation
		\begin{equation*}
			\begin{aligned}
				& \P_o\big(\cH(m)\big) 
				\\ & 
					\leq \int\limits_{s_m^{-1/\gamma}}^1 \d u_o \int\limits_{|x|^d>2^d m}\d x\int\limits_{s_m^{-1/\gamma}}^1 \hspace{-0.3cm}\d u_x \ \P_{\o,\x}\big(\exists \y\colon |y|^d<m, \o\xleftrightarrow[B(m^{1/d})\times(s_m^{-1/\gamma},1)]{}\y\sim\x \text{ in }\cG^\beta_{\o,\x}\big) 
				\\ & 
					\leq \int\limits_{s_m^{-1/\gamma}}^1 \hspace{-0.15cm} \d u_o \int\limits_{|y|^d<m}\hspace{-0.3cm}\d y \int\limits_{s_m^{-1/\gamma}}^1 \hspace{-0.15cm} \d u_y \int\limits_{|x|^d>2^d m}\hspace{-0.3cm}\d x\int\limits_{s_m^{-1/\gamma}}^1 \hspace{-0.15cm} \d u_x \ \P_{\o,\y}\big(\o\xleftrightarrow[B(m^{1/d})\times(s_m^{-1/\gamma},1)]{}\y \text{ in }\cG^{\beta}_{\o,\y}\big) \big(\tfrac{1}{\beta}(u_x\wedge u_y)^\gamma |x-y|^d\big)^{-\delta} 
				\\ & 
					\leq \beta^{\delta} \hspace{-0.15cm}\int\limits_{s_m^{-1/\gamma}}^1 \hspace{-0.15cm} \d u_o \int\limits_{|y|^d<m} \hspace{-0.3cm}\d y \int\limits_{s_m^{-1/\gamma}}^1 \hspace{-0.15cm} \d u_y \P_{\o,\y}\big(\o\xleftrightarrow[B(m^{1/d})\times(s_m^{-1/\gamma},1)]{}\y \text{ in }\cG^{\beta}_{\o,\y}\big)\int\limits_{s_m^{-1/\gamma}}^1 \hspace{-0.15cm}\d u_x (u_x\wedge u_y)^{-\gamma\delta} \hspace{-0.2cm}\int\limits_{|x-y|^d>m} \hspace{-0.4cm}\d x \, |x-y|^{-d\delta}.
			\end{aligned}
		\end{equation*}
		The integral with respect to \(\x=(x,u_x)\) reads
		\begin{equation*}
			\begin{aligned}
				\int\limits_{s_m^{-1/\gamma}}^1\d u_x (u_x\wedge u_y)^{-\gamma\delta} \int\limits_{|x-y|^d>m} \hspace{-0.2cm} \d x \, |x-y|^{-d\delta} & = \tfrac{\omega_d}{\delta-1}m^{1-\delta}\Bigg[\int\limits_{s_m^{-1/\gamma}}^{u_y} \d u_x \, u_x^{-\gamma\delta} + \int\limits_{u_y}^1 \d u_x \, u_y^{-\gamma\delta}\Bigg] \\
				& \leq \tfrac{\omega_d}{\delta-1}m^{1-\delta}\Big[\tfrac{u_y^{1-\gamma\delta}\vee s_m^{1-\gamma\delta}}{|1-\gamma\delta|} + u_y^{-\gamma\delta}\Big] \\	
                & \leq \big(1+\tfrac{1}{|1-\gamma\delta|}\big)\tfrac{\omega_d}{\delta-1}u_y^{-\gamma\delta}m^{1-\delta}+\tfrac{\omega_d }{|1-\gamma\delta|(\delta-1)}m^{1-\delta-\zeta(1-\gamma\delta)/\gamma},
			\end{aligned}
		\end{equation*}
		using \(s_m=m^{\zeta}\) in the last step. {Here, we assumed that \(\gamma\delta\neq 1\) and comment on this matter below.} To finish the proof, we therefore have to bound the two terms 
		\begin{equation}\label{eq:MainThmLGupperHToBound1}
			\begin{aligned}
				\tfrac{\beta^{\delta}\omega_d }{|1-\gamma\delta|(\delta-1)}m^{1-\delta-\zeta(1-\gamma\delta)/\gamma}\int\limits_{s_m^{-1/\gamma}}^1 \d u_o \int\limits_{|y|^d<m}\d y \int\limits_{s_m^{-1/\gamma}}^1 \d u_y \ \P_{\o,\y}\big(\o\xleftrightarrow[B(m^{1/d})\times(s_m^{-1/\gamma},1)]{}\y\text{ in }\cG^{\beta}_{\o,\y}\big)
			\end{aligned}
		\end{equation}
		and
		\begin{equation}\label{eq:MainThmLGupperHToBound2}
			\begin{aligned}
				\big(1+\tfrac{1}{|1-\gamma\delta|}\big)\tfrac{\beta^\delta \omega_d}{\delta-1} m^{1-\delta}\int\limits_{s_m^{-1/\gamma}}^1 \d u_o \int\limits_{|y|^d<m}\d y \int\limits_{s_m^{-1/\gamma}}^1 \d u_y \ u_y^{-\gamma\delta}\P_{\o,\y}\big(\o\xleftrightarrow[B(m^{1/d})\times (s_m^{-1/\gamma},1)]{}\y\text{ in }\cG^{\beta}_{\o,\y}\big).
			\end{aligned}
		\end{equation}
  \emph{Step 1: Bounding the integral in~\eqref{eq:MainThmLGupperHToBound1}.} 
  We use the skeleton strategy and deduce by~\eqref{eq:pathsBoundI} and~\eqref{eq:pathsBoundII}, together with the same change of variables as above, 		
  		\begin{equation}\label{eq:MainThmLGupperHToBound1Calculation}
			\begin{aligned}
				\int\limits_{s_m^{-1/\gamma}}^1 
				&
					\d u_o \int\limits_{|y|^d<m}\d y \int\limits_{s_m^{-1/\gamma}}^1 \d u_y \ \P_{\o,\y}\big(\o\xleftrightarrow[B(m^{1/d})\times(s_m^{-1/\gamma},1)]{}\y \text{ in }\widehat{\cG}^\beta_{\0,\y}\big) 
				\\ & 
					\leq \sum_{\substack{n\in\N \\ k\leq n}}\binom{n}{k} (\beta\widehat{C})^{n-k}\ \int\limits_{s_m^{-1/\gamma}}^1\d u_0  \hspace{-0.3cm} \int\limits_{\substack{\big(B(m^{1/d})\times(s_m^{-1/\gamma},1)\big)^{k} \\ \x_0,\dots,\x_k \text{ form skeleton}}}\hspace{-1cm} \d \x_1\cdots\d\x_k \ \prod_{j=1}^{k} \rho\big(\beta^{-1}g_{\gamma,\nicefrac{\gamma}{\delta}}(u_{j-1},u_j)|x_j-x_{j-1}|^d\big)
				\\ & 
					\leq \sum_{\substack{n\in\N \\ k\leq n}}\binom{n}{k} \beta^n \widehat{C}^{n-k}(\tfrac{\omega_d \delta}{\delta-1})^k \hspace{-0.3cm}\int\limits_{\substack{(s_m^{-1/\gamma},1)^{k+1} \\ u_0,\dots,u_k \text{ have skeleton structure}}}\hspace{-1.5cm}\d u_0\cdots\d u_k \ \prod_{j=1}^k	g_{\gamma,\nicefrac{\gamma}{\delta}}(u_{j-1},u_j)^{-1}.
			\end{aligned}
		\end{equation}
		As  above in the proof of Proposition~\ref{prop:mainThm(i)upperBound}, we decompose the integral by considering
    skeletons where the \(h\)-th vertex has the minimum mark (or largest associated volume, respectively) and sum over all possible skeletons to obtain
		\begin{equation*}
			\begin{aligned}
				 \int\limits_{\substack{(s_m^{-1/\gamma},1)^{k+1} \\ u_0,\dots,u_k \text{ have skeleton structure}}} 
				 & 
				 	\hspace{-1.5cm}\d u_0\dots\d u_k \ \prod_{j=1}^k	g_{\gamma,\nicefrac{\gamma}{\delta}}(u_{j-1},u_j)^{-1} 
				 \\ & 
				 	\hspace{-2 cm} \leq \sum_{h = 0}^k \int\limits_{s_m^{-1/\gamma}}^1 \d u_0 \int\limits_{s_m^{-1/\gamma}}^{u_0}\d u_1 \dots \int\limits_{s_m^{-1/\gamma}}^{u_{h-1}}\d u_h \int\limits_{u_h}^1 \d u_{h+1}\dots\int\limits_{u_{k-1}}^1 \d u_k \ \prod_{j=1}^k g_{\gamma,\nicefrac{\gamma}{\delta}}(u_{j-1},u_j)^{-1}.		
				\end{aligned}
		\end{equation*}
		We have to distinguish the three cases, whether the minimum mark vertex is the first vertex, the last vertex, or some vertex in between. For \(h=0\) the integral reads
		\begin{equation*}
			\begin{aligned}
				\int\limits_{s_m^{-1/\gamma}}^1 \d u_0 \int\limits_{u_0}^{1}\d u_1 \dots \int\limits_{u_{k-1}}^{1}\d u_k \,  u_0^{-\gamma}\big(\prod_{j=1}^{k-1}u_j^{-\nicefrac{\gamma}{\delta}-\gamma}\big)u_k^{-\nicefrac{\gamma}{\delta}} \leq \tfrac{\delta}{\delta-\gamma}\big(\tfrac{\delta}{\delta-\gamma(\delta+1)}\big)^{k-1}\tfrac{1}{1-\gamma},
			\end{aligned}
		\end{equation*} 
		using \(1-\nicefrac{\gamma}{\delta}-\gamma>0\), similar to above. For \(h=k\), similarly and using that all marks are bounded by one,
		\begin{equation*}
			\begin{aligned}
				\int\limits_{s_m^{-1/\gamma}}^1 \d u_0 \int\limits_{s_m^{-1/\gamma}}^{u_0}\d u_1 \dots \int\limits_{s_m^{-1/\gamma}}^{u_{k-1}}\d u_k \,  u_0^{-\nicefrac{\gamma}{\delta}}\big(\prod_{j=1}^{k-1}u_j^{-\nicefrac{\gamma}{\delta}-\gamma}\big)u_k^{-\gamma} \leq \tfrac{\delta}{\delta-\gamma}\big(\tfrac{\delta}{\delta-\gamma(\delta+1)}\big)^{k-1}\tfrac{1}{1-\gamma}.
			\end{aligned}
		\end{equation*}
		Finally, for all other values of \(h\), we infer
\begin{equation}\label{eq:MainThmLGupperHInBetweenL}
			\begin{aligned}
				\int\limits_{s_m^{-1/\gamma}}^1 \d u_0 & \int\limits_{s_m^{-1/\gamma}}^{u_0}\d u_1 \dots \int\limits_{s_m^{-1/\gamma}}^{u_{h-1}}\d u_h \int\limits_{u_h}^1 \d u_{h+1}\dots\int\limits_{u_{k-1}}^1 \d u_k u_0^{-\gamma}\big(\prod_{j=1}^{h-1}u_j^{-\nicefrac{\gamma}{\delta}-\gamma}\big)u_h^{-2\gamma}\big(\prod_{j=h+1}^{k-1}u_j^{-\gamma-\nicefrac{\delta}{\gamma}}\big)u_k^{-\nicefrac{\gamma}{\delta}} \\
				& \leq \tfrac{\delta}{\delta-\gamma}\big(\tfrac{\delta}{\delta-\gamma(\delta+1)}\big)^{k-2}\tfrac{1}{1-\gamma}\tfrac{1}{|1-2\gamma|}\big(1\vee m^{-\zeta(1-2\gamma)/\gamma}\big),
			\end{aligned}
		\end{equation} 
		where we assumed \(\gamma\neq 1/2\), and we comment on the excluded case below. If \(\gamma<1/2\), the above does not depend on \(m\) and since \(0<1-2\gamma<1-\gamma(1+1/\delta)\), the above calculations together with~\eqref{eq:MainThmLGupperHToBound1Calculation} yield the following bound on~\eqref{eq:MainThmLGupperHToBound1}, 
		\begin{equation} \label{eq:MainThmLGupperHResult1}
			\begin{aligned}
			 	& \tfrac{\beta^{\delta} \omega_d}{|1-\gamma\delta|(\delta-1)}m^{1-\delta-\zeta(1-\gamma\delta)/\gamma}\int\limits_{s_m^{-1/\gamma}}^1 \d u_o \int\limits_{|y|^d<m}\d y \int\limits_{s_m^{-1/\gamma}}^1 \d u_y \ \P_{\o,\y}\big(\o\xleftrightarrow[B(m^{1/d})\times(s_m^{-1/\gamma},1)]{}\y \text{ in }\cG^{\beta}_{\o,\y}\big) \\
			 	& \leq m^{1-\delta-\zeta(1-\gamma\delta)}\tfrac{\beta^{\delta} \omega_d(\delta-\gamma(\delta+1))^2}{\delta|1-\gamma\delta|(\delta-1)(\delta-\gamma)(1-\gamma)|1-2\gamma|}\sum_{n\in\N}\beta^n\sum_{k=0}^n (k+1)\binom{n}{k}\widehat{C}^{n-k}\big(\tfrac{\omega_d\delta^2}{(\delta-\gamma(\delta+1))(\delta-1)}\big)^k,\\
				& \leq m^{1-\delta-\zeta(1-\gamma\delta)} \tfrac{\beta^{\delta} \omega_d(\delta-\gamma(\delta+1))^2}{\delta|1-\gamma\delta|(\delta-1)(\delta-\gamma)(1-\gamma)|1-2\gamma|}\sum_{n\in\N}(n+1)\big(\tfrac{\beta}{\beta_0}\big)^n,
			\end{aligned}
		\end{equation}
		using the form of \(\widehat{C}\). Further, the sum is finite as \(\beta<\beta_0\). Using \(\zeta=(\delta-1)/\delta\), we find for the {exponent of} \(m\)
		\begin{equation*}
			 1-\delta-\tfrac{\zeta(1-\gamma\delta)}{\gamma} =-\tfrac{\zeta}{\gamma}<-\tfrac{(1-\gamma)\zeta}{\gamma}.
		\end{equation*}
		For \(1/2<\gamma<\delta/(\delta+1)\) or, equivalently \(1/\delta<1/\gamma -1<1\), we infer performing the same calculations
	\begin{equation}\label{eq:MainThmLGupperHResult2}
			\begin{aligned}
				& \tfrac{\beta^{\delta} \omega_d}{|1-\gamma\delta|(\delta-1)}m^{1-\delta-\zeta(1-\gamma\delta)/\gamma}\int\limits_{s_m^{-1/\gamma}}^1 \d u_o \int\limits_{|y|^d<m}\d y \int\limits_{s_m^{-1/\gamma}}^1 \d u_y \ \P_{\o,\y}\big(\o\xleftrightarrow[B(m^{1/d})\times(s_m^{-1/\gamma},1)]{}\y \text{ in }\cG^\beta_{\o,\y}\big) \\
			 	& \leq m^{1-\delta-\tfrac{\zeta(1-\gamma\delta)}{\gamma}-\tfrac{\zeta(1-2\gamma)}{\gamma}} \tfrac{\beta^{\delta} \omega_d(\delta-\gamma(\delta+1))^2}{\delta|1-\gamma\delta|(\delta-1)(\delta-\gamma)(1-\gamma)|1-2\gamma|}\sum_{n\in\N}(n+1)\big(\tfrac{\beta}{\beta_0}\big)^n
			\end{aligned}
		\end{equation}
		and it remains to {compare the exponent of \(m\) with the desired one, i.e., }
		\begin{equation*} 
			\begin{aligned}
				1-\delta-\tfrac{\zeta(1-\gamma\delta)}{\gamma}-\tfrac{\zeta(1-2\gamma)}{\gamma} & = -\tfrac{\zeta}{\gamma} - \tfrac{\zeta(1-2\gamma)}{\gamma} = -2\tfrac{(1-\gamma)}{\zeta}<-\tfrac{(1-\gamma)\zeta}{\gamma}	.		
			\end{aligned}
		\end{equation*}
 		If \(\gamma=1/2\) in~\eqref{eq:MainThmLGupperHResult1} only the \(|1-2\gamma|^{-1}\) term in the constant has to be replaced by \(\zeta\log(m)/\gamma\). Hence, for \(\beta<\beta_0\), there exists a constant \(C>0\) and some \(\zeta'>\zeta\) such that we find for~\eqref{eq:MainThmLGupperHToBound1}
		\begin{equation}\label{eq:MainThmLGupperHResult3}
			m^{1-\delta-\zeta(1-\gamma\delta)}\int\limits_{s_m^{-1/\gamma}}^1 \d u_o \int\limits_{|y|^d<m}\d y \int\limits_{s_m^{-1/\gamma}}^1 \d u_y \ \P_{\o,\y}\big(\o\xleftrightarrow[B(m^{1/d})\times(s_m^{-1/\gamma},1)]{}\y \text{ in }\cG^{\beta}_{\o,\y}\big)\leq  C m^{-(1-\gamma)\zeta'/\gamma}.
		\end{equation}
		
  \emph{Step 2: Bounding the integral in~\eqref{eq:MainThmLGupperHToBound2}}. We infer with the same arguments as above
		\begin{equation}\label{eq:MainThmLGupperHSecondBound}
			\begin{aligned}
				 \big(1+\tfrac{1}{|1-\gamma\delta|}\big) & \tfrac{\beta^\delta \omega_d}{\delta-1} m^{1-\delta}\int\limits_{s_m^{-1/\gamma}}^1  \d u_o \int\limits_{|y|^d<m}\d y \int\limits_{s_m^{-1/\gamma}}^1 \d u_y \ u_y^{-\gamma\delta}\P_{\o,\y}\big(\o\xleftrightarrow[B(m^{1/d})\times(s_m^{-1/\gamma},1)]{}\y\text{ in }\cG^{\beta}_{\o,\y}\big) \\
				& \leq \big(1+\tfrac{1}{|1-\gamma\delta|}\big)\tfrac{\beta^\delta \omega_d}{\delta-1} m^{1-\delta}\sum_{\substack{n\in\N}} \sum_{k=1}^n \beta ^n \binom{n}{k}\widehat{C}^{n-k}\big(\tfrac{\omega_d \delta}{\delta-1}\big)^k \\
				 & \phantom{\big(1+\tfrac{1}{|1-\gamma\delta|}\big)\tfrac{2^{d\delta}\beta^\delta \omega_d}{\delta-1}}
				 \times\sum_{h = 0}^k \ \int\limits_{\substack{u_0\geq \dots \geq u_h\geq s_m^{-1/\gamma} \\ u_h \leq u_{h+1}\leq \dots \leq u_k}} \hspace{-0.5cm} \d u_0\cdots\d u_k \ u_k^{-\gamma\delta} \prod_{j=1}^k g_{\gamma,\nicefrac{\gamma}{\delta}}(u_{j-1},u_j)^{-1},	
			\end{aligned}
		\end{equation}
		and we focus on the integral on the right-hand side. Again, we have to consider the three different cases depending on where the most powerful vertex is located within the path. We start with \(h=0\) and infer, similarly to the above, 
		\begin{equation*}
			\begin{aligned}
				 \int\limits_{s_m^{-1/\gamma}}^1 \d u_0 \int\limits_{u_0}^{1}\d u_1 \dots \int\limits_{u_{k-1}}^1 \d u_k \ u_0^{-\gamma}\Big(\prod_{j=1}^{k-1} u_j^{-\nicefrac{\gamma}{\delta}-\gamma}\Big) u_k^{-\gamma\delta-\nicefrac{\gamma}{\delta}} \leq \tfrac{1\vee m^{-\zeta(1-\gamma(\delta+\nicefrac{1}{\delta}))/\gamma}}{|1-\gamma(\delta+\nicefrac{1}{\delta})|}\big(\tfrac{\delta}{\delta-\gamma(\delta+1)}\big)^{k-1}\tfrac{1}{1-\gamma},
			\end{aligned}
		\end{equation*}
		using \(u_k>m^{-\zeta/\gamma}\). Hence, with the same arguments as above, combining this with~\eqref{eq:MainThmLGupperHSecondBound}, we infer for some constant \(C>0\) that 
		\begin{equation}\label{eq:MainThmLGupperHResult4}
			\begin{aligned}
				m^{1-\delta} & \sum_{\substack{n\in\N}}\beta^n \sum_{k=1}^n \binom{n}{k} \widehat{C}^{n-k}\big(\tfrac{\omega_d \delta}{\delta-1}\big)^k \int\limits_{s_m^{-1/\gamma}}^1 \d u_0 \int\limits_{u_0}^{1}\d u_1 \dots \int\limits_{u_{k-1}}^1 \d u_k \ u_0^{-\gamma}\Big(\prod_{j=1}^{k-1} u_j^{-\nicefrac{\gamma}{\delta}-\gamma}\Big) u_k^{-\gamma\delta-\nicefrac{\gamma}{\delta}}	\\
				& \leq C \big(m^{1-\delta}\vee m^{1-\delta-\zeta(1-\gamma(\delta+\nicefrac{1}{\delta}))/\gamma}\big) \tfrac{\beta_0}{\beta_0-\beta},
			\end{aligned}
		\end{equation}
		since \(\beta<\beta_0\). To deduce the order in \(m\), we only have to check the case when \(m^{1-\delta-\zeta(1-\gamma(\delta+1/\delta))/\gamma}\) is the dominant term. In the other case we immediately have \(m^{1-\delta}<m^{-(1-\gamma)\zeta/\gamma}\) since {\(\delta<1/\gamma-1\)}. As 
		\[
			1-\delta-\tfrac{\zeta(1-\gamma(\delta+\nicefrac{1}{\delta}))}{\gamma} = -\tfrac{\zeta(1-\nicefrac{\gamma}{\delta})}{\gamma}<-\tfrac{\zeta(1-\gamma)}{\gamma},
		\]
		since \(\delta>1\), we find that~\eqref{eq:MainThmLGupperHResult4} is bounded by \(Cm^{-(1-\gamma)\zeta'/\gamma}\) for some \(\zeta'>\zeta\). 
		
		The next case we consider is \(1\leq h\leq k-1\). Then, the integral on the right-hand side of~\eqref{eq:MainThmLGupperHSecondBound} reads
		\begin{equation*}
			\begin{aligned}
				\int\limits_{s_m^{-1/\gamma}}^1 \d u_0 & \int\limits_{s_m^{-1/\gamma}}^{u_0} \d u_1 \dots \int\limits_{s_m^{-1/\gamma}}^{u_{h-1}} \d u_h \int\limits_{u_h}^1 \d u_{h +1} \dots \int\limits_{u_{k-1}}^1 \d u_k \ u_0^{-\gamma/\delta} \Big(\prod_{\substack{j=1 \\ j\neq h}}^{k-1} u_j^{-\gamma-\gamma/\delta}\Big) u_h^{-2\gamma} u_k^{-\gamma\delta-\gamma/\delta}.
			\end{aligned}
		\end{equation*}
		We only consider the case when \(1-\gamma\delta-\gamma/\delta<0\),  since the other case coincides with~\eqref{eq:MainThmLGupperHInBetweenL} with a slightly changed constant. Now integrating out $u_k$ and then using that \(u_{k-1}\geq u_h\) we infer
		\begin{equation*}
			\begin{aligned}
				\int\limits_{s_m^{-1/\gamma}}^1 \d u_0 & \int\limits_{s_m^{-1/\gamma}}^{u_0} \d u_1 \dots \int\limits_{s_m^{-1/\gamma}}^{u_{h-1}} \d u_h \int\limits_{u_h}^1 \d u_{h +1} \dots \int\limits_{u_{k-1}}^1 \d u_k \, u_0^{-\gamma/\delta} \Big(\prod_{\substack{j=1 \\ j\neq h}}^{k-1} u_j^{-\gamma-\gamma/\delta}\Big) u_h^{-2\gamma} u_k^{-\gamma\delta-\gamma/\delta} \\
				& \leq \tfrac{\delta}{\gamma(\delta^2+1)-\delta}\big(\tfrac{\delta}{\delta-\gamma(\delta+1)}\big)^{k-h-1}
                {\int\limits_{s_m^{-1/\gamma}}^1 \d u_0 \int\limits_{s_m^{-1/\gamma}}^{u_0} \d u_1 \dots }\int\limits_{s_m^{-1/\gamma}}^{u_{h-1}} \d u_h
                 \ u_0^{-\gamma/\delta} \Big(\prod_{j=1}^{h-1}u_j^{-\gamma-\nicefrac{\gamma}{\delta}}\Big)u_h^{1-2\gamma-\gamma\delta-\gamma/\delta} \\
				& \leq \tfrac{\delta}{\gamma(\delta^2+1)-\delta}\big(\tfrac{\delta}{\delta-\gamma(\delta+1)}\big)^{k-2}\tfrac{\delta}{\delta-\gamma} m^{-\zeta(2-2\gamma-\gamma\delta-\gamma/\delta)/\gamma},
			\end{aligned}
		\end{equation*}
		where we have again restricted ourselves to the case \(2-2\gamma-\gamma\delta-\gamma/\delta<0\) since otherwise the \(\gamma<1/2\) case of~\eqref{eq:MainThmLGupperHInBetweenL} applies again. As above, combining with~\eqref{eq:MainThmLGupperHSecondBound} and since \(\beta<\beta_0\), we find \(C>0\) such that
		\begin{equation}\label{eq:MainThmLGupperHResult5}
			\begin{aligned}
				m^{1-\delta}\sum_{\substack{n\in\N}} 
				& 
					\beta^n \sum_{k=1}^n \binom{n}{k}\widehat{C}^{n-k}\big(\tfrac{\omega_d \delta}{\delta-1}\big)^k \sum_{h = 1}^{k-1} \ \int\limits_{\substack{u_0\geq \dots \leq u_h\geq s_m^{-1/\gamma} \\ u_h \leq u_{h+1}\leq \dots \leq u_k}} \hspace{-0.5cm} \d u_0\cdots \d u_k \ u_k^{-\gamma\delta} \ \prod_{j=1}^k g_{\gamma,\gamma/\delta}(u_{j-1},u_j)^{-1}  
				\\&
					\leq C m^{1-\delta-\zeta(2-2\gamma-\gamma\delta-\gamma/\delta)/\gamma}	< Cm^{-(1-\gamma)\zeta'/\gamma}
			\end{aligned}
		\end{equation}
		for some \(\zeta'>\zeta\) since
		\begin{equation*} 
			\begin{aligned}
				1-\delta-\tfrac{\zeta(2-2\gamma-\gamma\delta-\gamma/\delta)}{\gamma} &= -2\tfrac{\zeta(1-\gamma)}{\gamma} + \tfrac{\zeta}{\delta} < -\tfrac{(1-\gamma)\zeta}{\gamma},		
			\end{aligned}
		\end{equation*}
		as \(1/\delta<1/\gamma -1 \).
		
		It remains to bound the case when the most powerful vertex of the path is the end vertex. This case can be seen as a pendant strategy to the one used in the lower bound in Lemma~\ref{lem:upperPowerfulVertex}: Namely, instead of connecting the origin to a vertex with mark slightly more powerful than \(s_m\), which then is connected to distant vertices, we now connect to vertices slightly less powerful. We shall see however that this strategy dominates in probability all former ones and is of the same order as the one from the lower bound. More precisely, for \(h=k\) the integral in~\eqref{eq:MainThmLGupperHSecondBound} under consideration reads
		\begin{equation*}
			\begin{aligned}
				\int\limits_{s_m^{-1/\gamma}}^1 \d u_0 \int\limits_{s_m^{-1/\gamma}}^{u_0}\d u_1 \dots \int\limits_{s_m^{-1/\gamma}}^{u_{k-1}}\d u_k \ u_0^{-\nicefrac{\gamma}{\delta}}\Big(\prod_{j=1}^{k-1} u_j^{-\gamma-\nicefrac{\gamma}{\delta}}\Big) u_k^{-\gamma\delta-\gamma} \leq \tfrac{1}{\gamma(\delta+1)-1}\big(\tfrac{\delta}{\delta-\gamma(\delta+1)}\big)^{k-1} \tfrac{\delta}{\delta-\gamma}m^{-\zeta(1-\gamma(\delta+1))/\gamma},
			\end{aligned}
		\end{equation*}
		where we have used that \(\gamma>1/(\delta+1)\), or equivalently \(1/\gamma-1<\delta\). Plugging this into the right-hand side of~\eqref{eq:MainThmLGupperHSecondBound} yields, performing similar calculations as before,		
	\begin{equation}\label{eq:MainThmLGupperHResult6}
			\begin{aligned}
				\big(1 & +\tfrac{1}{|1-\gamma\delta|}\big) \tfrac{\beta^\delta \omega_d}{\delta-1} m^{1-\delta}\sum_{\substack{n\in\N}}\beta^n  \sum_{k=1}^n \binom{n}{k}\widehat{C}^{n-k}\big(\tfrac{\omega_d \delta}{\delta-1}\big)^k \\
				& \phantom{abcdefg} \times \int\limits_{s_m^{-1/\gamma}}^1 \d u_0 \int\limits_{s_m^{-1/\gamma}}^{u_0}\d u_1 \dots \int\limits_{s_m^{-1/\gamma}}^{u_{k-1}}\d u_k \ u_0^{-\nicefrac{\gamma}{\delta}}\Big(\prod_{j=1}^{k-1} u_j^{-\gamma-\nicefrac{\gamma}{\delta}}\Big) u_k^{-\gamma\delta-\gamma} \\
				& \leq m^{-(1-\gamma)\zeta/\gamma} \big(1 +\tfrac{1}{|1-\gamma\delta|}\big)\tfrac{\beta^\delta \omega_d(\delta-\gamma(\delta+1))}{(\delta-1)(\gamma(\delta+1)-1)(\delta-\gamma)}\cdot \tfrac{\beta_0}{\beta_0-\beta},
			\end{aligned}
		\end{equation}  
		using again \(\beta<\beta_0\) and {\(1-\delta-\zeta(1-\gamma(\delta+1))/\gamma=-(1-\gamma)\zeta/\gamma\)}. Hence, combining~\eqref{eq:MainThmLGupperHResult4}, \eqref{eq:MainThmLGupperHResult5} and~\eqref{eq:MainThmLGupperHResult6}, we find that Term~\eqref{eq:MainThmLGupperHToBound2} is bounded by
		\[
			Cm^{-(1-\gamma)\zeta'/\gamma} + m^{-(1-\gamma)\zeta/\gamma} \big(1 +\tfrac{1}{|1-\gamma\delta|}\big)\tfrac{\beta^\delta (\delta-\gamma(\delta+1))^2}{(\gamma(\delta+1)-1)(\delta-\gamma)(1+2^{d\delta+3}) \delta^2}\cdot \tfrac{1}{\beta_0-\beta},
		\]
		where \(\zeta'>\zeta\) and hence \(m^{-(1-\gamma)\zeta'/\gamma}\) is of strictly smaller order than \(m^{-(1-\gamma)\zeta/\gamma}\). We set 
		\begin{equation}\label{eq:MainThmLGupperLemHConst}
			 C_6^{(2)}:= 2\big(1 +\tfrac{1}{|1-\gamma\delta|}\big)\tfrac{\beta^\delta (\delta-\gamma(\delta+1))^2}{(\gamma(\delta+1)-1)(\delta-\gamma)\delta^2}\cdot \tfrac{1}{\beta_0-\beta}
		\end{equation}
		and infer for \(\beta<\beta_0\) and sufficiently large \(m\)
		\[
			\P_o\big(\cH(m)\big)\leq C_6^{(2)}m^{-(1-\gamma)\zeta/\gamma},
		\]
		as claimed. {It only remains to comment on the previously excluded \(\gamma\delta=1\) case and argue that no qualitative changes appear. In this case, the \(m^{1-\delta-\zeta(1-\gamma\delta)/\gamma}/|1-\gamma\delta|\) term in~\eqref{eq:MainThmLGupperHToBound1} has to be replaced by \((\zeta/\gamma)\log(m) m^{1-\delta}\). However, all remaining calculations in Step~1 remain the same so that the right-hand side in~\eqref{eq:MainThmLGupperHResult3} still applies. Observe further that the \(1/|1-\gamma\delta|\) term in the constant of~\eqref{eq:MainThmLGupperHToBound2} is no longer required and nothing else changes, noting that \(u_y^{-\gamma\delta}=u_y^{-1}\). Carefully revisiting the calculations in Step~2, we observe that due to the additional product \(\prod g_{\gamma,\nicefrac{\gamma}{\delta}}(u_{j-1,u_j})^{-1}\) in~\eqref{eq:MainThmLGupperHSecondBound}, the exponent \(-1\) never appears so that all calculations can be performed, \emph{verbatim}, yielding the bound with the \(1/|1-\gamma\delta|\) term removed from \(C_6^{(2)}\).} 
		\end{proof}
		
		Our next result is to bound the probability of \(\cI(d(m))\), the event that the component of the origin restricted to vertices in \(B(2m^{1/d})\) contains {at least one edge} longer than \(d(m)^{1/d}\), where \(d(m)\in(\beta m^{\gamma\zeta},m]\), preparing \ref{stepD}.

		\begin{lemma}\label{lem:MainThmLGupperI}
			Consider the soft Boolean model \(\cG^\beta_o=\cG^{\beta,\gamma,0,\delta}(\xi_o)\) with \(\delta>1\), {\(1/\delta<1/\gamma-1<\delta\)}, and \(\beta<\beta_0\). Then, there exists \(M>1\) such that, for all \(m>M\), we have
		\[
			\P_o\big(\cI(d(m))\big)\leq C_6^{(2)} \big(m/d(m)\big)^{\delta-1} m^{-(1-\gamma)\zeta/\gamma},
		\]
		where \(C_6^{(2)}\) is the constant given above in~\eqref{eq:MainThmLGupperLemHConst}. 
		\end{lemma}
		\begin{proof}
			We calculate similarly to above
			\begin{equation*}
				\begin{aligned}
					& \P_o\big(\cI(d(m))\big) \\
					& \leq \int\limits_{s_m^{-1/\gamma}}^1 \d u_o \int\limits_{|x|^d<2^dm}\d x \int\limits_{s_m^{-1/\gamma}}^1 \d u_x \hspace{-0.3cm} \int\limits_{\substack{|y|^d<2^d m \\ |x-y|^d>d(m)}} \hspace{-0.3cm}\d y\int\limits_{s_m^{-1/\gamma}}\d u_y \, \P_{\o,\x}\big(\o\xleftrightarrow[B(2m^{1/d})\times(s_m^{-1/\gamma},1)]{}\x \text{ in }\cG^{\beta}_{\o,\x}\big) \tfrac{\beta^\delta}{(u_x\wedge u_y)^{\gamma\delta}|x-y|^{d\delta}} \\
					& \leq \beta^{\delta}\tfrac{\omega_d}{\delta-1}d(m)^{1-\delta}\int\limits_{s_m^{-1/\gamma}}^1 \d u_o \int\limits_{|x|^d<2^d m} \hspace{-0.3cm} \d x \int\limits_{s_m^{-1/\gamma}}^1 \hspace{-0.15cm} \d u_x \ \P_{\o,\x}\big(\o\xleftrightarrow[B(2m^{1/d})\times(s_m^{-1/\gamma},1)]{}\x \text{ in }\cG^{\beta}_{\o,\x} \big) \hspace{-0.15cm}\int\limits_{s_m^{-1/\gamma}}^1 \hspace{-0.15cm} \d u_y (u_x\wedge u_y)^{-\gamma\delta} \\
					& \leq \big(\tfrac{d(m)}{m}\big)^{1-\delta}\tfrac{\beta^{\delta}\omega_d}{\delta-1}\big((1+\tfrac{1}{|1-\gamma\delta|})m^{1-\delta}u_x^{-\gamma\delta}+\tfrac{m^{1-\delta-\zeta(1-\gamma\delta)/\zeta}}{|1-\gamma\delta|}\big) \\
					& \phantom{\leq \big(\tfrac{d(m)}{m}\big)^{1-\delta}}
					\times \int\limits_{s_m^{-1/\gamma}}^1 \d u_o \int\limits_{|x|^d<2^d m}\d x \int\limits_{s_m^{-1/\gamma}}^1 \d u_x \ \P_{\o,\x}\big(\o\xleftrightarrow[B(2m^{1/d})\times(s_m^{-1/\gamma},1)]{}\x\text{ in }\cG^\beta_{\o,\x}\big).			
				\end{aligned}
			\end{equation*} 
			Therefore, we can perform for the integral the exact same computations as in the previous proof, including the \(1=\gamma/\delta\) case, and we infer for \(\beta<\beta_0\) and sufficiently large \(m\),
			\begin{equation*}
				\P_o\big(\cI(d(m))\big)\leq C_6^{(2)} \big(m/d(m)\big)^{\delta-1} m^{-(1-\gamma)\zeta},
			\end{equation*}
				where \(C_6^{(2)}\) is the constant given above in~\eqref{eq:MainThmLGupperLemHConst}. 
   			This concludes the proof.
		\end{proof}
		
	\paragraph{Finalising the proof of the upper bound in Part (iii).}
	With the results of the previous sections, we have all tools at hand needed to finish the proof of Theorem~\ref{thm:Main}. Recall the events \(\cF(m)\) introduced in~\eqref{eq:UpperPowerfulEndvertex} as well as \(\calG(m)\), \(\cH(m)\), and \(\cI(d(m))\) introduced in~\eqref{eq:MainThmLGupperEvents}. The upper bound in Theorem~\ref{thm:Main} Part~(iii) is an immediate consequence of the following proposition.
	
	\begin{prop}\label{prop:mainThm(iii)upperBound}
		Consider the soft Boolean model \(\cG^\beta_o=\cG^{\beta,\gamma,0,\delta}(\xi_o)\) with \(\delta>1\), {\(1/\delta<1/\gamma-1<\delta\)}, and \(\beta<\beta_0\). Then, there exists \(M>1\) such that, for all \(m>M\), we have
		\[
			\P_o(\M_\beta>m)\leq C_6 \log(m)^{1\vee d(\delta-1)} m^{-(1-\gamma)\zeta/\gamma},
		\]
		where \(C_6\) is given below in~\eqref{eq:MainThmLGupperConst}.
	\end{prop}	
	\begin{proof}
	For large \(m\) by Lemma~\ref{lem:upperPowerfulVertex}
	\begin{equation*}
		\begin{aligned}
			\P_o(\M_\beta>m) 
			& 
				\leq \P_o(U_o^{-\gamma}\geq s_m)+\P_o(\{\M_\beta>m\}\cap\{U_o^{-\gamma }<s_m\}) 
			\\ &
				\leq s_m^{-1/\gamma} + \P_o(\cF(m)) +\P_o\big(\{\M_\beta>m\}\cap \cF(m)^c\cap\{U_o^{-\gamma}<s_m\}\big) 
			\\ &
				\leq \tfrac{3}{2}C_6^{(1)}m^{-(1-\gamma)\zeta/\gamma} + \P_o\big(\{\M_\beta>m\}\cap \cF(m)^c\cap\{U_o^{-\gamma}<s_m\}\big).
		\end{aligned}
	\end{equation*}
	To bound the remaining probability, we set \(\kappa=\frac{(\gamma \vee (1-\gamma))\zeta}{\gamma(\log\beta_0 -\log \beta)}\)
	and choose \(d(m)=\tfrac{m}{(\kappa\log m)^d}\) in \(\cI(d(m))\). We deduce, using Lemmas~\ref{lem:MainThmLGupperH} and~\ref{lem:MainThmLGupperI} 
		\begin{equation*}
			\begin{aligned}
				\P_o\big( 
				&
					\{\M_\beta>m\}\cap \cF(m)^c\cap\{U_o^{-\gamma}>s_m\}\big) 
				\\ & 
					\leq \P_o\big(\calG(m)\cap\cF(m)^c\big)+\P_o\big(\cH(m)\big)  
				\\ & 
					\leq \P_o\big(\calG(m)\cap \cF(m)^c\cap \cI(d(m))^c\big)+\tfrac{3C_6^{(2)}}{2}\big(\kappa \log(m)\big)^{d(\delta-1)}m^{-(1-\gamma)\zeta/\gamma},
			\end{aligned}
		\end{equation*}	
		for sufficiently large \(m\). It hence remains to bound the probability of \(\calG(m)\cap \cF(m)^c\cap \cI(d(m))^c\) which then finishes \ref{stepD}. On this event, there exists some path of length at least \(\kappa \log(m)\) where no vertex of the path has ball volume larger than \(s_m\) or mark smaller than \(s_m^{-1/\gamma}\), respectively. Hence, by~\eqref{eq:pathsBoundI}
		\begin{equation*}
			\begin{aligned}
				\P_o\big(\calG(m)\cap\cF(m)^c\cap \cI(d(m))^c\big) & \leq \sum_{n\geq \kappa \log(m)} \sum_{k=1}^n\E_o\Big[ \sum_{\substack{\x_1,\dots,\x_k \in\cX \\ u_i>s_m^{-1/\gamma} \, \forall i}}^{\neq} \P_{\x_0,\dots,\x_k}(\x_0\xleftrightarrow[\x_0,\dots,\x_k]{n}\x_k \text{ in }\widehat{\cG}^\beta_{\x_0,\dots,\x_k})\Big].
			\end{aligned}
		\end{equation*}
		Using again Mecke's equation and the skeleton strategy~\eqref{eq:pathsBoundII} once more, we infer
		\begin{equation*}
			\begin{aligned}
				\sum_{k=1}^n\E_o\Big[ & \sum_{\substack{\x_1,\dots,\x_k \in\cX \\ u_i>s_m^{-1/\gamma} \, \forall i}} \P_{\x_0,\dots,\x_k}(\x_0\xleftrightarrow[\x_0,\dots,\x_k]{n}\x_k \text{ in }\widehat{\cG}^\beta_{\x_0,\dots,\x_k})\Big] \\
				& \leq \beta^n \sum_{k=1}^n \sum_{h = 0}^k \binom{n}{k} \widehat{C}^{n-k} \big(\tfrac{\omega_d \delta}{\delta-1}\big)^k \hspace{-0.5cm} \int\limits_{\substack{u_0>u_1>\dots>u_h>s_m^{-1/\gamma}\\ u_h<u_{h+1}<\dots< u_k }} \hspace{-1cm} \d u_0\cdots\d u_k \ \prod_{j=1}^k g_{\gamma,\nicefrac{\gamma}{\delta}}(u_{j-1},u_j)^{-1} \\
				& \leq \big(\tfrac{\delta}{\delta-\gamma}\big)^2\big(\tfrac{\delta-\gamma(\delta+1)}{\delta}\big)^2(1 \vee m^{-(1-2\gamma)\zeta/\gamma}) (n+1) \big(\tfrac{\beta}{\beta_0}\big)^n,
			\end{aligned}
		\end{equation*}
		where we performed the same calculations as above in~\eqref{eq:MainThmLGupperHInBetweenL} to~\eqref{eq:MainThmLGupperHResult1} to bound the integral. Since \(\beta<\beta_0\), we have by our choice of \(\kappa\)
		\begin{equation*}
			\begin{aligned}
				\sum_{n\geq \kappa\log(m)}(n+1)(\tfrac{\beta}{\beta_0})^n &\leq \int\limits_{\kappa\log(m)}^\infty (x+1){\rm e}^{-x\log(\beta_0/\beta)} \d x \leq \tfrac{3}{2}\tfrac{\gamma\vee (1-\gamma)}{\gamma(\log\beta_0-\log\beta)^2}\log(m) m^{-(\gamma\vee (1-\gamma))\zeta/\gamma}.
			\end{aligned}
		\end{equation*}
		Therefore, 
		\begin{equation*}
			\begin{aligned}
				\P_o\big(\calG(m)\cap \cF(m)^c \cap \cI(d(m))^c\big) &  \leq  \big(\tfrac{\delta}{\delta-\gamma}\big)^2\big(\tfrac{\delta-\gamma(\delta+1)}{\delta}\big)^2 (1\vee m^{-(1-2\gamma)\zeta/\gamma}) \sum_{n\geq \kappa\log(m)} (n+1)\big(\tfrac{\beta}{\beta_0}\big)^n \\
				& \leq \tfrac{3}{2}\big(\tfrac{\delta}{\delta-\gamma}\big)^2\big(\tfrac{\delta-\gamma(\delta+1)}{\delta}\big)^2 \tfrac{\gamma\vee (1-\gamma)}{\gamma(\log\beta_0-\log\beta)^2} \log(m) m^{-(1-\gamma)\zeta/\gamma}.
			\end{aligned} 
		\end{equation*}
		Hence, we set
		\begin{equation}\label{eq:MainThmLGupperConst}
			C_6 = 
				\begin{cases}
 					2\big(\frac{(\gamma \vee (1-\gamma))\zeta}{\gamma(\log\beta_0 -\log \beta)}\big)^{d(\delta-1)} C_6^{(2)}, & \text{if } d(\delta-1)>1, \\
 					\tfrac{3}{2}\big(\frac{(\gamma \vee (1-\gamma))\zeta}{\gamma(\log\beta_0 -\log \beta)} C_6^{(2)} + \big(\tfrac{\delta}{\delta-\gamma}\big)^2\big(\tfrac{\delta-\gamma(\delta+1)}{\delta}\big)^2 \tfrac{\gamma\vee 1-\gamma}{(\log\beta_0-\log\beta)^2}\big), & \text{if } d(\delta-1)=1, \text{ and } \\
 					2\big(\tfrac{\delta}{\delta-\gamma}\big)^2\big(\tfrac{\delta-\gamma(\delta+1)}{\delta}\big)^2 \tfrac{\gamma\vee 1-\gamma}{\gamma(\log\beta_0-\log\beta)^2}, & \text{if } d(\delta-1)<1.
 				\end{cases}
		\end{equation}
		to infer 
		\[
			\P_o(\M_\beta>m) \leq C_6 \log(m)^{1\vee d(\delta-1)}m^{-(1-\gamma)\zeta/\gamma}	
		\]
		for large enough \(m\), concluding the proof.
	\end{proof}

	\subsubsection{The upper bound of Part~(ii)}\label{sec:MainThmBC}
	We shortly explain where we have to adapt our proofs to obtain the upper bound in the boundary case.
	\begin{prop}\label{prop:mainThm(ii)}
		Consider the soft Boolean model \(\cG^\beta_o=\cG^{\beta,\gamma,0,\delta}(\xi_o)\) with \(\delta>1\), {\(\delta=1/\gamma-1\)}, and \(\beta<\beta_0\). Then, there exists \(M>1\) such that, for all \(m>M\), we have
		\[
			\P_\beta(\M_\beta>m)\leq C_6 \log(m) m^{1-\delta},
		\]
		where \(C_4\) is given below in~\eqref{eq:MainThmBCupperConst}.
	\end{prop}
	\begin{proof}
		First recall that for {\(\delta=1/\gamma-1\)} the exponents \(1-\delta\) and \(-(1-\gamma)\zeta/\gamma\) coincide. {Hence, we can restrict to the case that no vertex with ball volume larger than \(s_m\) is used on the considered paths, since the complementary probability is of the desired order by Lemma~\ref{lem:upperPowerfulVertex}}. We then perform the same proof as for the case {\(\delta<1/\gamma-1\)} in Proposition~\ref{prop:mainThm(i)upperBound} {but now with the additional restriction that all skeleton vertices have associated volume smaller than \(s_m\), i.e., mark in} \((s_m^{- 1/\gamma},1)\). Observe that for {\(\delta=1/\gamma-1\)} we still have \(\gamma<1/2\) and \(1-\gamma(\delta+1/\delta)>0\). Observe that we only rely in~\eqref{eq:Bound(i)GammaCondition} on the condition \(\gamma<1/(\delta+1)\). Performing the same calculations but with the lower integral bound replaced by \(s_m^{-1/\gamma}\) and using {\(-\gamma-\gamma\delta=-1\), as \(\delta=1/\gamma-1\)}, we obtain instead
		\begin{equation*}
			\begin{aligned}
				\int\limits_{u_0>\dots>u_{h-2}>s_m^{-1/\gamma}}
				&
					\hspace{-0.7cm}\d u_0\cdots \d u_{h-2}\prod_{i=0}^{h-2} u_i^{-\gamma-\gamma/\delta} \hspace{-0.15cm} \int\limits_{s_m^{-1/\gamma}}^{u_{h-2}}\d u_{h-1} \, u_{h-1}^{-2\gamma} \int\limits_{s_m^{-1/\gamma}}^{u_{h-1}}\d u_{h} \ u_{h}^{-1}
					\hspace{-0.5cm}\int\limits_{s_m^{-1/\gamma}<u_{h+1}<\dots<u_{k}} \hspace{-0.7cm}\d u_{h+1}\cdots\d u_k\prod_{i=h+1}^{k} u_i^{-\gamma-\gamma/\delta}
				\\ & 
					\leq -\log(s_m^{-1/\gamma})\frac{1}{1-2\gamma}\Big(\frac{\delta}{\delta-\gamma(\delta+1)}\Big)^k,
			\end{aligned}
		\end{equation*}
		Due to the additional \(-\log(s_m^{-1/\gamma})=\zeta\log(m)/\gamma\) term, this summand is now clearly the dominant one in the whole sum. Hence, finalising the proof from here as done in Proposition~\ref{prop:mainThm(i)upperBound}, we obtain
		\[
			\P_o(\M_\beta>m)\leq C_4 \log(m) m^{1-\delta},
		\]
		where
		\begin{equation}\label{eq:MainThmBCupperConst}
			C_4 := \tfrac{2(\delta^2-1)(\delta+1)\beta^{\delta-1}}{\delta(\delta-1)}\sum_{n\in\N}(n+1)n^{d\delta} (\tfrac{\beta}{\beta_0})^n,
		\end{equation}
		as \(\zeta/\gamma=(\delta^2-1)/\delta\) and \((1-2\gamma)^{-1}=(\delta+1)/(\delta-1))\). This concludes the proof.
	\end{proof}

	\subsubsection{The special case \(\gamma<1/2\)} \label{sec:gammaSmall}
	The final result of this section is the following lemma, showing that the logarithmic term in the upper bound of Theorem~\ref{thm:Main} Part (iii) can be omitted if \(\gamma<1/2\) at least for a changed intensity threshold, as discussed in Section~\ref{sec:openProblems}. The proof of Theorem~\ref{thm:UpperBoundSmallGamma} follows immediately by summing the following lemma.

	\begin{lemma}\label{lem:PathBoundUpperInterGamma}
		Define \(K=\tfrac{\omega_d \delta}{\delta-1}\big(\tfrac{1}{1-2\gamma}+\tfrac{1}{1-\gamma}\big)\). Let \(\beta<K^{-1}\), \(\gamma<1/2\) and \(\delta>1\).  Then, there exist constants \(K' >0\) such that, for all \(n\in\N\), we have 
		\begin{enumerate}[(i)]
			\item for \(\delta<1/\gamma-1\)
				\begin{equation*}
					\P_o\big(\exists \y\colon |y|^d>m, u_y>s_m \text{ and }\0\xleftrightarrow[B(m^{1/d})]{n}\y \text{ in }\cG^{\beta}_\0\big) \leq K'\beta^{\delta-1}n^{d\delta+1}(\beta K)^n \, 
						m^{1-\delta},
				\end{equation*}
			\item for \(1<1/\gamma-1<\delta\)
				\begin{equation*}
					\begin{aligned}
						\P_o\big(\exists \y\colon |y|^d>m, u_y>s_m \text{ and }\0\xleftrightarrow[B(m^{1/d})]{n}\y \text{ in }\cG^{\beta}_\0\big) \leq K'\beta^{\delta-1}n^{d\delta+1}(\beta K)^n \, 
						m^{-(1-\gamma)\zeta/\gamma}. 			
					\end{aligned}
				\end{equation*}
		\end{enumerate}
	\end{lemma}
	
	Proving the lemma requires bounds on the expected number of paths as we no longer apply the skeleton strategy. 
		
	\begin{lemma}\label{lem:WholePathBound}
		Let \(\gamma<1/2\) and \(u_0\in(0,1)\). 
	\begin{enumerate}[(a)]
		\item For all \(k\in\N\), we have
    		\begin{equation*}
        		\int\limits_0^1 \d u_1 \cdots \int\limits_0^1 \d u_k \ \prod_{j=1}^k (u_{j-1}\wedge u_{j})^{-\gamma} \leq \tfrac{2-\gamma}{1-\gamma}\big(\tfrac{1}{1-2\gamma}+\tfrac{1}{1-\gamma}\big)^{k-1}u_0^{-\gamma}. 
    		\end{equation*}
    	\item Further, for all \(k\in\N\), we have
   			\[
   				\int\limits_0^1 \d u_1 \cdots \int\limits_0^1 \d u_k \ \Big(\prod_{j=1}^k (u_{j-1}\wedge u_{j})^{-\gamma}\Big) u_k^{-\gamma} \leq \big(\tfrac{1}{1-2\gamma}+\tfrac{1}{1-\gamma}\big)^{k-1} u_0^{-\gamma}.
   			\]
	\end{enumerate} 
	\end{lemma}
	\begin{proof}
		We prove the Statement (a) by induction. For \(k=1\), we have
		\begin{equation*}
    		\begin{aligned}
       			\int\limits_0^1 \d u_1 \ (u_{0}\wedge u_1)^{-\gamma} & \leq \int\limits_0^{u_0} u_1^{-\gamma}\d u_1 + u_0^{-\gamma}\leq \tfrac{2-\gamma}{1-\gamma} u_0^{-\gamma}.
    		\end{aligned}
		\end{equation*}
		For \(k\geq 2\) we have by using the induction hypothesis 
		\begin{equation*}
    		\begin{aligned}
       			\int\limits_0^1 \d u_1\cdots \int\limits_0^1 \d u_k \ \prod_{j=1}^k (u_{j-1}\wedge u_j)^{-\gamma} 
        		& \leq \tfrac{2-\gamma}{1-\gamma}\big(\tfrac{1}{1-2\gamma}+\tfrac{1}{1-\gamma}\big)^{k-2}\int\limits_0^1 u_1^{-\gamma}(u_0\wedge u_1)^{-\gamma}\d u_1 \\
       			& = \tfrac{2-\gamma}{1-\gamma}\big(\tfrac{1}{1-2\gamma}+\tfrac{1}{1-\gamma}\big)^{k-2}\Big(\int\limits_0^{u_0} u_1^{-2\gamma}\d u_1 + \int\limits_{u_0}^1 u_0^{-\gamma}u_1^{-\gamma}\d u_1 \Big) \\
      			& \leq \tfrac{2-\gamma}{1-\gamma}\big(\tfrac{1}{1-2\gamma}+\tfrac{1}{1-\gamma}\big)^{k-2}\big(\tfrac{u_0^{1-2\gamma}}{1-2\gamma}+\tfrac{u_0^{-\gamma}}{1-\gamma}\big) \\
       			&\leq \tfrac{2-\gamma}{1-\gamma}\big(\tfrac{1}{1-2\gamma}+\tfrac{1}{1-\gamma}\big)^{k-1} u_0^{-\gamma}.
    		\end{aligned}
		\end{equation*}
		For statement (b), we only show the induction start \(k=1\) as the induction step works analogously as above. We have
		\begin{equation*}
			\begin{aligned}
				\int\limits_0^1 \d u_1 (u_0\wedge u_1)^{-\gamma} u_1^{-\gamma} & = \int\limits_{0}^{u_0}\d u_1 \ u_1^{-2\gamma} +  \int\limits_{u_0}^1 \d u_1 \ u_0^{-\gamma} u_1^{-\gamma} \leq \big(\tfrac{1}{1-2\gamma}+\tfrac{1}{1-\gamma}\big)u_0^{-\gamma},
			\end{aligned}
		\end{equation*}
		as desired.
	\end{proof}

	\begin{proof}[Proof of Lemma \ref{lem:PathBoundUpperInterGamma}.]
		We restrict ourselves to the proof of Part~(ii). The adaptation to the easier case of Part~(i) is straightforward. With Lemma~\ref{lem:WholePathBound} it is also straight forward to adapt the proof of Lemma~\ref{lem:upperPowerfulVertex} to the case \(\gamma<1/2\) and \(\beta<K^{-1}\) to derive
		\begin{equation*}
			\begin{aligned}
				\sum_{n=1}^\infty\P_o\big(\exists \x\colon u_x^{-\gamma}>s_m \text{ and }\o\xleftrightarrow[\R^d\times(s_m^{-1/\gamma},1)]{n}\x\big) \leq m^{-(1-\gamma)\zeta/\gamma} K' \tfrac{1}{1-\beta K}.
			\end{aligned}
		\end{equation*}
		Hence, from now on, we work on the assumption that all considered vertices have mark no smaller than \(s_m^{-1/\gamma}\). On the path of length \(n\), either one of the intermediate edges has length at least \(m^{1/d}/n\), or the last edge is longer than \(|y|/n\). Hence, following the arguments of the proof of Proposition~\ref{prop:mainThm(i)upperBound}, we infer
		\begin{equation*}
			\begin{aligned}
				&\P_o\big(\exists \y\colon |y|^d>m, u_y^{-\gamma}<s_m \text{ and }\0\xleftrightarrow[B(m^{1/d})\times(s_m^{-1/\gamma},1)]{n}\y \text{ in }\cG^{\beta}_o\big) \\
				& \ \leq \sum_{\ell=1}^{n-1} \beta^{\delta}n^{d\delta}\big(\beta \omega_d\tfrac{\delta}{\delta-1}\big)^{n-1}m^{1-\delta} \int\limits_{(s_m^{-1/\gamma},1)^{n+1}} \hspace{-0.3cm}\d u_0\cdots\d u_n \, (u_{\ell-1}\wedge u_\ell)^{-\gamma\delta} \prod_{\substack{j=1 \\ j\neq \ell}}^n (u_{j-1}\wedge u_j)^{-\gamma} \\
				& \qquad + \beta^{\delta}n^{d\delta}\big(\beta \omega_d\tfrac{\delta}{\delta-1}\big)^{n-1}\int\limits_{|x_n|^d>m} \d x_n \, |x_n|^{-d\delta} \int\limits_{(s_m^{-1/\gamma},1)^{n+1}}\hspace{-0.3cm}\d u_0\cdots\d u_n \, (u_{n-1}\wedge u_n)^{-\gamma\delta} \prod_{\substack{j=1}}^{n-1} (u_{j-1}\wedge u_j)^{-\gamma} \\
				& \leq  \beta^{\delta-1}n^{d\delta}\big(\beta \omega_d\tfrac{\delta}{\delta-1}\big)^{n}m^{1-\delta} \sum_{\ell=1}^{n} \int\limits_{(s_m^{-1/\gamma},1)^{n+1}}\hspace{-0.3 cm}\d u_0\cdots\d u_n \, (u_{\ell-1}\wedge u_\ell)^{-\gamma\delta} \prod_{\substack{j=1 \\ j\neq \ell}}^n (u_{j-1}\wedge u_j)^{-\gamma},
			\end{aligned}
		\end{equation*}
		{using \(\omega_d\delta/(\delta-1)>1\).} We have using Lemma~\ref{lem:WholePathBound}
		\begin{equation*}
			\begin{aligned}
				\int\limits_{(s_m^{-1/\gamma},1)^{n+1}} & \d u_0\cdots\d u_n \, (u_{\ell-1}\wedge u_\ell)^{-\gamma\delta} \prod_{\substack{j=1 \\ j\neq \ell}}^n (u_{j-1}\wedge u_j)^{-\gamma} \\ 
				& \leq \tfrac{2-\gamma}{1-\gamma}\big(\tfrac{1}{1-2\gamma}+\tfrac{1}{1-\gamma}\big)^{n-\ell-1}\int\limits_{(s_m^{-1/\gamma},1)^{\ell+1}} \d u_0\cdots\d u_{\ell} \, {u_\ell^{-\gamma}}(u_{\ell-1}\wedge u_\ell)^{-\gamma\delta} \prod_{\substack{j=1}}^{\ell-1} (u_{j-1}\wedge u_j)^{-\gamma}.
			\end{aligned}
		\end{equation*}
		Since \(1/\gamma-1>\delta\), we infer
		\begin{equation}\label{eq:MainThmLGupperUseOfBound}
    		\begin{aligned}
       			\int\limits_{s_m^{-1/\gamma}}^1 & \d u_{\ell-1}(u_{\ell-1}\wedge u_{\ell-2})^{-\gamma}\Big[\int\limits_{s_m^{-1/\gamma}}^{u_{\ell-1}} \d u_\ell \ u_{\ell}^{-\gamma(\delta+1)} +\int\limits_{u_{\ell-1}}^1 \d u_\ell \ u_\ell^{-\gamma\delta} u_{\ell-1}^{-\gamma}  \Big] \\
       			& \leq \int\limits_{s_m^{-1/\gamma}}^1 \d u_{\ell-1}(u_{\ell-1}\wedge u_{\ell-2})^{-\gamma}\Big[ \tfrac{m^{-(1-\gamma(\delta+1))\zeta/\gamma}}{\gamma(\delta+1)-1} + \tfrac{u_{\ell-1}^{1-\gamma(\delta+1)}\vee u_{\ell-1}^{-\gamma}}{|1-\gamma\delta|}\Big]\\
       			&\leq \big(\tfrac{1}{\gamma(\delta+1)-1}+\tfrac{1}{|1-\gamma\delta|}\big)m^{-\zeta(1-\gamma(\delta+1))/\gamma}\int\limits_{s_m^{-1/\gamma}}^1 \d u_{\ell-1} (u_{\ell-1}\wedge u_{\ell-2})^{-\gamma} {u_{\ell-1}^{-\gamma}}. 
    		\end{aligned}
		\end{equation}
		{Here, we have again excluded the \(1=\gamma\delta\) case. However, it can be dealt with just like above in Lemma~\ref{lem:MainThmLGupperH}.} Using Lemma~\ref{lem:WholePathBound} once more, we hence get
		\[
			\int\limits_{(s_m^{-1/\gamma},1)^{n+1}} \d u_0\cdots \d u_n \, (u_{\ell-1}\wedge u_\ell)^{-\gamma\delta} \prod_{\substack{j=1 \\ j\neq \ell}}^n (u_{j-1}\wedge u_j)^{-\gamma}\leq K' \big(\tfrac{1}{1-2\gamma}+\tfrac{1}{1-\gamma}\big)^n m^{-(1-\gamma(\delta+1))\zeta/\gamma}
		\]
		for 
		\[
			K'=\tfrac{2-\gamma}{1-\gamma}(\tfrac{1}{\gamma(\delta+1)-1}+\tfrac{1}{|1-\gamma\delta|})(\tfrac{1}{1-2\gamma}+\tfrac{1}{1-\gamma})^{-2}.
		\] 
		Therefore,
		\begin{equation*}
			\begin{aligned}
				\sum_{\ell=1}^{n} & \beta^{\delta-1}n^{d\delta}\big(\beta \omega_d\tfrac{\delta}{\delta-1}\big)^{n}m^{1-\delta} \int\limits_{(s_m^{-1/\gamma},1)^{n+1}}\d u_0\cdots \d u_n \, (u_{\ell-1}\wedge u_\ell)^{-\gamma\delta} \prod_{\substack{j=1 \\ j\neq \ell}}^n (u_{j-1}\wedge u_j)^{-\gamma} \\
				&\leq n K' \beta^{\delta-1}n^{d\delta}\big(\beta K\big)^{n}m^{-(1-\gamma)\zeta/\gamma},
			\end{aligned}
		\end{equation*}
		where \(K\) is given in the formulation of the lemma, concluding the proof of Part~(ii). The proof of Part~(i) works analogously. However, the restrictions to vertices with not too small vertex marks is no longer necessary and we repeat the same computations without this restriction for the case \(\delta<1/\gamma-1\) to obtain the desired result.
\end{proof}

\subsection{Number of points contained in the component of the origin}\label{sec:NumberPoints}
In this final section we prove our results about the number of points in a subcritical component. Part~(ii) of Theorem~\ref{thm:NumberPoints}, stating that the expected cardinality of the component of the origin is infinite whenever the degree distribution has no second moment, is only stated for completeness and directly follows from the corresponding result for the classical Boolean model~\cite[Theorem 3.2]{MeesterRoy1996}. Hence, it suffices to prove Part~(i). Again, the lower bound \(\P_o(\scN_\beta\geq m)\geq c m^{1-1/\gamma}\) is an immediate consequence of the result for the classical Boolean model in~\cite{Gouere08,JahnelAndrasCali2022}. We however give the short proof for completeness. 

\begin{proof}[Proof of the lower bound in Theorem~\ref{thm:NumberPoints}, Part~(i).]
Let us denote by \(N(\x)\) the number of neighbours of the vertex \(\x=(x,u_x)\) in \(\cG^\beta_{\x}\) and let us write \(N^>(\x)\) for the number of neighbours of \(\x\) with mark no smaller than \(u_x\), {which are those with smaller associated ball}. It is clear that \(N^>(\x)\leq N(\x)\). Further, for given \(\x=(x,u_x)\), \(N^>(\x)\) is Poisson distributed with parameter \(\tfrac{\beta \omega_d \delta}{\delta-1}(u_x^{-\gamma}-1)\) independently from its location by~\cite[Proposition~2.1]{Lue2022}. 
Therefore, by the standard Poisson tail bound, each vertex with mark smaller than \((c2m)^{-1/\gamma}\) for \(c=(\delta-1)/(\beta \omega_d \delta)\) has at least \(m\) neighbours with exponentially small error probability. Additionally, each vertex located in \(B(m^{{1}/{d}})\) with mark smaller than \((c2m)^{-1/\gamma}\) is connected to~\(\0\). Therefore,
  \begin{equation*}
 	\begin{aligned}
 		\P_o&\big(\mathscr{N}_\beta>m\big) \\
 		&\geq \P\big(\exists \x\in\cX\cap \big(B(m^{1/d})\times(0,c(2m)^{-1/\gamma})\big)\big)- \P\big(\exists \x\in\cX\cap\big(B(m^{1/d})\times(0,c(2m)^{-1/\gamma})\big)\colon N^>(\x)<m\big) \\
 		&\geq \tfrac{c \omega_d}{2^{1+1/\gamma}} m^{1-1/\gamma},
 	\end{aligned}
 \end{equation*} 
 for all sufficiently large \(m>c^{1/\gamma}\), since 
 \begin{equation*}
 	\begin{aligned}
 		\P\big(\exists \x\in\cX\cap\big(B(m^{1/d})\times(0,c(2m)^{-1/\gamma})\big)\colon N^>(\x)<m\big) & \leq \int\limits_{|x|^d<m}\d x\int\limits_{0}^{c(2m)^{-1/\gamma}} \d u \ \P\big(N^<(\x)<m\big) \\
 		& \leq \tfrac{c \omega_d}{2^{1/\gamma}} m^{1-1/\gamma} \operatorname{Pois}_{2m-1}(m)\leq \tfrac{c \omega_d}{2^{1+1/\gamma}} m^{1-1/\gamma},
 	\end{aligned}
 \end{equation*}
 where we have written \(\operatorname{Pois}_{2m-1}\) for the cumulative distribution function of a Poisson random variable with parameter \(2m-1\). This proves the lower bound. 
 \end{proof}
 
 Consider now the scale-free percolation model corresponding to the choice of \(\alpha=\gamma\) in the interpolation kernel~\eqref{eq:InterpolKern}. That is, \(g_{\gamma,\gamma}(s,t)=s^\gamma t^{\gamma}\). Since \(g_{\gamma,0}\geq g_{\gamma,\gamma}\), we have \(E(\cG^{\beta,\gamma,0,\delta}(\xi))\subset E(\cG^{\beta,\gamma,\gamma,\delta}(\xi))\) by~\eqref{eq:edgeSet}. Put differently, the scale-free percolation model in this parametrisation contains all edges of the soft Boolean model and more. As a direct result, the lower bound for the tail probability of the cardinality of the component of the origin in the soft Boolean model is also valid for scale-free percolation. Contrarily, each upper bound for scale-free percolation is then also an upper bound in the soft Boolean model. Hence, it suffices to prove the upper bound for scale-free percolation in order to finish the proofs of the Theorems~\ref{thm:NumberPoints} and~\ref{thm:SFP}.
 
\begin{proof}[Proof of the upper bound in Theorem~\ref{thm:SFP}.]
 We couple the component of the origin to a multi-type branching process starting at the origin which we now describe. The individuals of the branching process are nodes \(\x=(x,u_x)\) that have a location in \(\R^d\) and a type \(u_x\in(0,1)\). Let \(\mathcal{Y}^{(1)}\) be a unit-intensity Poisson point process on \(\R^d\times(0,1)\times(0,1)\), independent of everything else, whose nodes we denote by \((x,u_x,v_x)\). Given the origin and its type \(\o=(o,u_o)\), the first generation consists of the points	
	\[
		\mathcal{Z}^{(1)} = \big\{(x,u)\colon \exists v \text{ s.t.\ }(x,u,v)\in \mathcal{Y}^{(1)} \text{ and }v\leq 1\wedge (\beta^{-1}u_o^{\gamma}u^{\gamma}|x|^d)^{-\delta}\big\}.
	\]
	Let us denote by \(Z^{(1)}(J)\) the number of points in \(\mathcal{Z}^{(1)}\) that have their type in \(J\subset(0,1)\). Then, given \(\o=(o,u_o)\), \(Z^{(1)}(J)\) is Poisson distributed with mean
	\begin{equation}\label{eq:expectedType}
		\E_{(o,u_o)}Z^{(1)}(J) = \int\limits_J \d u \int\limits_{\R^d} \d x \ 1 \wedge (\beta^{-1}u^\gamma  u_o^\gamma |x|^d)^{-\delta}  = \beta \tfrac{\omega_d \delta}{\delta-1} u_o^{-\gamma} \int\limits_J \d u \ u^{-\gamma}.
	\end{equation}
	For simplicity we write \(Z^{(1)}=Z^{(1)}((0,1))\) in which case we have
	\begin{equation}\label{eq:expectedGen1}
		\begin{aligned}
			\mu^{(1)}((o, u_o)) := \E_{(o,u_o)}Z^{(1)} = \beta \tfrac{\omega_d \delta}{(\delta-1)(1-\gamma)}u_o^{-\gamma}\leq \beta \tfrac{\omega_d \delta}{(\delta-1)(1-2\gamma)}u_o^{-\gamma},
		\end{aligned}
	\end{equation}
	 using \(\gamma<1/2\) in the last step. It is easy to see that the number of individuals in the first generation of this process coincides with the number of neighbours the origin has in \(\cG^{\beta,\gamma,\gamma,\delta}_\o\). To construct the second generation, let the first generation \(\{\x_1^{(1)},\dots,\x_k^{(1)}\}=\mathcal{Z}^{(1)}\) be given. For \(j\in\{1,\dots,k\}\) let \(\mathcal{Y}_{j}^{(2)}\) be a unit-intensity Poisson point process on \(\R^d\times(0,1)\times(0,1)\) independent of everything else. The children of \(\x_j^{(1)}=(x_j^{(1)},u_j^{(1)})\) are then given by
	\[
		\mathcal{Z}_j^{(2)} = \big\{(x,u)\colon \exists v\text{ s.t.\ }(x,u,v)\in{\mathcal{Y}_j^{(2)}}\text{ and } v\leq 1\wedge\big(\beta^{-1}u^{\gamma}(u_j^{(1)})^\gamma|x-x_j^{(1)}|^d\big)^{-\delta}\big\}.
	\]
	Let as above \(Z^{(2)}_j\) (resp.\ \(Z^{(2)}_j(J)\)) denote the number of children of the individual \(\x_j\) (of type in \(J\)), which again is Poisson distributed with mean \(\mu^{(1)}(\x_j)\), see~\eqref{eq:expectedGen1}. Given the whole first generation \(\mathcal{Z}^{(1)}\), the expected size of the second generation is hence given by \(\sum_{\x\in \mathcal{Z}^{(1)}}\mu^{(1)}(\x)\). Using the Markov property of this process together with Mecke's equation, we infer that the expected size of the second generation, given the origin \(\o=(o,u_o)\), is bounded by
	\begin{equation*}
		\begin{aligned}
			\mu^{(2)}((o,u_o)) & = \E_{(o,u_o)}\Big[ \sum_{\x\in \mathcal{Z}^{(1)}}\mu^{(1)}(\x)\Big] = \hspace{-0.1cm}\int\limits_0^1 \hspace{-0.1cm}\d u \hspace{-0.1cm}\int\limits_{\R^d} \d x \big(1 \wedge (\tfrac{1}{\beta}u^\gamma u_o^\gamma |x|^d)^{-\delta}\big) \mu^{(1)}((x,u))\leq \big(\beta \tfrac{\omega_d \delta}{(\delta-1)(1-2\gamma)}\big)^2 u_o^{-\gamma},
		\end{aligned}
	\end{equation*} 
	where we used~\eqref{eq:expectedGen1} and \(\gamma<1/2\). Again, it is easy to see that the branching process can be coupled with \(\cG^{\beta,\gamma,\gamma,\delta}_\o\) such that the second generation contains at least as many individuals as there are vertices at graph distance two away from \(\o\) in \(\cG^{\beta,\gamma,\gamma,\delta}_\o\). We continue constructing the subsequent generations in the obvious way. Then, for each \(n\), the \(n\)-th generation \(\mathcal{Z}^{(n)}\) contains at least as many individuals as there are vertices at graph distance \(n\) away from \(\o\) in \(\cG^\beta_\o\). Moreover, we obtain inductively by the Markov property of the process and Mecke's equation that
	\begin{equation*}
		\begin{aligned}
			\mu^{(n)}((o,u_o)) & := \E_{(o,u_o)} Z^{(n)} \leq \big(\beta \tfrac{\omega_d \delta}{(\delta-1)(1-2\gamma)}\big)^n u_o^{-\gamma}. 
		\end{aligned}
	\end{equation*}
	Choosing \(\beta<\tfrac{(\delta-1)(1-2\gamma)}{\omega_d \delta}\), the sequence \(\mu^{(n)}((o,u_o))\) is summable and hence 
	\[
		\E \scN_{\beta,\gamma,\gamma,\delta}\leq \int\limits_0^1 \d u \sum_n \mu^{(n)}(o,u_o) <\infty.
	\] 
	In particular, the constructed branching process is subcritical. Moreover, we can see in our calculations that the spatial embedding plays no particular role and only the polynomial types of the form \(u^{-\gamma}\) (i.e., the associated volume) influence the offspring distribution. Let us think in the following of the individual types as of the form \(W_x=u_x^{-\gamma}\in(1,\infty)\). By the Poisson nature of the branching process,~\eqref{eq:expectedType} and~\eqref{eq:expectedGen1} together yield that an offspring of a type-\(w\) individual has type in 	\((1,z)\) with probability
	\[
		\frac{\beta \tfrac{\omega_d \delta}{\delta-1} w \int\limits_{z^{-1/\gamma}}^1 \d u \ u^{-\gamma}}{\beta \tfrac{\omega_d \delta}{(\delta-1)(1-\gamma)}w} = 1- z^{1-1/\gamma},
	\]
	independently from the ancestor's type, for all \(z>1\). Hence, the constructed branching process has the same law as a single-type branching process with mixed-Poisson offspring distribution with mean \(\beta C W\) where \(W\) is \(\operatorname{Pareto}(1/\gamma)\) distributed and \(C=\tfrac{\omega_d \delta}{(\delta-1)(1-\gamma)}\). This reduction to a single-type branching process is a crucial feature of the underlying product structure of the kernel \(g_{\gamma,\gamma}\), see also~\cite[Chapter 3.4.3]{vdH2024}. In particular, the mixing parameter \(W'=\beta C W\) is heavy tailed with tail distribution
	\[
		\P(W'\geq w) = c w^{1-1/\gamma}, 
	\]
	for an appropriate \(c>0\). But this also implies that  a random variable \(Z\) that is distributed as mixed Poisson with mean \(W'\) is heavy tailed with tail distribution
	\[
		cz^{1-1/\gamma}\leq \P(Z\geq z) \leq C z^{1-1/\gamma}, 
	\] 
	for some constants \(c,C>0\), see~\cite[Chapter 6]{vdH2017}, which, in particular, implies \(\P(Z=z)\asymp z^{-1/\gamma}\) for all large \(z\in\N\). Summarising, the cardinality of the  component of the origin is stochastically dominated by the total progeny \(\mathcal{C}\) of a branching process with offspring distribution identically to \(Z\). Together with Dwass' Theorem~\cite{Dwass1969}, we obtain
	\begin{equation}\label{eq:Dwass}
		\begin{aligned}
		\P_o(\scN_{\beta,\gamma,\gamma,\delta} \geq m) & \leq \P(\mathcal{C}\geq m) = \sum_{k\geq m} \P(\mathcal{C}=k) = \sum_{k\geq m} \tfrac{1}{k} \P(Z_1+\dots+Z_k=k-1),
		\end{aligned}
	\end{equation}
	where \(Z_1,\dots,Z_k\) are i.i.d.\ copies of \(Z\). Finally, recall that we have already shown the subcriticality of the process which particularly implies \(\E Z_1<1\) (the integrability is due to \(\gamma<1/2\) and the boundedness is due to \(\beta\) being small enough). Hence, \(n-1>n\E Z_1\) for large enough \(n\) and as \(Z_1,Z_2,\dots\) are independent and heavy tailed with finite expectation, the \emph{single big jump paradigm}~\cite[Theorem 9.1]{DenisovEtAl2008} yields 
	\[
		\P(Z_1+\dots+Z_n\geq n-1)\sim n \P(Z_1\geq n-1) \asymp n^{{2-1/\gamma}},
	\]     
	as \(\gamma<1/2\). Again, this implies \(\P(Z_1+\dots+Z_n=n-1)\asymp n^{1-1/\gamma}\) for sufficiently large \(n\). Plugging this back into~\eqref{eq:Dwass} we obtain, for large enough \(m\), that
	\begin{equation*}
		\begin{aligned}
			\P_o(\scN_{\beta,\gamma,\gamma,\delta} \geq m) & \leq \sum_{k\geq m} \tfrac{1}{k}\P(Z_1+\dots+Z_k=k-1) \asymp \sum_{k\geq m} k^{-1/\gamma} \asymp m^{1-1/\gamma}. 
		\end{aligned}
	\end{equation*}
	This concludes the proof. 
	\end{proof}

\section*{Acknowledgement} 
This project was initiated by a research visit of LL to the University of Bath generously supported by the \emph{Federico Tonielli Award 2021} of the University of Cologne. 
BJ and LL further acknowledge the financial support of the Leibniz Association within the Leibniz Junior Research Group on \emph{Probabilistic Methods for Dynamic Communication Networks} as part of the Leibniz Competition. {Finally, we would like to thank the referees for their thorough reading of the manuscript and their suggestions that have considerably improved the discussion of our results.}
	
\section*{References}
\renewcommand*{\bibfont}{\footnotesize}
\printbibliography[heading = none]


\end{document}